\documentclass[10pt,letterpaper]{amsart}

\usepackage{fullpage}

\usepackage[margin=1in]{geometry}
\usepackage{amsmath}
\usepackage{amsfonts}
\usepackage{amssymb}
\usepackage{amsthm,mathtools,mathabx,accents,}
\usepackage{newlfont}
\usepackage{graphicx}
\usepackage{xcolor,enumerate,mathrsfs, scalerel}
\usepackage{esint}
\usepackage{hyperref}

\newtheorem{thm}{Theorem}[section]
\newtheorem{cor}[thm]{Corollary}
\newtheorem{lem}[thm]{Lemma}
\newtheorem{prop}[thm]{Proposition}
\newtheorem{conj}[thm]{Conjecture}

\theoremstyle{definition}
\newtheorem{defn}[thm]{Definition}
\newtheorem{rem}[thm]{Remark}
\numberwithin{equation}{section}

\newcommand{\na}{\nabla}
\newcommand{\pa}{\partial}

\renewcommand{\div}{\operatorname{div}}

\newcommand\al{\alpha}

\newcommand\de{\delta}

\newcommand\Ga{\Gamma}

\newcommand{\la}{\lambda}
\newcommand{\La}{\Lambda}

\newcommand{\T}{\mathbb{T}}
\newcommand{\R}{\mathbb{R}}
\newcommand{\Z}{\mathbb{Z}}
\newcommand{\N}{\mathbb{N}}

\newcommand{\supp}{\operatorname{supp}}

\newcommand{\I}{\operatorname{Id}}

\renewcommand{\dot}[1]{\accentset{\circ}#1}

\DeclarePairedDelimiter{\ceil}{\lceil}{\rceil}

\usepackage[usestackEOL]{stackengine}
\def\dint{\,\ThisStyle{\ensurestackMath{%
  \stackinset{c}{.2\LMpt}{c}{.5\LMpt}{\SavedStyle-}{\SavedStyle\phantom{\int}}}%
  \setbox0=\hbox{$\SavedStyle\int\,$}\kern-\wd0}\int}

\title{An Onsager-type theorem for SQG}

\author{Mimi Dai}
\address{Department of Mathematics, Statistics and Computer Science, University of Illinois at Chicago\\
851 S. Morgan Street, Chicago, IL 60607-7045, USA}
\email{mdai@uic.edu}

\author{Vikram Giri}
\address{Department of Mathematics, ETH Z\"urich\\
R\"amistrasse 101, 8092 Z\"urich, Switzerland}
\email{vikramaditya.giri@math.ethz.ch}

\author{R\u{a}zvan-Octavian Radu}
\address{Department of Mathematics, Princeton University\\
Fine Hall, Princeton, NJ 08544, USA}
\email{rradu@math.princeton.edu}

\keywords{surface quasi-geostrophic equation; Onsager conjecture; non-uniqueness}
\subjclass{35Q35; 35Q86; 76D03}

\begin{document}

\begin{abstract}
    We construct non-trivial weak solutions $\theta\in C_t^0C_x^{0-}$ to the surface quasi-geostrophic (SQG) equations, which have compact support in time and, thus, violate the conservation of the Hamiltonian. The result is sharp in view of the fact that such a conservation law holds for all weak solutions in the class $C_{t,x}^0 \subset L_{t,x}^3$ (\cite{IV}) and resolves the Onsager conjecture for SQG. The construction is achieved by means of a Nash iteration together with the linear decoupling method recently introduced in \cite{GR23}.
\end{abstract}
\maketitle

\bigskip

\section{Introduction}
\label{sec-intro}


Consider the surface quasi-geostrophic (SQG) equations on a periodic spatial domain $\mathbb T^2 = \R^2/(2\pi\Z)^2$:
\begin{equation}\begin{cases}\label{sqg}
\pa_t \theta+u\cdot\nabla \theta= 0,\\
u= \nabla^\perp (-\Delta)^{-\frac{1}{2}}\theta.
\end{cases}
\end{equation}
The system is structurally an active scalar equation: the incompressible transporting velocity field $u:\mathbb T^2 \times \mathbb R \rightarrow \mathbb R^2$ is determined through the Riesz transform $T=\na^\perp \Lambda^{-1}=\na^\perp (-\Delta)^{-\frac12}$ by the transported scalar $\theta:\mathbb T^2 \times \mathbb R \rightarrow \mathbb R$. 

The SQG equations \eqref{sqg} are of interest, on the one hand, because of physical applications in the study of atmospheric and oceanic fluid flows (\cite{Ped}), and, on the other, because of the mathematical similarities with the 3D Euler equations (\cite{CMT}). On the mathematical side, the inviscid system \eqref{sqg} and its dissipative variants have been studied extensively. We refer the reader to \cite{CMT, Res, Cor, CW, CF, KNV, Mar, CaV, KN, CasC, KN2, CoVi, CTV2, CaCoS} and references there-in for information on such developments. 

The equations formally possess a number of conservation laws, among which that of the Hamiltonian. Indeed, let $\theta:\mathbb T^2 \times \mathbb R \rightarrow \mathbb R$ be a smooth solution of \eqref{sqg} and define 
\[\mathcal H(t)=\frac{1}{2}\int_{\mathbb T^2}|\Lambda^{-\frac12}\theta|^2(x, t)\,dx.\]
Then, the simple calculation
\begin{equation*}
    \frac{d}{dt} \mathcal{H} = \int_{\mathbb T^2} \Lambda^{-1} \theta \partial_t \theta dx = - \int_{\mathbb T^2} \Lambda^{-1} \theta \div \left(\theta \nabla^\perp \Lambda^{-1} \theta \right)dx = \int_{\mathbb T^2} \theta \nabla \Lambda^{-1}\theta \cdot \nabla^{\perp} \Lambda^{-1} \theta dx = 0
\end{equation*}
shows that the Hamiltonian $\mathcal{H}$ is conserved in time. In fact, following the classical work of Constantin-E-Titi (\cite{CET}), it was shown in \cite{IV} that this property is satisfied for all weak solutions in the class $\theta \in L_{t,x}^3$. Working on the scale of H\"older spaces, this leads to the following conjecture (first expressed in \cite{BSV}) which is analogous to the one famously raised by Onsager for the Euler equations (\cite{Onsager49}).
\begin{conj}\label{conj-1}
[Onsager conjecture for SQG]
All weak solutions $\theta\in C_{t,x}^0$ of ~\eqref{sqg} conserve the Hamiltonian. However, for any $\frac{1}{2}\leq \gamma<1$, there exist weak solutions of class $\Lambda^{-1}\theta\in C_t^0C^\gamma_x$ that fail to conserve the Hamiltonian.
\end{conj}
The condition $\gamma \geq \frac{1}{2}$ is added so as to ensure that $\Lambda^{-\frac{1}{2}}\theta \in L^2$, which is, of course, necessary to make sense of the Hamiltonian.

As is the case for the Euler equations, the only known approach to the flexible side of conjecture \ref{conj-1} is that of the Nash iteration technique. This method, which was introduced by Nash in the context of the isometric embedding problem (\cite{Nash}), was adapted by De Lellis and Sz\'ekelyhidi (\cite{DLS1, DLS2}) in order to construct flexible solutions to the Euler equations. This triggered a series of works which eventually led to the resolution of the Onsager conjecture by Isett in 3D \cite{Is} (see also \cite{BDLSV}) and, more recently, by the second and third named authors in 2D \cite{GR23}. We refer the reader to the excellent surveys \cite{BV20, DLS3, DLSs1,DLSs2} for more information on the developments, as well as for applications to other equations. 

For the SQG system \eqref{sqg}, the question of constructing non-unique solutions was first raised by De Lellis and Sz\'ekelyhidi in \cite{DLS3}. This was accomplished by Buckmaster, Shkoller and Vicol in \cite{BSV}, who contributed the first partial result towards conjecture \ref{conj-1} by constructing $\mathcal{H}$-non-conservative weak solutions which satisfy $\Lambda^{-1}\theta \in  C_t^0C_x^\gamma$ for any $\frac{1}{2}<\gamma < \frac{4}{5}$. Their approach consists in a Nash iteration scheme at the level of the potential field $v = \Lambda^{-1} u$ (note that $-\nabla^\perp \cdot v = \theta$), where the high-high-to-low interaction is carefully exhibited in Fourier space. Later, solutions in the same regularity class were constructed by Isett and Ma (\cite{IM}), who work at the level of the scalar $\theta$ and employ a bilinear microlocal lemma (extending the linear microlocal lemma of \cite{IV}) to treat the high-high-to-low cascade. These works have influenced all other Nash iterative constructions of solutions to \eqref{sqg}: in \cite{CKL} infinitely many stationary solutions to dissipative SQG are constructed; the works \cite{ BHP, DP, DP-s} show non-uniqueness for a forced SQG system.

As also pointed out in \cite{BSV}, the main obstacle to achieving the sharp regularity of conjecture \ref{conj-1} is the presence of unwanted interaction between high-frequency oscillations corresponding to different directions. This difficulty can be overcome for the 3D Euler equations (\cite{Is}) by defining perturbations based on spatially separated pipe (Mikado) flows. For \eqref{sqg}, as for 2D Euler, these flows are not available, and the decoupling of directions has to be achieved by a different method. The goal of this paper is to provide a proof of the flexible side of conjecture \ref{conj-1}. The decoupling is achieved by adapting the linear iteration introduced in \cite{GR23} for the 2D Euler equations, while in the Nash perturbation step, we use a combination of the ideas of \cite{BSV} and \cite{IM}.

\subsection{Notion of weak solutions and the main result}
\label{sec-result}

As remarked in \cite{BSV}, weak solutions can be defined in spaces of sufficiently low regularity so that conjecture \ref{conj-1} makes sense.
\begin{defn} 
A function $\theta\in L^2_{loc}(\mathbb R; \dot H^{-\frac12}(\mathbb T^2))$ is said to be a weak solution of (\ref{sqg}) if 
\begin{equation}\notag
\int_{\mathbb R}\int_{\mathbb T^2} \Lambda^{-\frac12}\theta\partial_t\Lambda^{\frac12}\psi \, dxdt-\frac12\int_{\mathbb R}\int_{\mathbb T^2}\Lambda^{-\frac12}\theta \Lambda^{\frac12}\left([\nabla^{\perp}\Lambda^{-1}, \nabla\psi]\theta \right)\, dxdt=0, \,\,\, \forall \psi \in C_c^\infty(\mathbb T^2 \times \mathbb R),
\end{equation}
where
  $[\nabla^{\perp}\Lambda^{-1}, \nabla\psi]\theta := \nabla^{\perp}\Lambda^{-1} \cdot (\theta \na\psi) - \na\psi \cdot \nabla^{\perp}\Lambda^{-1}\theta$. 
\end{defn}

The definition implicitly uses the fact that the commutator $[\nabla^{\perp} \Lambda^{-1}, \nabla\psi]$ is bounded from $H^{-\frac12}$ to $H^{\frac12}$. We refer the reader to \cite{BSV} for further discussions regarding different notions of weak solutions and their equivalence. 


The main result of this paper is the following answer to conjecture \ref{conj-1}.
\begin{thm}[Main Theorem]\label{thm.main}
Let $\frac12\leq \gamma<1$. There exists a non-trivial weak solution $\theta$ of \eqref{sqg} satisfying $\Lambda^{-1}\theta \in C^0_tC_x^{\gamma}$ and having compact support in time.
\end{thm}

To achieve the critical Onsager regularity for SQG, the Nash iteration scheme involves two steps, as for the 2D Euler case in \cite{GR23}. In the first step, we construct perturbations which solve the Newtonian linearization of the SQG equations, augmented with temporally oscillatory forcing (following \cite{GR23}, we call this the Newton step). The purpose of this procedure is to decouple the different directions of the Reynolds stress, so that we are left with errors which are, at each time, essentially simple tensors. In the second step, we perform the standard Nash perturbation to reduce the size of the error. In terms of setup, we solve the SQG-Reynolds system iteratively directly at the level of the scalar $\theta$: 
\begin{equation}\begin{cases}\label{ASE-q}
\pa_t \theta_q+u_q\cdot\nabla \theta_q= \ \nabla^{\perp}\cdot \div R_q\\
u_q=\ T[\theta_q], 
\end{cases}
\end{equation}
where $q \in \mathbb N$ and $R_q$ is a symmetric $2$-tensor. We consider the differential operator $\nabla^{\perp}\cdot\div$, different from the double divergence form used in~\cite{IM}. This allows us to erase errors which are built on modulated simple tensors of the form $\xi \otimes \xi$, $\xi \in \mathbb Z^2$, and, therefore, the standard geometric decomposition lemma used for the Euler equations (and, for example, the isometric embedding problem) can be employed (see lemma \ref{le-geo}). This relaxation is closer in spirit to the original approach of \cite{BSV} though we stay at the level of $\theta$ and do not pass to the potential $v = \nabla^\perp \Delta^{-1} \theta$.
Implementing the scheme requires sharp estimates of commutators and bilinear forms associated with the nonlinear structure. This is done in section~\ref{sec-tech} and is the main technical difference between the present paper and the 2D Euler case studied in \cite{GR23}.

In the final stages of writing this manuscript, we have learned that Isett and Looi~\cite{IL} have an independently obtained approach that also resolves the Onsager conjecture for SQG.

\subsection{Further questions}
As the Newton-Nash iteration is able to construct flexible solutions up to the Onsager exponent for both the 2D Euler equations (\cite{GR23}) and the SQG equations, it would be natural to expect that the method can be used to construct such solutions to the generalized SQG equations, which interpolate between the two. The generalized SQG equations are active scalar equations where the velocity field $u$ is related to the density $\theta$ by
\[ u = \na^\perp (-\Delta)^{-(s+1)/2}, \]
where $s=-1$ corresponds to 2D Euler in vorticity form and $s=0$ corresponds to SQG. The main difficulties seem to be related to obtaining sharp Fourier analytic estimates analogous to the ones proved in section \ref{sec-tech}. While, due to the fact that the solutions constructed in the present work are below $C_x^0$, we only have to consider bilinear Fourier multiplier operators, this would not be the case for any other exponent in the gSQG equations and bounds for trilinear Fourier multiplier operators seem to become necessary.  

Perhaps a more fundamental problem is that of obtaining high regularity solutions to \eqref{sqg} for which $\|\theta\|_{L^2}(t)$ is not conserved. The regularity threshold for this conservation law is $C_x^{1/3}$, and, thus, any flexible solution with regularity close to the threshold would necessarily conserve the Hamiltonian. In particular, the solutions cannot have compact support in time -- this seems to be a serious obstruction for all available techniques.

On the other hand, if one considers the case of active scalars for which the structure law is given by a zero-order Fourier multiplier operator with non-odd symbol (as is the case in \cite{IV}), the $C_x^{1/3}$ Onsager threshold is likely within reach of the methods of \cite{GR23} and of the present paper. Importantly, these systems no longer admit Hamiltonian structures.

\subsection{Outline of the paper}
To conclude this section, we provide an outline for the rest of the paper. 

Section \ref{sec-tech} is devoted to proving technical lemmas on the boundedness of bilinear Fourier multiplier operators on H\"older spaces. 

In section \ref{sec-induct}, we state the main iterative proposition \ref{prop.main} and prove the main theorem~\ref{thm.main} assuming the proposition. We end the section with a heuristic analysis which leads to the critical regularity threshold for the non-conservative weak solutions. 

The proof of the main iterative proposition \ref{prop.main} will be completed in two steps -- the Newton step and Nash step -- in sections \ref{sec-Newton} and \ref{sec-Nash}, respectively. Some auxiliary estimates and well-known lemmas are provided in the appendices. 

\subsection*{Acknowledgements}
{The authors wish to thank Camillo De Lellis for useful discussions and, in particular, for an algebraic identity that was used in the proofs of lemmas~\ref{lem.bilin} and~\ref{le-bilinear-odd}. R.Radu is also thankful to Noah Stevenson for discussions related to the Fourier analytic content of this paper. M.Dai is grateful for the support of the NSF grants DMS-2009422 and DMS-2308208.}

\section{Preliminary harmonic analysis}
\label{sec-tech}
We will require a couple of Fourier analysis lemmas that provide H\"older estimates on the bilinear operators that will appear in course of our proof. To avoid tracking non-essential constants, we use the notation $A\lesssim B$. The reader can find the conventions and preliminaries used in this section in appendix~\ref{sec.bha}.

\begin{lem}\label{lem.bilin}
      Let $0<\alpha<1$. Then, the bilinear Fourier multiplier operator 
    \begin{equation}
        T[f,g] = \Lambda^{-1}\big((\Lambda f)g - f(\Lambda g)\big),
    \end{equation}
    applied to smooth functions $f, g: \mathbb
     T^2 \rightarrow \mathbb R$, satisfies the estimate
    \begin{equation}
        \|T[f,g]\|_{N+\alpha} \lesssim \|f\|_{N+\alpha} \|g\|_\alpha + \|f\|_{\alpha} \|g\|_{N+\alpha}, \,\,\, \forall N \geq 0,
    \end{equation}
    with implicit constant depending only on $\alpha$ and $N$.
\end{lem}

\begin{proof}
    It suffices to show the estimate in the case $N=0$. Indeed, 
    with $|\gamma|\leq N$ a multi-index, we have
    \begin{eqnarray*}
        \|\partial^\gamma T[f,g]\|_\alpha &\lesssim& \sum_{\beta \leq \gamma} \|T[\partial^\beta f, \partial^{\gamma-\beta}g]\|_\alpha \\ 
        &\lesssim& \sum_{\beta \leq \gamma} \|f\|_{|\beta| + \alpha} \|g\|_{|\gamma| - |\beta|+\alpha} \\ 
        &\lesssim& \sum_{\beta \leq \gamma} \big(\|f\|_{|\gamma|+\alpha} \|g\|_\alpha\big)^{|\beta|/|\gamma|} \big(\|f\|_\alpha \|g\|_{|\gamma |+\alpha}\big)^{1-|\beta|/|\gamma|} \\ 
        &\lesssim& \|f\|_{|\gamma|+\alpha}\|g\|_\alpha + \|f\|_\alpha\|g\|_{|\gamma|+\alpha}.
    \end{eqnarray*}

    To argue for the case $N=0$, let $j_0 \in \mathbb N$ be fixed and decompose the operator as 
    \begin{equation*}
        T[f,g] = \underbrace{\sum_{j \in \mathbb Z} T[\Delta_j f, S_{j-j_0} g]}_{T_{HL}[f, g]} + \underbrace{\sum_{j \in \mathbb Z} T[\Delta_j f, S_{j+j_0-1}g - S_{j-j_0} g]}_{T_{HH}[f, g]} + \underbrace{\sum_{j \in \mathbb Z} T[S_{j - j_0} f , \Delta_j g] }_{T_{LH}[f, g]}.
    \end{equation*}

    \textit{Estimate for $T_{HL}$ and $T_{LH}$.} It suffices to argue for $T_{HL}$. Note that 
    \begin{equation*}
        \supp \widehat{T[\Delta_j f, S_{j-j_0} g]} \subset B_{2^{j+1}+ 2^{j-j_0+1}} \setminus B_{2^{j-1} - 2^{j-j_0 + 1}},
    \end{equation*}
    which implies that for $j_0 \geq 4$ and $l \in \mathbb Z$, 
    \begin{equation*}
        \Delta_l T_{HL}[f,g] = \sum_{|j-l| \leq 2} \Delta_l T[\Delta_j f, S_{j-j_0}g],
    \end{equation*}
    which implies 
    \begin{eqnarray*}
        \|\Delta_l T_{HL}[f,g]\|_0 &\lesssim& \sum_{|j-l|\leq 2} 2^{-l} \big(\|\Lambda \Delta_j f\|_0 \|S_{j-j_0}g\|_0 + \|\Delta_j f\|_0 \|\Lambda S_{j-j_0}g\|_0 \big) \\ 
        &\lesssim& 2^{-l} \sum_{|j-l|\leq 2}\big(2^{j(1-\alpha)} \|f\|_\alpha \|g\|_0 + 2^{-j\alpha}\|f\|_\alpha \sum_{m=0}^{j-j_0} 2^m \|g\|_0 \big) \\ 
        &\lesssim& 2^{-l\alpha}\|f\|_\alpha\|g\|_0,
    \end{eqnarray*}
    and the H\"older estimate follows.
    
    \textit{Estimate for $T_{HH}$.} From the definition, we have for $k \in \mathbb Z^2 \setminus \{0\}$ that
    \begin{equation*}
        \widehat{T[f,g]}(k) = \sum_{j \in \mathbb Z^2} \frac{|j| - |k-j|}{|k|} \hat f(j) \hat g(k-j),
    \end{equation*}
    and, thus, 
    \begin{eqnarray*}
        T[f,g] &=& \sum_{j + k \neq 0} \frac{|j| - |k|}{|j+k|} \hat f (j) \hat g(k) e^{i(j+k)\cdot x} \\ 
        &=& \sum_{j+ k \neq 0} \frac{j+k}{|j+k|}\cdot \frac{j-k}{|j|+|k|} \hat f (j) \hat g(k) e^{i(j+k)\cdot x} \\ 
        &=& \mathcal{R} \cdot \sum_{(j,k) \in \mathbb Z^4 \setminus \{0\}} i \frac{j-k}{|j| + |k|} \hat f(j) \hat g(k) e^{i(j+k)\cdot x} \\ 
        &=:& \mathcal{R} \cdot P[f,g].
    \end{eqnarray*}
    Since
    \begin{equation*}
        \supp \widehat{T[\Delta_j f, S_{j+j_0-1}g - S_{j-j_0}g]} \subset B_{2^{j+1} + 2^{j+j_0}},
    \end{equation*}
    we have that, for sufficiently large $j_0$,
    \begin{equation*}
        \Delta_l T_{HH}[f,g] = \sum_{j \geq l - j_0 -1} \mathcal{R}\cdot \Delta_l P[\Delta_j f, S_{j+j_0-1}g - S_{j-j_0}g],
    \end{equation*}
    and, therefore, 
    \begin{equation*}
        \|\Delta_l T_{HH}[f,g]\|_0 \lesssim \sum_{j \geq l-j_0-1} \|P[\Delta_j f, S_{j+j_0-1}g - S_{j-j_0}g]\|_0.
    \end{equation*}
    Let $\bar \chi_j$ be as in the proof of lemma \ref{Bernstein}. Then, 
    \begin{equation*}
        P[\Delta_j f, S_{j+j_0-1}g - S_{j-j_0}g](x) = \sum_{m = j-j_0+1}^{j+j_0-1} \sum_{(l,k) \in \Z^4\setminus\{0\}} \bar \chi_j(l) \bar \chi_m(k) M(l,k)\widehat{\Delta_j f}(l) \widehat{\Delta_m g}(k) e^{i(l+k)\cdot x},
    \end{equation*}
    where 
    \begin{equation*}
        M(\eta, \xi) = i \frac{\eta - \xi}{|\eta| + |\xi|}
    \end{equation*}
    is the multiplier in the definition of $P$. It follows that 
    \begin{equation*}
        P[\Delta_j f, S_{j+j_0-1}g - S_{j-j_0}g](x) = \sum_{m=j-j_0+1}^{j+j_0-1} \int_{\mathbb R^2 \times \mathbb R^2} K_{j,m}(x-y_1, x-y_2) \Delta_j f(y_1) \Delta_m g(y_2) dy_1 dy_2,
    \end{equation*}
    where $\Delta_j f$ and $\Delta_m g$ are identified with their periodic extensions and
    \begin{eqnarray*}
        K_{j,m}(x,y) &=& \frac{1}{(2\pi)^4} \int_{\mathbb R^2 \times \mathbb R^2}M(\eta, \xi) \bar \chi_j(\eta) \bar \chi_m(\xi) e^{i(\eta \cdot x + \xi \cdot y)} d\eta d\xi \\ 
        &=&2^{4j} \frac{1}{(2\pi)^4} \int_{\mathbb R^2 \times \mathbb R^2} M(\eta, \xi) \bar \chi_0(\eta) \bar \chi_0 (2^{j-m} \xi) e^{i2^j(\eta \cdot x + \xi \cdot y)} d\eta d\xi \\ 
        &=:& 2^{4j} K_{0, m-j} (2^jx, 2^jy). 
    \end{eqnarray*}
    We note that $M(\eta, \xi)\bar\chi_0(\eta)\bar\chi_0(2^{j-m}\xi)$ are smooth and compactly supported, since they vanish in neighbourhoods of the planes $\{|\eta|=0\}$ and $\{|\xi|=0\}$. Thus, the kernels $K_{0, j-m}$ are in $L^1(\mathbb R^2 \times \mathbb R^2)$, and it follows that 
    \begin{equation*}
        \|P[\Delta_j, S_{j+j_0-1}g - S_{j-j_0}g]\|_0 \lesssim \sum_{m=j-j_0+1}^{j+j_0-1} \|K_{0, j-m}\|_{L^1(\mathbb R^2 \times \mathbb R^2)} \|\Delta_j f\|_0\|\Delta_m g\|_0 \lesssim 2^{-2j\alpha} \|f\|_\alpha\|g\|_\alpha.
    \end{equation*}
    Consequently, 
    \begin{equation*}
        \|\Delta_l T_{HH}[f,g]\|_0 \lesssim \sum_{j \geq l - j_0 + 1} 2^{-2j\alpha} \|f\|_\alpha \|g\|_\alpha \lesssim 2^{-l\alpha}\|f\|_\alpha \|g\|_\alpha,
    \end{equation*}
    and the conclusion follows.
\end{proof}

Let $T$ be the operator in (\ref{sqg}), i.e. $T=\nabla^{\perp}\Lambda^{-1}$. Define the bilinear form 
\[S[f,g]=\Delta^{-1}\div \left(T[\Delta f]g+T[g]\Delta f\right).\]
As a consequence of Lemma \ref{lem.bilin}, we can see that
\begin{lem}\label{le-bilinear-S}
Let $0<\alpha<1$ and $N\geq 0$. For smooth and mean-zero functions $f,g: \mathbb T^2\to \mathbb R$, the estimate 
\[\|S[f,g]\|_{N+\alpha}\lesssim _{\alpha,N} \|f\|_{N+1+\alpha}\|g\|_\alpha+\|f\|_{1+\alpha}\|g\|_{N+\alpha}\]
holds.
\end{lem}
\begin{proof}
    We expand the Fourier transform of $S$ as
    \begin{align*}
        \widehat{S[f,g]}(\xi) &= \frac{i \xi}{|\xi|^2} \cdot \sum_{\eta\in \Z^2} |\xi-\eta|^2 \left( \frac{i (\xi-\eta)^\perp}{|\xi-\eta|} + \frac{i \eta^\perp}{|\eta|}\right) \hat f(\xi-\eta) \hat g(\eta) \\
        &= -\frac1{|\xi|^2} \sum_{\eta\in \Z^2} |\xi-\eta|^2 \xi\cdot\eta^\perp \left( -\frac{1}{|\xi-\eta|} + \frac{1}{|\eta|}\right) \hat f(\xi-\eta) \hat g(\eta)  \\
        &= -\frac{\xi}{|\xi|^2}\cdot \sum_{\eta\in \Z^2} |\xi-\eta| \frac{\eta^\perp}{|\eta|} \left({|\xi-\eta|-|\eta|}\right) \hat f(\xi-\eta) \hat g(\eta)  \\
        &= \frac{i \xi}{|\xi|^2}\cdot \sum_{\eta\in \Z^2}  \left({|\xi-\eta|-|\eta|}\right) \widehat{\Lambda f}(\xi-\eta) \widehat{\mathcal{R}_i^\perp g}(\eta)  \\
        &= \widehat{\mathcal{R}_i T[\Lambda f, \mathcal{R}_i^\perp g]}(\xi)
    \end{align*}
to conclude that
\begin{equation}\label{eqn.S}
    S[f,g] = \mathcal{R}_i T[\Lambda f, \mathcal{R}_i^\perp g]
\end{equation}
where $\mathcal{R}_i$ is the Riesz transform, corresponding to the Fourier multiplier $i \xi_i / |\xi|$, and $T$ is the bilinear Fourier multiplier operator analyzed in the previous lemma~\ref{lem.bilin}.

Now we can use lemma~\ref{lem.bilin} to estimate
\begin{align*}
    \|S[f,g]\|_{N+\al} &= \|\mathcal{R}_i T[\Lambda f, \mathcal{R}_i^\perp g]\|_{N+\al}\\
    &\lesssim \|T[\Lambda f, \mathcal{R}_i^\perp g]\|_{N+\al}\\
    &\lesssim \|\Lambda f\|_{N+\al}\|\mathcal{R}_i^\perp g\|_\al + \|\Lambda f\|_{\al}\|\mathcal{R}_i^\perp g\|_{N+\al}\\
    &\lesssim \|f\|_{N+1+\al}\|g\|_\al + \|f\|_{1+\al}\|g\|_{N+\al},
\end{align*}
as desired.
\end{proof}

\section{The main inductive proposition}
\label{sec-induct} 

Let $\lambda_q$ be a frequency parameter, defined as
\begin{equation}\label{def.laq}
    \la_q :=  \ceil{a^{b^q}},
\end{equation}
where $a$ is a large real number and $b$ is such that $0<b-1\ll 1$. Define an amplitude parameter $\de_q$, which is defined as
\begin{equation}\label{def.deq}
    \de_q := \la_q^{-2\beta}
\end{equation}
where $\beta$ is the coefficient which will determine the regularity of the constructed solution.

Let $L_\theta, L_R, L_t \in \mathbb{N}_{\neq 0}$, $M>0$, and $0<\alpha\ll 1$. We assume the following inductive estimates: 
\begin{align} 
    \|\theta_q\|_N + \|u_q\|_N \leq M \de_q^{\frac12} \la_q^{N+\frac{1}{2}}, &\qquad \forall N \in \{0, 1, ... ,L_\theta\},\label{induct-u}\\
    \|R_q\|_N \leq \de_{q+1} \la_q^{N-2\alpha}, &\qquad \forall N \in \{0,1,...,L_R\}, \label{induct-R}\\
    \|D_t R_q\|_N \leq  \de_{q+1} \de_q^{\frac12} \la_q^{N-2\alpha+\frac{3}{2}}, &\qquad \forall N \in \{0,1,..., L_t\}. \label{induct-DR}
\end{align}
We further assume that the temporal support of $R_q$ satisfies
\begin{equation}\label{induct-sup}
    \supp_t R_q \subset  [-2 +(\de_q^{\frac12}\la_q^{\frac{3}2})^{-1}, -1 -(\de_q^{\frac12}\la_q^{\frac{3}2})^{-1} ]  
     \cup [1 + (\de_q^{\frac12}\la_q^{\frac{3}2})^{-1} , 2 - (\de_q^{\frac12}\la_q^{\frac{3}2})^{-1} ].  
\end{equation}

We can now state the main inductive proposition.
\begin{prop}[Main inductive proposition]\label{prop.main}
    Let $L_\theta=30$, $L_R=20$, $L_t=10$, $0<\beta<1/2$, $1<b<\frac{1+2\beta}{4\beta}$. There exist $M_0 > 0$ depending only on $\beta$ and $L_\theta, L_R, L_t$, and a coefficient $0<\alpha_0<1$ depending on $\beta$ and $b$, such that for any $M > M_0$ and $0 < \alpha < \alpha_0$, there exists $a_0 > 1$ depending on $\beta$, $b$, $\alpha$, $M$ and $L_\theta$, $L_R$, $L_t$, such that for any $a > a_0$ the following holds: Given a smooth solution $(\theta_q, u_q, R_q)$ of \eqref{ASE-q} and the inductive assumptions \eqref{induct-u}- \eqref{induct-sup}, there exists another smooth solution $(\theta_{q+1},u_{q+1},R_{q+1})$ again satisfying \eqref{induct-u}- \eqref{induct-sup} with $q$ replaced by $q+1$. Moreover, it holds that
    \begin{equation}\label{eq.prop.main}
        \la_{q+1} \|\La^{-1} (\theta_{q+1}-\theta_q)\|_0 + \|\theta_{q+1}-\theta_q\|_0 \leq 2M\delta_{q+1}^\frac12 \la_{q+1}^\frac12\,
    \end{equation}
    and
    \begin{equation}\label{eq.prop.main-2}
        \supp_t \theta_{q+1} \cup \supp_t u_{q+1} \subset (-2,-1)\cup(1,2)\,.
    \end{equation}
\end{prop}

\begin{rem}
    In contrast with \cite{GR23}, we do not control the number of propagated derivatives by only one parameter $L$, but by three such parameters $L_\theta$, $L_R$, $L_t$. There is great freedom in the choice of values for these parameters. For instance, any choice of the form $1 \ll L_t \ll L_R \ll L_\theta$ will work.
\end{rem}

\subsection{Proof of the main theorem}
Here we give a proof of theorem~\ref{thm.main} assuming proposition~\ref{prop.main}. The proof is essentially the same as the analogous ones in most other Nash iterative constructions of fluid flows (see, for instance,~\cite{GR23}). The rest of the paper will then be devoted to proving proposition~\ref{prop.main}.
\begin{proof}

Let $\beta < 1/2$ such that $\gamma < \beta +1/2$, where $\gamma$ is the H\"older coefficient in the statement of the theorem. Fix $b$ so that it satisfies 
\begin{equation*}
    1 < b <  \frac{1+2\beta}{4\beta} \,,
\end{equation*}
and let $M_0$ and $\alpha_0$ be the constants given by proposition \ref{prop.main}. We fix also $M > \max\{M_0, 2\}$ and $\alpha < \min \{\alpha_0, 1/2\}$. Then, let $a_0$ be given by proposition \ref{prop.main} in terms of these fixed parameters. We do not fix $a > a_0$ until the end of the proof.

We now aim to construct the base case for the inductive proposition \ref{prop.main}. Let $f : \R \to [0,1]$ be a smooth function supported in $[-7/4, 7/4]$, such that $f = 1$ on $[-5/4, 5/4]$. Consider
\begin{align*}
    \theta_0 (x,t) = f(t) \delta_0^{1/2} \lambda_0^{1/2}\cos ( \la_0 x_1)\,, \,\, u_0(x,t) = f(t) \delta_0^{1/2} \lambda_0^{1/2}\sin ( \la_0 x_1) e_2,\\
    R_0(x,t) = -f'(t) \frac{\delta_0^{1/2}}{\la_0^{3/2}}  \begin{pmatrix}
        0 & \cos ( \la_0 x_1)\\ \cos(\la_0 x_1) & 0
    \end{pmatrix},
\end{align*}
where $(x_1, x_2)$ denote the standard coordinates on $\T^2$ and $(e_1,e_2)$ are the associated unit vectors. Note that $R_0$ is symmetric and traceless.
It can be checked directly that the tuple $(\theta_0,u_0,R_0)$ solves the relaxed SQG system~\eqref{ASE-q}. 

We have for any $N \geq 0$, 
\begin{equation*}
    \|\theta_0\|_N + \|u_0\|_N \leq M\delta_0^{1/2} \lambda_0^{N+1/2},
\end{equation*}
and, so, \eqref{induct-u} also holds. Moreover, for any $N \geq 0$,
\begin{equation*}
    \|R_0\|_N \leq 2 \sup_t |f'(t)| \frac{\delta_0^{1/2}}{\lambda_0^{3/2}} \lambda_0^N.
\end{equation*}
Since it holds that $\beta(2b-1) < 1/2$, we can ensure that 
\begin{equation*}
    2 \sup_t |f'(t)| < \delta_1 \delta_0^{-1/2} \lambda_0^{1/2},
\end{equation*}
by choosing $a$ sufficiently large. Then, 
\begin{equation*}
    \|R_0\|_N \leq \delta_1 \lambda_0^{-1} \lambda_0^N,
\end{equation*}
and it follows that \eqref{induct-R} holds, since we have chosen $\alpha < 1/2$. For the estimate concerning the material derivative, we calculate
\begin{equation*}
    \partial_t R_0 + u_0 \cdot \nabla R_0 = -f''(t) \frac{\delta_0^{1/2}}{\lambda_0^{3/2}} \begin{pmatrix}
        0 & \cos ( \la_0 x_1)\\ \cos(\la_0 x_1) & 0
    \end{pmatrix}.
\end{equation*}
In order to ensure that \eqref{induct-DR} is satisfied, it suffices to choose $a$ large enough so that 
\begin{equation*}
    2\sup_t |f''(t)| < \delta_1 \delta_0^{-1/2} \lambda_0^{1/2} (\delta_0^{1/2} \lambda_0^{3/2}) = \delta_1 \lambda_0^{2}.
\end{equation*}
Finally, we note that $\supp_t R_0 \subset [-7/4, 7/4] \setminus (-5/4, 5/4)$, and, thus, the condition \eqref{induct-sup} is satisfied provided 
\begin{equation*}
    (\delta_0^{1/2} \lambda_0^{3/2})^{-1} < \frac{1}{4},
\end{equation*}
which, once again, can be guaranteed by the choice of $a$.

We now finally fix $a$ so that all of the wanted inequalities are satisfied, and conclude that the triple $(u_0, p_0, R_0)$ satisfies all the requirements to be the base case for the inductive proposition \ref{prop.main}. Let, $\{(\theta_q, u_q, R_q)\}$ be the sequence of solutions to the  system \eqref{ASE-q} given by the iterative application of the proposition. Equation \eqref{eq.prop.main} implies that 
\begin{equation*}
    \|\La^{-1}(\theta_{q+1} - \theta_q)\|_\gamma \lesssim \|\La^{-1}(\theta_{q+1} - \theta_q)\|_0^{1-\gamma} \|\theta_{q+1} - \theta_q\|_0^{\gamma} \lesssim \delta_{q+1}^{1/2} \lambda_{q+1}^{\gamma-1/2} \lesssim \lambda_{q+1}^{\gamma - \beta-1/2}.
\end{equation*}
Since $\gamma<\beta+1/2$, $\{\Lambda^{-1}\theta_q\}$ is a Cauchy sequence in $C_tC^{\gamma}_x$ and, thus, it converges in this space to a scalar $\Lambda^{-1}\theta$. Moreover, 
$R_q$ converges to zero in $C_{t,x}^0$. To verify that $\theta$ is a weak solution, it suffices to verify convergence of $\{\Lambda^{-\frac{1}{2}} \theta_q\}$ in $C_{t,x}^0$, which is enforced by the condition $\gamma \geq 1/2$. In view of \eqref{eq.prop.main-2}, the constructed solution moreover satisfies $\supp_t \theta \subset [-2, 2]$ and 
\begin{equation*}
    \theta(x,t) = \delta_0^{1/2} \lambda_0^{1/2}\cos(\lambda_0 x_1),
\end{equation*}
whenever $t \in [-1,1]$. 
\end{proof}

\subsection{Heuristic outline of the Newton-Nash scheme}
\label{sec-heuristics}

We now present the main ideas of the proof of proposition \ref{prop.main} at the level of heuristics. Before we begin, however, let us caution the reader that the values given below for the various parameters ($\tau_q$, $\mu_{q+1}$, $\Gamma$, etc.), as well as the definitions of the perturbations and the generated errors will not exactly match those which we will use in the proof. The reasons for these discrepancies are essentially of technical nature.

\subsubsection{Temporal localization and stress decomposition}
Assume that $R_q$ has temporal support over a time interval of length $\tau_q=\left(\de_q^{\frac{1}{2}} \la_q^{\frac{3}{2}}\right)^{-1}$ centered at some time $t_0$. We point out that such a localization is consistent with the inductively assumed estimates on the material derivative of $R_q$. Let $\Phi$ be the backward flow of $u_q$ with origin at $t_0$,
\begin{equation}
\begin{cases}
\partial_t\Phi+u_q\cdot\nabla \Phi=0,\\
\Phi_{t=t_0}=x.
\end{cases}
\end{equation}
Denote by $X$ the Lagrangian flow of $u_q$ starting at $t=t_0$.
Applying the geometric lemma \ref{le-geo}, we can find a finite set of directions $F\subset\mathbb Z^2$ and corresponding smooth functions $\gamma_\xi$ ($\xi \in F$) such that the amplitude functions $a_{\xi}$ defined as
\[a_\xi=2\lambda_{q+1}^{\frac{1}{2}}\delta_{q+1}^{\frac12}|(\nabla \Phi)^{T}\xi|^{\frac{3}2}\gamma_\xi\Big((\nabla \Phi)^{-T} (\nabla \Phi)^{-1} - (\nabla \Phi)^{-T} \frac{R_{q}}{\delta_{q+1}} (\nabla \Phi)^{-1} \Big)\]
satisfy
\[\nabla^{\perp}\cdot\div\sum_{\xi\in F}\underbrace{\frac14\frac{1}{\lambda_{q+1}|(\nabla\Phi)^{T}\xi|^{3}}a_{\xi}^2(\nabla\Phi)^{T}\xi\otimes\xi(\nabla\Phi)}_{A_\xi}=-\nabla^{\perp}\cdot\div R_q.\]
In this way, we achieve a decomposition of $R_q$ into simple tensors $A_\xi$.


\subsubsection{The Newton steps}
Let $\{g_{\xi}\}_{\xi\in F}$ be a set of 1-periodic functions of time with unit norm in $L^2(0,1)$ satisfying
\[\supp_tg_{\xi}\cap \supp_t g_{\xi'}=0, \ \ \mbox{for} \ \xi\neq\xi'.\]
These profiles will be utilized to divert temporal supports for the Nash perturbations corresponding to different directions in the set $F$. We further define
\begin{equation}\notag
\begin{split}
f_{\xi}&=1-g_{\xi}^2,\\
f_{\xi}^{[1]}&=\int_0^tf_{\xi}(s)\, ds.
\end{split}
\end{equation}
Let $\mu_{q+1}\gg \tau_q^{-1}$ be a temporal frequency parameter to be fixed later. Define the first Newton perturbation $\theta_{q+1,1}^{(t)}$ as the solution to the linearization of the SQG equations around $u_q$ starting from the initial time $t_0$:
\begin{equation}\label{newton-1}
\begin{cases}
\partial_t \theta_{q+1,1}^{(t)}+u_q\cdot\nabla \theta_{q+1,1}^{(t)}+T[\theta_{q+1,1}^{(t)}]\cdot\nabla \theta_q=\sum_{\xi\in F}f_{\xi}(\mu_{q+1}t)\nabla^{\perp}\cdot\div A_{\xi},\\
\theta_{q+1,1}^{(t)}|_{t=t_0}=\frac{1}{\mu_{q+1}}\sum_{\xi\in  F}f_{\xi}^{[1]}(\mu_{q+1}t_0)\nabla^{\perp}\cdot\div A_{\xi}|_{t=t_0}.
\end{cases}
\end{equation}
Note 
\[
\begin{split}
\sum_{\xi\in F}f_{\xi}(\mu_{q+1}t)\nabla^{\perp}\cdot\div A_{\xi}&=\sum_{\xi\in F}\nabla^{\perp}\cdot\div A_{\xi}-\sum_{\xi\in F}g_{\xi}^2(\mu_{q+1}t)\nabla^{\perp}\cdot\div A_{\xi}\\
&=-\nabla^{\perp}\cdot\div R_q-\sum_{\xi\in F}g_{\xi}^2(\mu_{q+1}t)\nabla^{\perp}\cdot\div A_{\xi}.
\end{split}
\]
Thus it follows
\begin{equation}\label{newton-1-divert}
\partial_t \theta_{q+1,1}^{(t)}+u_q\cdot\nabla \theta_{q+1,1}^{(t)}+T[\theta_{q+1,1}^{(t)}]\cdot\nabla \theta_q+\nabla^{\perp}\cdot\div R_q=-\sum_{\xi\in F}g_{\xi}^2(\mu_{q+1}t)\nabla^{\perp}\cdot\div A_{\xi}
\end{equation}
Treating the first equation of (\ref{newton-1}) as a transport equation with lower-order perturbation $T[\theta_{q+1,1}^{(t)}]\cdot\nabla \theta_q$, we expect that
\begin{equation}\notag
\begin{split}
\theta_{q+1,1}^{(t)}(X,t)&\approx\frac{1}{\mu_{q+1}}\sum_{\xi\in F}f_{\xi}^{[1]}(\mu_{q+1}t_0)\nabla^{\perp}\cdot\div A_{\xi}|_{t=t_0}+\int_{t_0}^t \sum_{\xi\in F} f_{\xi}(\mu_{q+1}s)\nabla^{\perp}\cdot\div A_{\xi}(X(\cdot,s),s)\,ds\\
&=\frac{1}{\mu_{q+1}}\sum_{\xi\in F}f_{\xi}^{[1]}(\mu_{q+1}t)\nabla^{\perp}\cdot\div A_{\xi}(X,t)-\int_{t_0}^t \sum_{\xi\in F} f_{\xi}^{[1]}(\mu_{q+1}s)\frac{D_t\nabla^{\perp}\cdot\div A_{\xi}}{\mu_{q+1}}(X(\cdot,s),s)\,ds.
\end{split}
\end{equation}
Thanks to $\tau_q^{-1}\ll \mu_{q+1}$ we argue that the second term on the right hand in the above equation is negligible. Then we have 
\begin{equation}\label{w-t-size}
\theta_{q+1,1}^{(t)}\approx \frac{1}{\mu_{q+1}}\sum_{\xi\in F}f_{\xi}^{[1]}\nabla^{\perp}\cdot\div A_{\xi}
\end{equation}
and can infer the estimate
\begin{equation}\notag
\|\theta_{q+1,1}^{(t)}\|_0\lesssim \frac{\lambda_q^2 \delta_{q+1}}{\mu_{q+1}}.
\end{equation}
Since the SQG structure law is zero-order, the same estimate is expected for the induced velocity field $w_{q+1,1}^{(t)}=\nabla^{\perp}\Lambda^{-1}\theta_{q+1,1}^{(t)}$.
The fundamental error generated in this step is the Newton error which appears due to the nonlinearity of the SQG equations:
\[R_{q+1}^{\text{Newton}}=\div^{-1}\Delta^{-1}\nabla^{\perp}\cdot\left(w_{q+1,1}^{(t)}\cdot\nabla \theta_{q+1,1}^{(t)} \right)=\div^{-1}\Delta^{-1}\nabla^{\perp}\cdot \div \left(w_{q+1,1}^{(t)} \theta_{q+1,1}^{(t)}\, \right).\]
Here and throughout the paper we use the anti-divergence operator described in appendix \ref{sec-geo}. Since the operator above is $(-1)$-order, we hope to gain a factor of the spatial frequency $\la_q^{-1}$ and, thus, have the estimate
\[\|R_{q+1}^{\text{Newton}}\|_0\lesssim \lambda_{q}^{-1}\frac{\lambda_q^{4} \delta_{q+1}^2}{\mu_{q+1}^2}. \]
Making this precise will require passing to a double potential formulation of the equations, which we analyze using the Fourier analytic estimates of section~\ref{sec-tech}.

We observe that, since we are essentially solving a transport equation, $\theta_{q+1,1}^{(t)}$ does not have the precise form \eqref{w-t-size} and, thus, does not satisfy the desired estimates globally in time. In order to restrict the perturbation to a time scale which we can control, we glue together temporally localized perturbations. Let $\widetilde \chi$ be a standard smooth cut-off function satisfying $\widetilde \chi=1$ on $\cup_{\xi} \supp A_{\xi}$ and $|\partial_t\widetilde \chi| \lesssim \tau_q^{-1}$. Defining the localized perturbation by $\widetilde \chi \theta_{q+1,1}^{(t)}$ results in a gluing error
\[R_q^{\text{glue}}=\div^{-1}\Delta^{-1}\nabla^{\perp}\partial_t\widetilde \chi\theta_{q+1,1}^{(t)}.\]
In view of (\ref{w-t-size}) we have
\[\div^{-1}\Delta^{-1}\nabla^{\perp}\theta_{q+1,1}^{(t)}\approx \frac{1}{\mu_{q+1}}\sum_{\xi\in F}f_{\xi}^{[1]}\div^{-1}\Delta^{-1}\nabla^{\perp}\cdot\nabla^{\perp}\cdot\div A_{\xi}
=\frac{1}{\mu_{q+1}}\sum_{\xi\in F}f_{\xi}^{[1]} A_{\xi}.\]
Hence we expect the estimate
\[\|R_q^{\text{glue}}\|_0\lesssim \frac{1}{\mu_{q+1}}|\partial_t\widetilde \chi|\|A_{\xi}\|_0\lesssim \frac{\tau_q^{-1}\delta_{q+1}}{\mu_{q+1}}.\]
Using~\eqref{newton-1} and the fact that the forcing term vanishes on the support of $\partial_t \tilde \chi$ to estimate the material derivative of $\theta_{q+1,1}^{(t)}$, we expect the material derivative cost of the gluing error to be $\tau_q^{-1}$ and, thus, 
\[\|D_tR_q^{\text{glue}}\|_0\lesssim  \frac{\tau_q^{-2}\delta_{q+1}}{\mu_{q+1}}.\]
We note that, compared to $R_q$, $R_q^{\text{glue}}$ has improved estimates by a factor of $\tau_q^{-1}/\mu_{q+1}$. However the improvement is not good enough to place $R_q^{\text{glue}}$ into the new error $R_{q+1}$ which is $\lesssim \delta_{q+2}$. For this reason, it is necessary to repeat the procedure inductively until the remaining gluing error is smaller than $\delta_{q+2}$. The number $\Gamma$ of iterations will only depend on $\beta$ and is thus fixed independent of $q$.

After $\Gamma$ iteration of this linear procedure, we obtain a new solution 
\[\theta_{q,\Gamma}=\theta_q+\theta_{q+1}^{(t)}=\theta_q+\sum_{n=1}^\Gamma\theta_{q+1,n}^{(t)}\]
to the SQG-Reynolds system, for which the remaining errors are: (1) the remaining temporally decoupled errors corresponding to the chosen basis of simple tensors:
\begin{equation*}
    R_q^{\text{rem}} = - \sum_{\xi \in F} g_\xi^2 \nabla^\perp \cdot \div A_{\xi};
\end{equation*}
and (2) the non-linear Newton error highlighted above.

\subsubsection{The Nash step}
Let $u_{q,\Gamma} = \nabla^\perp \Lambda^{-1}(\theta_q + \theta_{q+1}^{(t)})$ be the velocity field of the newly obtained solution and $\widetilde \Phi$ be its backward flow starting at time $t=t_0$. To avoid significant interactions between $\theta_{q+1}^{(t)}$ and $\theta_{q+1}^{(p)}$, we define the Nash perturbation in terms of $\widetilde \Phi$ instead of $\Phi$. Analogously to the amplitude functions $a_\xi$ defined previously, we let
\[\bar a_\xi=2\lambda_{q+1}^{\frac{1}{2}}\delta_{q+1}^{\frac12}|(\nabla \widetilde \Phi)^{T}\xi|^{\frac{3}2}\gamma_\xi\Big((\nabla \widetilde \Phi)^{-T} (\nabla \widetilde \Phi)^{-1} - (\nabla \widetilde \Phi)^{-T} \frac{R_{q}}{\delta_{q+1}} (\nabla \widetilde \Phi)^{-1} \Big),\]
which determines the Nash perturbation via
\[\theta_{q+1}^{(p)}= P_{\approx \lambda_{q+1 }}\sum_{\xi\in F}g_\xi \bar a_\xi \cos\left(\lambda_{q+1}\widetilde \Phi\cdot \xi\right),\]
where $P_{\approx \lambda_{q+1}}$ localizes in Fourier space to an annulus of radii $\approx \lambda_{q+1}$. The bilinear microlocal lemma~\ref{le-bilinear-odd} has the purpose of writing the quadratic self-interaction in the form 
$u_{q+1}^{(p)}\cdot\nabla \theta_{q+1}^{(p)} = \nabla^\perp \cdot \div B$,
and to identify the first order behaviour of the symmetric $2$-tensor B. Since by construction the profiles $\{g_{\xi}\}$ have disjoint temporal supports, we have 
\begin{equation*}
    u_{q+1}^{(p)}\cdot \nabla \theta_{q+1}^{(p)} = \sum_{\xi \in F} g_\xi^2  \nabla^\perp \Lambda^{-1} P_{\approx \lambda_{q+1}}\big(\bar a_\xi \cos(\lambda_{q+1} \widetilde \Phi \cdot \xi)\big) \cdot \nabla P_{\approx \lambda_{q+1}}\big(\bar a_\xi \cos(\lambda_{q+1} \widetilde \Phi \cdot \xi)\big)).
\end{equation*}
By a series of manipulations in Fourier space, it can be shown that 
\begin{eqnarray*}
     &&u_{q+1}^{(p)}\cdot \nabla \theta_{q+1}^{(p)}=\\
     && \,\,\, \sum_{\xi \in F} g_\xi^2 \nabla^\perp \cdot \div \sum_{\eta, \zeta \in \{-\xi, \xi\}} \int_{\mathbb R^2 \times \mathbb R^2}K_{\lambda_{q+1}} (h_1, h_2) \bar a_\xi (x-h_1) \bar a_\xi(x-h_2) e^{i\lambda_{q+1} \widetilde \Phi(x-h_1) \cdot \eta} e^{i\lambda_{q+1} \widetilde \Phi(x-h_2) \cdot \zeta} dh_1 dh_2,
\end{eqnarray*}
where the kernel $K_{\la_{q+1}}$, which takes values in the space of $2$-tensors, has the Fourier transform
\begin{equation*}
    \widehat K_{\lambda_{q+1}}(\nu_1, \nu_2) = \frac{1}{16} \frac{(\nu_1 - \nu_2)\otimes(\nu_1 - \nu_2)}{|\nu_1||\nu_2|(|\nu_1| + |\nu_2|)} \varphi(\lambda_{q+1}^{-1} \nu_1)\varphi(\lambda_{q+1}^{-1} \nu_2),
\end{equation*}
where $\varphi(\lambda_{q+1}^{-1}\cdot)$ is the multiplier of $P_{\approx \lambda_{q+1}}$. Since $K_{\lambda_{q+1}}$ is localized at frequencies $\approx \lambda_{q+1}$ and $\bar a_\xi$ and $\widetilde \Phi$ (heuristically) have frequencies $\lesssim \lambda_q$, we are justified to perform the approximations 
\begin{equation*}
    \bar a_\xi(x - h_1) \approx a_\xi(x),
\end{equation*}
\begin{equation*}
    \widetilde \Phi(x - h_1) \approx \widetilde \Phi(x) - \nabla \widetilde \Phi(x) h_1.
\end{equation*}
This determines the first order of the quadratic interaction:
\begin{eqnarray*}
    u_{q+1}^{(p)} \cdot \nabla \theta_{q+1}^{(p)} &=& \nabla^\perp \cdot \div \left(\sum_{\xi \in F} \sum_{\eta, \zeta \in \{-\xi, \xi\}}  g_\xi^2 \bar a_\xi^2 \widehat K_{\la_{q+1}}(\lambda_{q+1} (\na \widetilde \Phi)^T \eta, \lambda_{q+1} (\na \widetilde \Phi)^T \zeta) e^{i \la_{q+1} \widetilde \Phi \cdot (\eta + \zeta)} + g_\xi^2 \delta B_\xi \right) \\ 
    &=& \nabla^\perp \cdot \div \left(\sum_{\xi \in F} g_{\xi}^2 \underbrace{\frac14\frac{1}{\lambda_{q+1}}\bar a_\xi^2\frac{(\nabla\widetilde \Phi)^{T}\xi\otimes \xi (\nabla\widetilde \Phi)}{|(\nabla\widetilde \Phi)^{T}\xi|^{3}}}_{\bar A_\xi} + g_\xi^2 \delta B_\xi \right),
\end{eqnarray*}
where $\delta B_\xi$ are lower order terms. The following cancellation is, therefore, achieved
\begin{equation*}
    \nabla^\perp \cdot \div R_q^{\text{rem}} +  u_{q+1}^{(p)}\cdot \nabla \theta_{q+1}^{(p)} = \nabla^\perp \cdot \div \sum_{\xi \in F} g_\xi^2\left( \underbrace{\bar A_\xi - A_\xi}_{R_{q+1}^{\text{flow}}} +  \delta B_\xi \right).
\end{equation*}

The flow error $R_{q+1}^{\text{flow}}$ is due to the discrepancy between the flow maps $\widetilde \Phi$ and $\Phi$. We expect, then, to have the estimate
\[\|R_{q+1}^{\text{flow}}\|_0\lesssim \delta_{q+1}\|\nabla\Phi-\nabla\widetilde \Phi\|_0.\]
The difference of the flow maps satisfies
\begin{equation}\notag
\begin{cases}
\partial_t(\Phi-\widetilde \Phi)+u_q\cdot\nabla (\Phi-\widetilde \Phi)=w_{q+1}^{(t)}\cdot\nabla\widetilde \Phi,\\
(\Phi-\widetilde \Phi)|_{t=t_0}=0.
\end{cases}
\end{equation}
It follows that, on time-scales of size $\tau_q$,
\[\|\nabla\Phi-\nabla\widetilde \Phi\|_0\lesssim \lambda_q\tau_q\|w_{q+1}^{(t)}\|_0,\]
and, hence,
\[\|R_{q+1}^{\text{flow}}\|_0\lesssim \frac{\delta_{q+1}^2 \lambda_q^3 \tau_q}{\mu_{q+1}}.\]

On the other hand, the lower order term $\delta B_\xi$ will satisfy estimates which have a $\lambda_q/\lambda_{q+1}$ improvement over the first order term. Consequently, we expect the estimate 
\begin{equation*}
    \|\delta B_\xi\|_0 \lesssim \delta_{q+1}\frac{\lambda_q}{\lambda_{q+1}}.
\end{equation*}

The main drawback of using temporal oscillations to decouple directions can be seen in the transport error:
\[R_{q+1}^{\text{transport}}=\div^{-1}\Delta^{-1}\nabla^{\perp}\cdot\left((\partial_t+u_{q,\Gamma}\cdot\nabla)\theta_{q+1}^{(p)}\right),\]
which involves a term in which the material derivative falls on the temporally oscillatory profiles $g_\xi$. Given that the differential operator above is of order $-2$, we expect the estimate
\[\|R_{q+1}^{\text{transport}}\|_0\lesssim \lambda_{q+1}^{-2}\mu_{q+1}\lambda_{q+1}^{\frac{1}{2}}\delta_{q+1}^{\frac12}.\]

\subsubsection{Optimizing the errors and choosing $\mu_{q+1}$}

The estimates for the transport error $R_{q+1}^{\text{transport}}$ become better when the temporal frequency $\mu_{q+1}$ is smaller, whereas the estimates for the Newton error $R_{q+1}^{\text{Newton}}$ are improved when $\mu_{q+1}$ is large. This leads us to optimizing the size of $\mu_{q+1}$ by enforcing the same bounds for both of these errors. This results in the following choice:
\[\mu_{q+1}=\delta_{q+1}^{\frac12}\lambda_{q+1}^{\frac12}\lambda_q.\]
The necessary condition $\mu_{q+1} > \tau_q^{-1}$ is, thus, satisfied.

\subsubsection{Critical regularity threshold}
We let $R_{q+1}$ be the sum of the errors remaining upon adding the Nash perturbation. In order to propagate the inductive estimates on the Reynolds stress, we require $R_{q+1} \lesssim \delta_{q+2}$. By the choice of $\mu_{q+1}$, we have 
\begin{equation*}
    \|R_{q+1}^{\text{transport}}\|_0 + \|R_{q+1}^{\text{Newton}}\|_0 \lesssim \delta_{q+1} \frac{\lambda_q}{\lambda_{q+1}},
\end{equation*}
which leads to the condition 
\begin{eqnarray*}
    \delta_{q+1} \frac{\lambda_q}{\lambda_{q+1}} \leq \delta_{q+2} \iff \beta < \frac{1}{2b}.
\end{eqnarray*}
Since 
\begin{equation*}
    \|\theta_{q+1}^{(p)}\|_0 \lesssim \delta_{q+1}^{1/2} \lambda_{q+1}^{1/2},
\end{equation*}
and $b = 1^{+}$, this is consistent with solutions $\theta \in C_t^0 C_x^{0-}$. We remark that the error given by the lower order terms $\delta B_\xi$ from the microlocal lemma satisfy the same estimate. 

Requiring $\|R_{q+1}^{\text{flow}}\|_0\leq \delta_{q+2}$ yields the condition
\begin{equation*}
    \delta_{q+1} \left(\frac{\delta_{q+1} \lambda_q}{\delta_q \lambda_{q+1}} \right)^{1/2} \leq \delta_{q+1} \iff \beta < \frac{1}{2(2b-1)},
\end{equation*}
which, once again, allows $\beta = 1/2^{-}$ and is consistent with the sharp Onsager regularity.

\section{The Newton steps}
\label{sec-Newton}
\subsection{Spatial mollification} As by now standard in Nash iteration schemes, the construction starts with a mollification of the SQG-Reynolds system, whose purpose is to deal with the problem of loss of derivatives. This subsection is the same as the analogous one in \cite{GR23} and follows ideas used in all Nash iterations.

We define the mollification length scale as
\begin{equation*}
    \ell_q = (\lambda_q \lambda_{q+1})^{-1/2},
\end{equation*}
and let $\zeta:\mathbb R^2 \rightarrow \mathbb R$ be a smooth function such that its Fourier transform 
\begin{equation*}
    \hat \zeta(\xi) = \int_{\mathbb R^2} \zeta(x) e^{-ix\cdot \xi} dx
\end{equation*}
satisfies $\hat \zeta(\xi) = 1$ when $|\xi| \leq 1$ and $\hat \zeta (\xi) = 0$ when $|\xi| \geq 2$. We, moreover, assume $\hat \zeta$ (and, thus, $\zeta$) is symmetric. Given a periodic function $f:\mathbb T^2 \rightarrow \mathbb R$, let 
\begin{eqnarray*}
    P_{\lesssim \ell_q^{-1}} f(x) = \sum_{k \in \mathbb Z^2} \hat f(k) \hat \zeta (k \ell_q) e^{i k \cdot x}.
\end{eqnarray*}
Equivalently, 
\begin{equation*}
    P_{\lesssim \ell_q^{-1}} f(x) = \int_{\mathbb R^2} f(x-y) \zeta_{\ell_q}(y) dy,
\end{equation*}
where $f$ is identified with its periodic extension, and 
\begin{equation*}
    \zeta_{\ell_q}(x) = \ell_q^{-2} \zeta(\ell_q^{-1}x).
\end{equation*}
The definition is easily extended to vector fields, tensor fields, etc. 

We denote
\begin{equation*}
    \bar \theta_q = P_{\lesssim \ell_q^{-1}} \theta_q,
\end{equation*}
\begin{equation*}
    \bar u_q= T[\bar \theta_q] = P_{\lesssim \ell_q^{-1}} u_q , 
\end{equation*}
\begin{equation*}
    R_{q, 0} = P_{\lesssim \ell_q^{-1}} R_q .
\end{equation*} 
and record the relevant estimates in the following lemma. To ease notation, we write
\begin{equation*}
    \bar D_t = \partial_t + \bar u_q \cdot \nabla
\end{equation*}
for the material derivative associated to the mollified velocity field $\bar u_q$. We remark that $\zeta$ satisfies all of the required properties of proposition \ref{prop-molli}, and, thus, the proof of the following lemma, detailed in \cite{GR23}, goes through.

\begin{lem} \label{smoli_estim}
    The following estimates 
\begin{equation} \label{smoli_1}
    \|\bar \theta_q\|_N+\|\bar u_q\|_N \lesssim \delta_q^{\frac12} \lambda_q^{\frac{1}2+N}, \,\,\, \forall N \in \{0,1,2,...,L_\theta\},
\end{equation}
\begin{equation} \label{smoli_2}
    \|R_{q,0}\|_N \lesssim \delta_{q+1} \lambda_q^{N-2\alpha}, \,\,\, \forall N \in \{0, 1,..., L_R\},
\end{equation}
\begin{equation} \label{smoli_3}
    \|\bar D_t R_{q, 0}\|_N \lesssim \delta_{q+1} \delta_q^{\frac12} \lambda_q^{N+\frac{3}{2}-2\alpha}, \,\,\, \forall N \in \{0, 1,..., L_t\}
\end{equation}
\begin{equation} \label{smoli_4}
    \|\bar \theta_q\|_{N + L_\theta}+\|\bar u_q\|_{N + L_\theta} \lesssim \delta_q^{\frac12} \lambda_q^{\frac{1}2+L_\theta} \ell_q^{-N}, \,\,\, \forall N \geq 0,
\end{equation}
\begin{equation} \label{smoli_5}
    \|R_{q,0}\|_{N+L_R} \lesssim \delta_{q+1} \lambda_q^{L_R-2\alpha} \ell_q^{-N}, \,\,\, \forall N \geq 0,
\end{equation}
\begin{equation} \label{smoli_6}
    \|\bar D_t R_{q,0}\|_{N+L_t} \lesssim \delta_{q+1} \delta_q^{\frac12}\lambda_q^{L_t +\frac{3}2-2\alpha} \ell_q^{-N}, \forall N \geq 0
\end{equation}
hold with implicit constants depending on $M$ and $N$ \footnote{Here, and throughout, by dependence on $N$, we mean dependence on the norm being estimated. Strictly speaking, the constant in the estimate of $\|\cdot\|_{N+L}$ will depend on $N+L$.}. 
\end{lem}

For a detailed proof, we refer the reader to \cite{GR23}. We point out only that all but the statements on the material derivative are standard mollification estimates. For the material derivative, the Constantin-E-Titi commutator estimate is employed, and we use implicitly that $L_t \leq \min\{L_\theta, L_R\} - 1$. 

\subsection{Flow map estimates} 

We define the backward flow $\Phi_t : \T^2 \times \R \to \T^2$ starting at time $t \in \mathbb R$ as 
\begin{equation} \label{Flow_t}
    \begin{cases}
        \partial_s \Phi_t(x,s) + \bar u_q(x,s) \cdot \nabla \Phi_t(x,s) = 0 \\ 
        \Phi_t \big|_{s = t}(x) = x,
    \end{cases}
\end{equation}
and the corresponding Lagrangian flow $X_t$ by 
\begin{equation} \label{Lagr_t}
    \begin{cases}
        \frac{d}{ds}X_t(\alpha, s) = \bar u_q(X_t(\alpha, s), s) \\ 
        X_t(\alpha, t) = \alpha.
    \end{cases}
\end{equation}
The following contains standard estimates on the Lagrangian and backward flows of $\bar u_q$. The reader is referred to \cite{GR23} for the proof.

\begin{lem} \label{Flow_estim}
    Let $\tau \leq  \|\bar u_q\|_1^{-1}$. For $\Phi_t$ defined in  \eqref{Flow_t} and $X_t$ the Lagrangian flow \eqref{Lagr_t}, we have for any $|s - t| < \tau$, 
    \begin{equation} \label{Flow_estim_1}
        \|(\nabla \Phi_t)^{-1}(\cdot, s)\|_N + \|\nabla \Phi_t (\cdot, s)\|_N \lesssim \lambda_q^N, \,\,\, \forall N \in\{0,1,..., L_\theta-1\},
    \end{equation}
    \begin{equation} \label{Flow_estim_2}
        \|\bar D_t (\nabla \Phi_t)^{-1}(\cdot, s)\|_N + \|\bar D_t \nabla \Phi_t (\cdot, s)\|_N \lesssim \delta_q^{1/2} \lambda_q^{N+3/2}, \,\,\, \forall N \in\{0,1,..., L_\theta-1\},
    \end{equation}
    \begin{equation} \label{Lagr_estim_1}
         \|D X_t(\cdot, s)\|_N  \lesssim \lambda_q^N, \,\,\, \forall N \in \{0,1,..., L_\theta-1\},
    \end{equation}
    \begin{equation} \label{Flow_estim_3}
        \|(\nabla \Phi_t)^{-1}(\cdot, s)\|_{N+L_\theta-1} + \|\nabla \Phi_t(\cdot, s)\|_{N+L_\theta-1}  \lesssim \lambda_q^{L_\theta-1} \ell_q^{-N}, \, \, \, \forall N \geq 0,
    \end{equation}
    \begin{equation} \label{Flow_estim_4}
        \|\bar D_t (\nabla \Phi_t)^{-1}(\cdot, s)\|_{N+L_\theta-1} + \|\bar D_t \nabla \Phi_t(\cdot, s)\|_{N+L_\theta-1}  \lesssim \delta_q^{1/2} \lambda_q^{L_\theta+1/2} \ell_q^{-N}, \, \, \, \forall N \geq 0,
    \end{equation}
    \begin{equation} \label{Lagr_estim_2}
        \|D X_t(\cdot, s)\|_{N+L_\theta-1} \lesssim \lambda_q^{L_\theta-1} \ell_q^{-N}, \, \, \, \forall N \geq 0,
    \end{equation}
    where the implicit constants depend on $M$ and $N$.
\end{lem}

\subsection{Toolbox for temporal oscillation and localization}\label{sec-prep}
To prepare for the construction of the iterative Newton perturbations, we need to define several time-dependent functions: partition of unity, cut-offs, temporally oscillatory profiles. This subsection is also the same as the analogous one in \cite{GR23}. We start by introducing the natural time scale
\begin{equation*}
    \tau_q = \delta_q^{-\frac12} \lambda_q^{-\frac{3}2} \lambda_{q+1}^{-\alpha},
\end{equation*}
which is chosen so as to satisfy
\begin{equation*}
    \|\bar u_q\|_{1+\alpha} \tau_q \lesssim \bigg(\frac{\lambda_q}{\lambda_{q+1}}\bigg)^\alpha \lesssim 1,
\end{equation*}
and 
\begin{equation*}
     \|\bar u_q\|_1 \tau_q \leq  C \lambda_{q+1}^{-\alpha},
\end{equation*}
with $C>0$ depending only on $M$. In particular, for sufficiently large $a_0$ depending on $M$ and $\alpha$ such that 
\begin{equation*}
    C \lambda_{q+1}^{-\alpha} \leq 1,
\end{equation*}
lemma \ref{Flow_estim} holds with $\tau$ replaced with $\tau_q$. 

Let $t_k = k \tau_q$, for $k \in \mathbb{Z}$. We first define a partition of unity in time. To this end, we choose the cut-off functions $\{\chi_k\}_{k \in \mathbb Z}$ which satisfy: 
\begin{itemize}
    \item The squared cut-offs form a partition of unity: 
    \begin{equation*}
        \sum_{k \in \mathbb Z} \chi_k^2(t) = 1;
    \end{equation*}
    \item $\supp \chi_k \subset (t_k - \frac{2}{3} \tau_q, t_k + \frac{2}{3} \tau_q)$. Consequently, 
    \begin{equation*}
        \supp \chi_{k-1} \cap \supp \chi_{k+1} = \varnothing, \forall k \in \mathbb Z;
    \end{equation*}
    \item For $N \geq 0$ and $k \in \mathbb Z$,
    \begin{equation*}
        |\partial_t^N \chi_k| \lesssim \tau_q^{-N},
    \end{equation*}
    where the implicit constant depends only on $N$.
\end{itemize}

We also need another set of cut-off functions $\{\tilde \chi_k \}_{k \in \mathbb Z}$, which will be used to temporally localize the Newton perturbations. We require: 
\begin{itemize}
    \item $\supp \tilde \chi_k \subset (t_k - \tau_q, t_k + \tau_q)$ and $\tilde \chi_k = 1$ on $(t_k - \frac{2}{3} \tau_q, t_k + \frac{2}{3} \tau_q)$. Note in particular that 
    \begin{equation*}
        \chi_k \Tilde \chi_k = \chi_k, 
        \,\,\, \forall k \in \mathbb Z.
    \end{equation*}
    \item For any $N \geq 0$ and $k \in \mathbb Z$, 
    \begin{equation*}
        |\partial_t^N \tilde \chi_k| \lesssim \tau_q^{-N},
    \end{equation*}
    where the implicit constant depends only on $N$.
\end{itemize}
Let the number of implemented Newton steps be given by
\begin{equation}\label{def:ga}
    \Gamma := \bigg\lceil \frac{1}{\frac12 - \beta} \bigg \rceil.
\end{equation}
This determines the temporally oscillatory profiles through the following simple lemma, which is proved in \cite{GR23}.

\begin{lem} \label{osc_prof}
Let $F \subset \mathbb Z^2$ be the set given by lemma \ref{le-geo}, and $\Gamma \in \mathbb N$. For any $\xi \in F$, there exist $2\Gamma$ smooth $1$-periodic functions $g_{\xi, e, n}, g_{\xi, o, n}:\mathbb R \rightarrow \mathbb R$ with $n \in \{1, 2,..., \Gamma\}$ such that
\begin{equation*}
    \int_0^1 g_{\xi, p, n}^2 = 1, \ \ \forall \xi \in F \text{, } p \in \{e, o\} \text{, and } n \in \{1,2,..., \Gamma\} ;
\end{equation*}
and 
\begin{equation*}
    \supp g_{\xi, p, n} \cap \supp g_{\eta, q, m} = \varnothing,
\end{equation*}
whenever $(\xi, p, n) \neq (\eta, q, m) \in F \times \{e, o\} \times \{1, 2,..., \Gamma\} $. 
\end{lem}

\subsection{The SQG-Reynolds system after \texorpdfstring{$n$}{nnn} Newton steps}
Let $n \in \{0, 1, ..., \Gamma - 1\}$. The iterative system after $n$ perturbations will be in the form 
\begin{equation} \label{steps}
    \begin{cases}
        \partial_t \theta_{q, n} + u_{q, n}\cdot\na \theta_{q, n} = \nabla^{\perp}\cdot\div R_{q, n} + \nabla^{\perp}\cdot\div S_{q, n} + \nabla^{\perp}\cdot\div P_{q + 1, n}, \\
        u_{q, n} = T[\theta_{q,n}],
    \end{cases}
\end{equation}
where 
\begin{itemize}
    \item $\theta_{q, n}$ is the SQG scalar to be defined starting from $\theta_{q, 0} = \theta_q$ by adding $n$ perturbations; 
    \item $R_{q, n}$ is the gluing error of the $n^{\text{th}}$ perturbation for $n \geq 1$, while $R_{q, 0}$ is the already defined mollified stress;
    \item $S_{q, n}$ is the error to be erased by the non-interacting highly-oscillatory Nash perturbations. It will be inductively defined starting from $S_{q, 0} = 0$;
    \item $P_{q + 1, n}$ is a small residue error which will be also inductively defined starting from $P_{q+1, 0} = R_q - R_{q, 0}$.
\end{itemize}
Note that \eqref{steps} with $n=0$ is simply the relaxed system \eqref{ASE-q}.

\subsection{Construction of the Newton perturbations} 

We now construct the solution of (\ref{steps}) at the level $(q,n+1)$. As a first step, we use the geometric lemma \ref{le-geo} to decompose the stress $R_{q, n}$ into simple tensors which are adapted to the coordinates imposed by the coarse grain flow of $\bar u_q$. Let $\Phi_k$ be the backward flow satisfying 
\begin{equation*}
    \begin{cases}
        \partial_t \Phi_k + \bar u_q \cdot \nabla \Phi_k = 0 \\ 
        \Phi_k \big|_{t = t_k} = x.
    \end{cases}
\end{equation*}
For $n \in \{0,1,...,\Gamma-1\}$, $k \in \mathbb Z$ and $\xi \in F$, we define the amplitude functions
\begin{equation}\label{def.a.coeff}
    a_{\xi, k, n} = 2 \delta_{q+1, n}^{\frac12}\la_{q+1}^{\frac{1}2}|(\nabla\Phi_k)^{T}\xi|^{\frac{3}2} \chi_k \gamma_\xi \Big((\nabla \Phi_k)^{-T} (\nabla \Phi_k)^{-1} - (\nabla \Phi_k)^{-T} \frac{R_{q, n}}{\delta_{q+1,n}} (\nabla \Phi_k)^{-1} \Big),
\end{equation}
with $\gamma_\xi$ given by lemma \ref{le-geo}, and the amplitude parameters $\delta_{q+1, n}$ defined by 
\begin{equation*}
    \delta_{q+1, n} := \delta_{q+1} \bigg(\frac{\lambda_q}{\lambda_{q+1}}\bigg)^{n(\frac12 - \beta)}.
\end{equation*}
The functions $a_{\xi, k, n}$ are well-defined since 
\begin{equation*}
(\nabla \Phi_k)^{-T} (\nabla \Phi_k)^{-1} - (\nabla \Phi_k)^{-T} \frac{R_{q, n}}{\delta_{q+1,n}} (\nabla \Phi_k)^{-1} \lesssim \lambda_q^{-\alpha},
\end{equation*}
and, therefore, $a_0$ can be chosen sufficiently large so that 
\begin{equation*}
    (\nabla \Phi_k)^{-T} (\nabla \Phi_k)^{-1} - (\nabla \Phi_k)^{-T} \frac{R_{q, n}}{\delta_{q+1,n}} (\nabla \Phi_k)^{-1} \in B_{1/2}(\I).
\end{equation*}

Let $\mathcal N_{\tau}(A)$ denote the neighbourhood of size $\tau$ of the set $A$, and define
\begin{equation*}
    \mathbb Z_{q,n} := \big\{ k \in \mathbb Z \mid k \tau_q \in \mathcal{N}_{\tau_q}(\supp_t R_{q,n})\big\}.
\end{equation*}
We observe that 
\begin{equation*}
    \sum_{k \in \mathbb Z_{q,n}} \chi_k^2(t) = 1 , \ \ \mbox{for}\ \ t \in \supp_t R_{q,n}.
\end{equation*}
Invoking lemma \ref{le-geo}, we find that 
\begin{equation}\label{decomp}
    \nabla^{\perp}\cdot\div \left[ \sum_{k \in \mathbb Z_{q,n}} \sum_{\xi \in F} \frac{a^2_{\xi, k, n}}{4 \lambda_{q+1} |(\nabla \Phi_k)^T \xi|^3} (\nabla \Phi_k)^T \xi \otimes \xi \nabla \Phi_k   \right] =-\nabla^{\perp}\cdot \div R_{q,n}.
\end{equation}
To ease notation, we denote
\begin{equation*}
    A_{\xi, k, n} := \frac{a^2_{\xi, k, n}}{4 \lambda_{q+1} |(\nabla \Phi_k)^T \xi|^3} (\nabla \Phi_k)^T \xi \otimes \xi \nabla \Phi_k .
\end{equation*}

The temporal oscillations will be characterized by the frequency parameter
\begin{equation*}
    \mu_{q+1} = \delta_{q+1}^{\frac12} \lambda_q \lambda_{q+1}^{\frac12} \lambda_{q+1}^{4\alpha},
\end{equation*}
which satisfies
\begin{equation*}
    \mu_{q+1} \tau_q = \bigg( \frac{\lambda_{q+1}}{\lambda_q} \bigg)^{\frac12 - \beta} \lambda_{q+1}^{3\alpha} > 1,
\end{equation*}
since $\beta < 1/2$. Consider $f_{\xi, k, n+1}:\mathbb R \rightarrow \mathbb R$ defined by
\begin{equation}\label{def-temporal-f}
    f_{\xi, k, n+1} := 1 - g^2_{\xi, k, n+1},
\end{equation}
with
\begin{equation*}
    g_{\xi, k, n+1} = 
    \begin{cases}
        g_{\xi, e, n+1} & \text{if } k \text{ is even}, \\ 
        g_{\xi, o, n+1}  & \text{if } k \text{ is odd}.
    \end{cases}
\end{equation*}
The primitive of $f_{\xi, k, n+1}$ is given by
\begin{equation*}
    f^{[1]}_{\xi, k, n+1} (t) = \int_0^t  f_{\xi, k, n+1} (s) ds,
\end{equation*}
which is a well-defined $1$-periodic function, since $g_{\xi, e, n+1}$ and $g_{\xi, o, n+1}$ are normalized in $L^2$. We note that the functions $g_{\xi,k,n+1}$, $f_{\xi, k, n+1}$ and $f^{[1]}_{\xi, k, n+1}$ do not depend on the iteration stage $q$. In addition, the number of these functions is finite, depending only on $\Ga$ and the size of $F$.

We finally consider the solution of the forced Newtonian linearization of the SQG equations,
\begin{equation} \label{LocalNewt}
\begin{cases}
    \partial_t \theta_{k, n + 1} + \bar u_q \cdot \nabla \theta_{k, n + 1} + T[\theta_{k, n + 1}] \cdot \nabla \bar \theta_q = \sum_{\xi\in F} f_{\xi, k, n+1}(\mu_{q + 1} t) \nabla^{\perp}\cdot \div A_{\xi, k, n}(x,t), \\
    \theta_{k, n + 1} \big |_{t = t_k}(x) = \frac{1}{\mu_{q+1}}\sum_{\xi\in F} f^{[1]}_{\xi, k, n+1} (\mu_{q+1} t_k) \nabla^{\perp}\cdot \div A_{\xi, k, n}(x, t_k).
    \end{cases}
\end{equation}
The well-posedness theory for smooth solutions of (\ref{LocalNewt}) can be obtained by standard arguments. The reader could, for instance, slightly modify the proof given in appendix E of \cite{GR23}.

We can now define the $(n+1)^{\text{th}}$ Newton perturbation by the superposition of temporal localizations of the scalars $\theta_{k, n+1}$:
\begin{equation*}
    \theta^{(t)}_{q + 1 , n + 1} (x, t) = \sum_{k \in \mathbb Z_{q,n}} \tilde \chi_k (t) \theta_{k, n + 1}(x, t)\,.
\end{equation*}

\subsection{The errors after the \texorpdfstring{$(n+1)^{\text{th}}$}{thth} step and the inductive proposition}
With the perturbation $\theta_{q+1, n+1}^{(t)}$ at hand, we can compute the new error terms $R_{q,n+1}$, $S_{q,n+1}$, and $P_{q+1,n+1}$. Applying the linearized SQG operator to $\theta_{q,n+1} = \theta_{q, n} + \theta_{q+1, n+1}^{(t)}$ yields
\begin{align*}
    \partial_t \theta^{(t)}_{q+1,n+1} + \bar u_q \cdot \nabla \theta^{(t)}_{q+1,n+1} + T[\theta^{(t)}_{q+1,n+1}] \cdot \nabla \bar \theta_q &= \sum_{k \in \mathbb Z_{q,n}} \sum_{\xi\in F} \tilde \chi_k (t) f_{\xi, k,n+1}(\mu_{q + 1} t) \nabla^{\perp}\cdot\div A_{\xi, k, n}\\
    &\qquad + \sum_{k \in \mathbb Z_{q,n}} \partial_t \tilde \chi_k \theta_{k, n+1}.
\end{align*}
 We note that $\tilde \chi_k A_{\xi, k, n} = A_{\xi, k, n}$ for $\forall k \in \mathbb Z$, since $\supp A_{\xi, k, n} \subset \supp a_{\xi, k, n} \subset \supp \chi_k \times \mathbb{T}^2$. Therefore, it follows from \eqref{decomp} and \eqref{def-temporal-f} that
\begin{eqnarray*}
    \sum_{k \in \mathbb Z_{q,n}} \sum_{\xi\in F} \tilde \chi_k (t) f_{\xi, k, n+1}(\mu_{q + 1} t) \nabla^{\perp}\cdot\div A_{\xi, k, n}
    & = & \sum_{k \in \mathbb Z_{q,n}} \sum_{\xi\in F} \nabla^{\perp}\cdot\div A_{\xi, k, n} - \sum_{k \in \mathbb Z_{q,n}} \sum_{\xi \in F}  g_{\xi, k, n+1}^2  \nabla^{\perp}\cdot\div A_{\xi, k, n} \\ 
    & = & - \nabla^{\perp}\cdot\div R_{q,n} - \sum_{k \in \mathbb Z_{q,n}} \sum_{\xi \in F}  g_{\xi, k, n+1}^2  \nabla^{\perp}\cdot\div A_{\xi, k, n}.
\end{eqnarray*}
As a consequence, we see that the system \eqref{steps} at the $(n+1)^{\text{th}}$ step is satisfied with 
\begin{equation} \label{Newvelo}
    \theta_{q, n + 1} = \theta_{q, n} + \theta_{q+1, n + 1}^{(t)} = \theta_q + \sum_{m =1}^{n+1} \theta_{q+1, m}^{(t)},
\end{equation}
\begin{equation} \label{NewStr}
    R_{q, n + 1} = \div^{-1} \Delta^{-1}\nabla^{\perp}\sum_{k \in \mathbb Z_{q,n}} \partial_t \tilde \chi_k \theta_{k, n+1},
\end{equation}
\begin{equation} \label{NewNash}
     S_{q, n + 1} = S_{q, n} - \sum_{k \in \mathbb Z_{q,n}} \sum_{\xi \in F}   g_{\xi, k, n+1}^2  A_{\xi, k, n}
\end{equation}
\begin{eqnarray} \label{SmalStr}
    P_{q + 1, n + 1} &= & P_{q + 1, n} + \div^{-1}\Delta^{-1}\nabla^{\perp}\left(T[\theta_{q+1,n+1}^{(t)}]\cdot\nabla \theta_{q+1, n+1}^{(t)}\right)\notag \\
    &&+ \div^{-1}\Delta^{-1}\nabla^{\perp}\sum_{m = 1}^{n} \big (T[\theta_{q+1, n+1}^{(t)}]\cdot\nabla\theta_{q+1, m}^{(t)} + T[ \theta_{q+1, m}^{(t)}]\cdot\nabla\theta_{q+1, n+1}^{(t)} \big ) \notag \\ 
    && + \div^{-1}\Delta^{-1}\nabla^{\perp}\left(T[(\theta_q - \bar \theta_q)]\cdot\nabla \theta_{q+1, n+1}^{(t)} + T[\theta_{q+1, n+1}^{(t)}]\cdot\nabla(\theta_q - \bar \theta_q)\right).
\end{eqnarray}
The operator $\div^{-1}$ used above and throughout the paper is described in appendix \ref{sec-geo}. The following is the main inductive proposition of this section.

\begin{prop} \label{NewIter}
Assume $R_{q, n}$ satisfies 
\begin{equation} \label{NewIter_1}
    \|R_{q, n}\|_{N} \leq \delta_{q+1, n} \lambda_q^{N - \alpha}, \,\,\, \forall N \in \{0, 1,..., L_t\},
\end{equation} 
\begin{equation} \label{NewIter_2}
    \|\bar D_t R_{q, n}\|_N \leq \delta_{q+1,n} \tau_q^{-1} \lambda_q^{N - \alpha}, \,\,\, \forall N \in \{0, 1,..., L_t\},
\end{equation}
\begin{equation} \label{NewIter_3}
    \|R_{q, n}\|_{N+L_t} \lesssim \delta_{q+1, n} \lambda_q^{L_t -\alpha} \ell_q^{-N}, \,\,\, \forall N \geq 0
\end{equation}
\begin{equation} \label{NewIter_4}
    \|\bar D_t R_{q,n}\|_{N + L_t} \lesssim \delta_{q+1, n} \tau_q^{-1} \lambda_q^{L_t - \alpha} \ell_{q}^{-N}, \,\,\, \forall N \geq 0
\end{equation}
with implicit constants depending on $n$, $\Gamma$, $M$, $\alpha$ and $N$. In addition, suppose that 
\begin{eqnarray} \label{NewIter_5}
    \supp_t R_{q,n} &\subset& [-2 +(\de_q^{\frac12}\la_q^{\frac{3}2})^{-1} - 2n\tau_q, -1 -(\de_q^{\frac12}\la_q^{\frac{3}2})^{-1} + 2n\tau_q] \\ 
    && \cup [1 + (\de_q^{\frac12}\la_q^{\frac{3}2})^{-1} - 2n \tau_q, 2 - (\de_q^{\frac12}\la_q^{\frac{3}2})^{-1} + 2n \tau_q], \nonumber
\end{eqnarray}
Then, the new stress $R_{q, n+1}$ satisfies \eqref{NewIter_1}-\eqref{NewIter_5} with $n$ replaced by $n+1$. 
\end{prop}

We first claim that, since $\tau_q < (\delta_q^{1/2} \lambda_q^{3/2})^{-1}$, lemma \ref{smoli_estim} implies that the assumptions of proposition \ref{NewIter} are satisfied at $n = 0$. Indeed, there exists a constant $C>0$ depending only on $L_t \leq L_R$ such that 
\begin{equation*}
    \|R_{q, 0}\|_N \leq C \delta_{q+1} \lambda_q^{N-2\alpha}, \,\,\, \forall N \leq L_R,
\end{equation*}
and
\begin{equation*}
    \|\bar D_t R_{q, 0}\|_N \leq C \delta_{q+1} \tau_q^{-1}\lambda_q^{N-2\alpha}, \,\,\, \forall N \leq L_t.
\end{equation*}
Thus, the assumptions are verified provided
\begin{equation*}
    C_L \lambda_q^{- \alpha} \leq 1,
\end{equation*}
which can be ensured for any $\alpha>0$ by choosing $a_0$ sufficiently large. 

Moreover, it follows from the definitions of $\tilde \chi_k$ and $\mathbb Z_{q,n}$ that
\begin{equation*}
    \supp_t R_{q, n+1} \subset \overline{\mathcal{N}_{2\tau_q} (\supp_t R_{q, n})},
\end{equation*}
which immediately shows that the assumptions on the support propagate. 

It remains to show the estimates at the $(n+1)$-th step in proposition \ref{NewIter}. This is the content of the next subsection, along with the estimates for the perturbation $\theta_{q+1, n+1}^{(t)}$. 

\subsection{Proof of the inductive proposition}
Let $\psi_{k, n+1}$ be the Poisson potential of the density $\theta_{k, n+1}$,
\begin{eqnarray*}
    \begin{cases}
        \Delta \psi_{k, n+1} = \theta_{k, n+1} \\
        \fint_{\mathbb T^2} \psi_{k, n+1} = 0,
    \end{cases}
\end{eqnarray*}
which is well-defined since $\theta_{k, n+1}$ has zero-mean. We have that
\begin{equation}\label{R-q-potential}
    R_{q, n+1} = \div^{-1}\nabla^\perp \sum_{k \in \mathbb Z_{q,n}} \partial_t \tilde \chi_k \psi_{k, n+1}. 
\end{equation} 
Proposition \ref{NewIter} will, therefore, follow from the estimates for the potential functions. By applying $\Delta^{-1}$ to \eqref{LocalNewt}, we find that $\psi_{k, n+1}$ satisfies:
\begin{equation} \label{psi_eqn}
    \begin{cases}
        \partial_t \psi_{k, n + 1} + \bar u_q \cdot \nabla \psi_{k, n + 1} + T[\psi_{k,n+1}]\cdot\na \bar \theta_q\\
        \qquad\qquad + \Delta^{-1} \div([\bar u_q,\Delta]\psi_{k, n + 1} + [\bar \theta_q,\Delta]T[\psi_{k,n+1}]) = \sum_{\xi \in F} f_{\xi, k, n}(\mu_{q+1} t) \Delta^{-1} \nabla^\perp \cdot \div A_{\xi, k, n} \\ 
        \psi_{k, n+1} \big|_{t = t_k} = \frac{1}{\mu_{q+1}}\sum_{\xi \in F} f_{\xi, k, n}^{[1]}(\mu_{q+1} t_k) \Delta^{-1} \nabla^\perp \cdot \div A_{\xi, k, n}\big|_{t = t_k},
    \end{cases}
\end{equation}

We now recast the term 
$$ \Delta^{-1} \div([\bar u_q,\Delta]\psi_{k, n + 1} + [\bar \theta_q,\Delta]T[\psi_{k,n+1}]) $$
on the left-hand side of ~\eqref{psi_eqn} into a more amenable form. In the following, we use the convention of summation over repeated indices. Note that
    \begin{align*}
        [\bar u_q,\Delta]\psi_{k, n + 1} &= \bar u_q \pa_{jj} \psi_{k, n + 1} - \pa_{jj}(\bar u_q \psi_{k, n + 1}) = \pa_{jj} \bar u_q \psi_{k, n + 1} -2 \pa_j(\pa_j \bar u_q \psi_{k, n + 1})\,,
    \end{align*}
    and, similarly,
    \begin{align*}
        [\bar \theta_q,\Delta]T[\psi_{k,n+1}] &= \pa_{jj}\bar \theta_q T[\psi_{k, n + 1}] - 2\pa_j (\pa_j \bar \theta_q T[\psi_{k, n + 1}])\,.
    \end{align*}
    Consider $S[\theta, \psi] := \Delta^{-1}\div(\psi \Delta T[\theta] +  T[\psi] \Delta \theta)$ and note that
    \begin{align}\label{eqn.simpl}
        \Delta^{-1} \div([\bar u_q,\Delta]\psi_{k, n + 1} + [\bar \theta_q,\Delta]T[\psi_{k,n+1}]) = -2 \Delta^{-1} \div(\pa_j(\pa_j &\bar u_q \psi_{k, n + 1}) + \pa_j (\pa_j \bar \theta_q T[\psi_{k, n + 1}])) \notag\\ &+ S[\bar \theta_q, \psi_{k,n+1}]\,.
    \end{align}
    The advantage of this form is that $\Delta^{-1} \div \partial_j$ is a zero-order operator. The bilinear Fourier multiplier operator $S$ requires a more careful analysis which is carried out in section 
    \ref{sec-tech}.
    

To proceed, we collect estimates for $a_{\xi, k, n}$ and $A_{\xi, k, n}$. These are by now standard and follow from the assumed estimates of proposition \ref{NewIter} and lemmas \ref{smoli_estim} and \ref{Flow_estim}, together with basic interpolation of $C^k$ spaces. We refer the reader to lemma 3.5 and corollary 3.6 in \cite{GR23} for complete proofs of very similar results. The only difference here is that, unlike \cite{GR23} where only one parameter is used in controlling the number of controlled derivatives, here we use three parameters $L_\theta$, $L_R$, $L_t$. It is not difficult to see that the estimates below can be obtained by following the proof of the corresponding results in \cite{GR23}, while keeping in mind that $L_t \leq L_R \leq L_\theta -1$.
\begin{lem} \label{a_estim}
With the same assumptions of proposition \ref{NewIter}, the estimates 
\begin{equation} \label{a_estim_1}
    \|a_{\xi, k, n}\|_N \lesssim \delta_{q+1, n}^{1/2} \la_{q+1}^{1/2} \lambda_q^N, \,\,\, \|A_{\xi, k, n}\|_N \lesssim \delta_{q+1, n} \lambda_q^N, \,\,\, \forall N \in \{0,1,..., L_t\}
\end{equation}
\begin{equation} \label{a_estim_2}
     \|\bar D_t a_{\xi, k, n}\|_N  \lesssim \delta_{q+1, n}^{1/2} \la_{q+1}^{1/2} \tau_q^{-1} \lambda_q^N, \,\,\,\|\bar D_t A_{\xi, k, n} \|_N  \lesssim \delta_{q+1, n} \tau_q^{-1} \lambda_q^N, \,\,\,  \forall N \in \{0,1,..., L_t\}
\end{equation}
\begin{equation} \label{a_estim_3}
    \|a_{\xi, k, n}\|_{N+L_t} \lesssim \delta_{q+1, n}^{1/2} \la_{q+1}^{1/2} \lambda_q^{L_t} \ell_q^{-N}, \,\,\,\|A_{\xi, k, n}\|_{N+L_t} \lesssim \delta_{q+1, n} \lambda_q^{L_t} \ell_q^{-N},  \,\,\,  \forall N \geq 0
\end{equation}
\begin{equation} \label{a_estim_4}
    \|\bar D_t a_{\xi, k, n}\|_{N+L_t} \lesssim \delta_{q+1, n}^{1/2} \la_{q+1}^{1/2} \lambda_q^{L_t}  \tau_q^{-1} \ell_q^{-N}, \,\,\,\|\bar D_t A_{\xi, k, n} \|_{N+ L_t} \lesssim \delta_{q+1, n} \lambda_q^{L_t} \tau_q^{-1} \ell_q^{-N}, \,\,\, \forall N \geq 0,
\end{equation}
hold true where the implicit constants depend on $n$, $\Gamma$, $M$, $\alpha$, and $N$.
\end{lem}

We are now in a position to prove the main technical lemma concerning the Newton perturbations. 

\begin{lem} \label{psi_estim}
   With the same assumptions of proposition \ref{NewIter}, we have the following estimates on $\supp \tilde \chi_k$:
   \begin{equation} \label{psi_estim_1}
       \|\psi_{k, n+1}\|_{N+\alpha} \lesssim \frac{\delta_{q+1, n} \lambda_q^{N} \ell_q^{-\alpha}}{\mu_{q+1}}, \,\,\, \forall N \in \{0,1,..., L_t\},
   \end{equation}
   \begin{equation} \label{psi_estim_2}
        \|\bar D_t \psi_{k, n+1}\|_{N+\alpha} \lesssim  \delta_{q+1, n} \lambda_q^N \ell_q^{-\alpha}, \,\,\, \forall N \in \{0, 1,..., L_t\},
   \end{equation}
   \begin{equation} \label{psi_estim_3}
       \|\psi_{k, n+1}\|_{N + L_t + \alpha} \lesssim \frac{\delta_{q+1, n} \lambda_q^{L_t} \ell_q^{-N-\alpha}}{\mu_{q+1}},  \,\,\, \forall N \geq 0,
   \end{equation}
   \begin{equation} \label{psi_estim_4}
       \|\bar D_t \psi_{k, n+1}\|_{N+ L_t + \alpha} \lesssim \delta_{q+1, n}  \lambda_q^{L_t} \ell_q^{-N-\alpha}, \,\,\, \forall N \geq 0;
   \end{equation}
   and on $\supp \partial_t \tilde \chi_k$:
   \begin{equation} \label{psi_estim_5}
       \|\bar D_t \psi_{k, n+1}\|_{N+\alpha} \lesssim \frac{\delta_{q+1, n} \lambda_q^N \ell_q^{-\alpha}}{\mu_{q+1}\tau_q}  , \,\,\, \forall N \in \{0, 1,..., L_t\},
   \end{equation}
   \begin{equation} \label{psi_estim_6}
       \|\bar D_t \psi_{k, n+1}\|_{N+L_t + \alpha} \lesssim \frac{\delta_{q+1, n}  \lambda_q^{L_t} \ell_q^{-N-\alpha}}{\mu_{q+1} \tau_q}, \,\,\, \forall N \geq 0,
   \end{equation}
  where all the implicit constants depend on $n$, $\Gamma$, $M$, $\alpha$, and $N$.
\end{lem}
\begin{proof}
    We first decompose $\psi_{k, n+1}$ into 
    \begin{equation} \label{psidecomp}
        \psi_{k, n+1}  = \tilde \psi + \tilde \Xi + \Xi,
    \end{equation}
    where $\tilde \psi$ is the unique solution to the transport equation
    \begin{equation*}
        \begin{cases}
            \bar D_t \tilde \psi = - T[\psi_{k,n+1}] \cdot \nabla \bar\theta_q -\Delta^{-1} \div([\bar u_q,\Delta]\psi_{k, n + 1} + [\bar \theta_q,\Delta]T[\psi_{k,n+1}]) \\
            \tilde \psi \big|_{t = t_k} = 0,
        \end{cases}
    \end{equation*}
    $\tilde \Xi$ solves
    \begin{equation*}
    \begin{cases}
        \bar D_t \tilde \Xi = - \frac{1}{\mu_{q+1}}  \sum_{\xi \in F} f^{[1]}_{\xi, k, n}(\mu_{q+1} \cdot ) \bar D_t \Delta^{-1} \nabla^\perp \cdot \div A_{\xi, k, n} \\ 
        \tilde \Xi \big |_{t = t_k} = 0,
        \end{cases}
    \end{equation*}
    and 
    \begin{equation*}
        \Xi = \frac{1}{\mu_{q+1}} \sum_{\xi \in F} f_{\xi, k, n}^{[1]}(\mu_{q+1} \cdot) \Delta^{-1} \nabla^\perp \cdot \div A_{\xi, k, n}. 
    \end{equation*}

    \textit{Estimates for $\tilde \psi$ when $N=0$.} Since $T$ is an operator of zero order, we have 
    \begin{equation*}
        \|T[\psi_{k,n+1}] \cdot \nabla \bar\theta_q\|_\alpha \lesssim \|\psi\|_\alpha \|\bar \theta_q\|_{1 + \alpha},
    \end{equation*}
    and similarly we estimate
    \begin{equation*}
        \|\Delta^{-1} \div(\pa_j(\pa_j \bar u_q \psi_{k, n + 1}) + \pa_j (\pa_j \bar \theta_q T[\psi_{k, n + 1}]))\|_\alpha \lesssim \|\psi_{k,n+1}\|_\alpha (\|\bar u_q\|_{1+\alpha} + \|\bar \theta_q\|_{1+\alpha})\,.
    \end{equation*}
    Furthermore, lemma \ref{le-bilinear-S} yields
    \begin{equation*}
        \|S[\bar \theta_q,\psi_{k, n+1}]\|_\alpha \lesssim \|\bar \theta_q\|_{1+\alpha}\|\psi_{k, n+1}\|_\alpha\,,
    \end{equation*}
    and, thus, it follows from \eqref{eqn.simpl} that
    \begin{equation*}
        \| T[\psi_{k,n+1}] \cdot \nabla \bar\theta_q +\Delta^{-1} \div([\bar u_q,\Delta]\psi_{k, n + 1} + [\bar \theta_q,\Delta]T[\psi_{k,n+1}])\|_\alpha \lesssim \|\psi_{k, n+1}\|_\alpha \|\bar \theta_q\|_{1 + \alpha}.
    \end{equation*}
    Consequently, proposition \ref{transport_estim} implies
    \begin{equation*}
        \|\tilde \psi(\cdot, t)\|_\alpha \lesssim \tau_q^{-1} \int_{t_k}^t \|\psi_{k, n+1}(\cdot, s) \|_\alpha ds.
    \end{equation*}

    \textit{Estimates for $\tilde \Xi$ when $N=0$.} Using the commutator estimate of proposition \ref{prop-commu}, we have 
    \begin{eqnarray*}
        \|\bar D_t \tilde \Xi\|_{\alpha} &\lesssim & \frac{1}{\mu_{q+1}}\sup_\xi \|\bar D_t A_{\xi, k, n}\|_{\alpha} + \frac{1}{\mu_{q+1}}\sup_\xi\|[\bar u_q \cdot \nabla, \Delta^{-1}\nabla^\perp \div] A_{\xi, k, n}\|_{\alpha} \\ 
        & \lesssim & \frac{1}{\mu_{q+1}}\sup_\xi\|\bar D_t A_{\xi, k, n}\|_{\alpha} + \frac{1}{\mu_{q+1}} \|\bar u_q\|_{1+\alpha} \sup_\xi\|A_{\xi, k, n}\|_\alpha \lesssim  \frac{\delta_{q+1, n}  \lambda_q^\alpha}{\mu_{q+1} \tau_q}, 
    \end{eqnarray*}
    where interpolation from the estimates of lemma \ref{a_estim} was used in the last inequality. It follows again from  proposition \ref{transport_estim} that
    \begin{equation*}
        \|\tilde \Xi\|_\alpha \lesssim \frac{1}{\mu_{q+1}} \delta_{q+1, n} \lambda_q^{\alpha}, \ \ \mbox{on}\ \ \supp \tilde \chi_k.
    \end{equation*}
    
    \textit{Estimates for $\Xi$ when $N=0$.} Finally, 
    \begin{equation*}
        \|\Xi\|_{\alpha} \lesssim \frac{1}{\mu_{q+1}} \sup_\xi\|A_{\xi, k, n}\|_{\alpha} \lesssim \frac{\delta_{q+1, n} \lambda_q^{\alpha}}{\mu_{q+1}}.
    \end{equation*}

    Going back to \eqref{psidecomp}, we obtain
    \begin{equation*}
        \|\psi_{k, n+1}(\cdot, t)\|_\alpha \lesssim \frac{\delta_{q+1, n} \lambda_q^{\alpha}}{\mu_{q+1}} + \tau_q^{-1} \int_{t_k}^t \|\psi_{k, n+1}(\cdot, s)\|_\alpha ds. 
    \end{equation*}
    It thus follows from Gr\"onwall's inequality that
    \begin{equation*}
        \|\psi_{k, n+1}\|_\alpha \lesssim \frac{\delta_{q+1, n} \lambda_q^{\alpha}}{\mu_{q+1}}, \ \ \mbox{on}\ \ \supp \tilde \chi_k.
    \end{equation*}

    \textit{Estimates for $\tilde \psi$ when $N\geq1$.} For a multi-index $\sigma$ with $|\sigma| = N$, we have 
    \begin{align*}
        \|\bar D_t \partial^\sigma \tilde \psi\|_\alpha \lesssim \|\partial^\sigma \bar D_t  \tilde \psi\|_\alpha + \|[\bar u_q \cdot \nabla, \partial^\sigma] \tilde \psi\|_\alpha 
    \end{align*}
    On the one hand, applying lemma~\ref{le-bilinear-S}, we have
    \begin{align*}
        \|\partial^\sigma \bar D_t  \tilde \psi\|_\alpha &\lesssim \|\partial^\sigma(T[\psi_{k, n+1}]\cdot \nabla \bar \theta_q)\|_\alpha + \|\partial^\sigma(\Delta^{-1} \div(\pa_j(\pa_j \bar u_q \psi_{k, n + 1}) + \pa_j (\pa_j \bar \theta_q T[\psi_{k, n + 1}])))\|_\alpha \\
        &\qquad + \|\partial^\sigma S[\bar \theta_q, \psi_{k,n+1}]\|_\alpha \\
        &\lesssim (\|\bar \theta_q\|_{1+\alpha} + \|\bar u_q\|_{1+\alpha}) \|\psi_{k, n+1}\|_{N+\alpha} + (\|\bar \theta_q\|_{N+ 1+\alpha} + \|\bar u_q\|_{N+ 1+\alpha}) \|\psi_{k, n+1}\|_{\alpha} \\
        &\qquad + \|S[\bar \theta_q, \psi_{k,n+1}]\|_{N+\alpha} \\
        &\lesssim (\|\bar \theta_q\|_{1+\alpha} + \|\bar u_q\|_{1+\alpha}) \|\psi_{k, n+1}\|_{N+\alpha} + (\|\bar \theta_q\|_{N+ 1+\alpha} + \|\bar u_q\|_{N+ 1+\alpha}) \|\psi_{k, n+1}\|_{\alpha} \,,
    \end{align*}
    while on the other hand, using interpolation and Young's inequality for products, we deduce
    \begin{eqnarray*}
        \|[\bar u_q \cdot \nabla, \partial^\sigma] \tilde \psi\|_\alpha &\lesssim& \|\bar u_q\|_{N+\alpha}\|\tilde \psi\|_{1+\alpha} + \|\bar u_q\|_{1+\alpha}\|\tilde \psi\|_{N+\alpha} \\ 
        &\lesssim & \|\bar u_q\|_{1+\alpha}\|\tilde \psi\|_{N+\alpha} + \|\bar u_q\|_{N+ 1+\alpha}\|\tilde \psi\|_{\alpha}.
    \end{eqnarray*}
    Again, proposition \ref{transport_estim} implies
    \begin{eqnarray*}
        \|\tilde \psi (\cdot, t)\|_{N+ \alpha} \lesssim && \frac{\delta_{q+1, n} \lambda_q^{\alpha}\tau_q }{\mu_{q+1}} (\|\bar u_q\|_{N+1+\alpha} + \|\bar \theta_q\|_{N+1+\alpha}) + (\|\bar u_q\|_{1+\alpha}+\|\bar \theta_q\|_{1+\alpha}) \int_{t_k}^t \|\psi_{k, n+1}(\cdot, s) \|_{N+\alpha}ds \\ 
        && + \|\bar u_q\|_{1+\alpha} \int_{t_k}^t \|\tilde \psi(\cdot, s) \|_{N+\alpha}ds.
    \end{eqnarray*}
    Hence, an application of Gr\"onwall's inequality gives 
    \begin{equation*}
        \|\tilde \psi (\cdot, t)\|_{N+ \alpha} \lesssim \frac{\delta_{q+1, n} \lambda_q^{\alpha} \tau_q}{\mu_{q+1}} (\|\bar u_q\|_{N+1+\alpha} + \|\bar \theta_q\|_{N+1+\alpha})  + \tau_q^{-1} \int_{t_k}^t \|\psi_{k, n+1}(\cdot, s) \|_{N+\alpha}ds.
    \end{equation*}

    \textit{Estimates for $\tilde \Xi$ when $N \geq 1$.} Using  proposition \ref{prop-commu} and the commutator estimate for $[\bar u_q \cdot \nabla, \partial^\sigma]$, we have
    \begin{eqnarray*}
        \|\bar D_t \partial^\sigma \tilde \Xi\|_\alpha &\lesssim& \|\bar D_t \tilde \Xi\|_{N+\alpha} + \|[\bar u_q \cdot \nabla, \partial^\sigma] \tilde \Xi\|_\alpha  \\ 
        &\lesssim & \frac{1}{\mu_{q+1}}\sup_\xi \big( \|\bar D_t A_{\xi, k, n}\|_{N+\alpha} + \|[\bar u_q \cdot \nabla, \Delta^{-1} \nabla^\perp \div] A_{\xi, k, n}\|_{N+\alpha} \big) + \|[\bar u_q \cdot \nabla, \partial^\sigma] \tilde \Xi\|_\alpha \\ 
         &\lesssim & \frac{1}{\mu_{q+1}}\sup_\xi \big( \|\bar D_t A_{\xi, k, n}\|_{N+\alpha} + \|\bar u_q\|_{1+\alpha} \|A_{\xi, k, n}\|_{N+\alpha} + \|\bar u_q\|_{N+1 + \alpha} \|A_{\xi, k, n}\|_\alpha \big) \\ 
         && + \|\bar u_q\|_{N+1+\alpha} \|\tilde \Xi\|_\alpha + \|\bar u_q\|_{1+\alpha} \|\tilde \Xi\|_{N+ \alpha}.
    \end{eqnarray*}
    A similar application of proposition \ref{transport_estim} and Gr\"onwall's inequality yields
    \begin{equation*}
        \|\tilde \Xi\|_{N+\alpha} \lesssim \frac{1}{\mu_{q+1}} \sup_\xi \big( \tau_q \|\bar D_t A_{\xi, k, n} \|_{N+\alpha} + \|A_{\xi, k, n}\|_{N+\alpha}\big) + \frac{\delta_{q+1, n} \lambda_q^\alpha \tau_q}{\mu_{q+1}} \|\bar u_q\|_{N+1+\alpha}.
    \end{equation*}

    \textit{Estimates for $\Xi$ when $N \geq 1$.} Finally,
    \begin{equation*}
        \|\Xi\|_{N+\alpha} \lesssim \frac{1}{\mu_{q+1}} \sup_\xi \|A_{\xi, k, n}\|_{N+\alpha}. 
    \end{equation*}

    Therefore, we conclude 
    \begin{eqnarray*}
        \|\psi_{k, n+1}(\cdot, t)\|_{N+\alpha} & \lesssim &  \frac{1}{\mu_{q+1}}\sup_\xi \big( \tau_q \|\bar D_t A_{\xi, k, n}\|_{N+\alpha} + \|A_{\xi, k, n}\|_{N+\alpha} + \delta_{q+1, n} \lambda_q^\alpha \tau_q \|\bar u_q\|_{N+1+\alpha} \big) \\ 
        && + \tau_q^{-1} \int_{t_k}^t \|\psi_{k, n+1}(\cdot, s)\|_{N+ \alpha} ds,
    \end{eqnarray*}
    and once again, by Gr\"onwall's inequality, 
    \begin{equation*}
        \|\psi_{k, n+1}\|_{N+\alpha} \lesssim \frac{1}{\mu_{q+1}}\sup_\xi\big( \tau_q \|\bar D_t A_{\xi, k, n}\|_{N+\alpha} + \|A_{\xi, k, n}\|_{N+\alpha} + \delta_{q+1, n} \lambda_q^\alpha \tau_q \|\bar u_q\|_{N+1+\alpha} \big).
    \end{equation*}
    The estimates \eqref{psi_estim_1} and \eqref{psi_estim_3} then follow from the last estimate together with lemma \ref{smoli_1} and lemma \ref{a_estim}. Here, we use the relation $L_t \leq L_R \leq L_{\theta} - 1$.

    Regarding the material derivative estimates, it follows from the equation \eqref{psi_eqn} that
    \begin{eqnarray*}
        \|\bar D_t \psi_{k, n+1}(\cdot, t)\|_{N+\alpha}  &\lesssim & \|T[\psi_{k, n+1}]\cdot \nabla \bar \theta_q\|_{N+\alpha} + \sup_\xi\|A_{\xi, k, n}(\cdot, t)\|_{N+\alpha} \\ 
        & &+ \|\Delta^{-1} \div([\bar u_q,\Delta]\psi_{k, n + 1} + [\bar \theta_q,\Delta]T[\psi_{k,n+1}])\|_{N+\alpha} \\
         &\lesssim & \|\psi_{k, n+1}\|_{N+\alpha}( \| \bar u_q\|_{1+\alpha} + \|\bar \theta_q\|_{1+\alpha}) + \|\psi_{k, n+1}\|_{\alpha} (\|\bar u_q\|_{N+1+\alpha} + \|\bar \theta_q\|_{N+1+\alpha})\\
         && + \sup_\xi \|A_{\xi, k, n}(\cdot, t)\|_{N+\alpha}.
    \end{eqnarray*}
    Thus, the estimates \eqref{psi_estim_1} and \eqref{psi_estim_3} together with lemma \ref{a_estim} imply \eqref{psi_estim_2} and \eqref{psi_estim_4}. Moreover, since $A_{\xi, k, n} = 0$ on $\supp \partial_t \tilde \chi_k$, the final term in the estimate above does not appear, and \eqref{psi_estim_5} and \eqref{psi_estim_6} follow.
\end{proof}

\begin{proof} [Proof of proposition \ref{NewIter}]
Since the set $\{\supp \tilde \chi_k\}$ is locally finite and $\div^{-1} \nabla^\perp$ is of Calder\'on-Zygmund type, we conclude from \eqref{R-q-potential} that
\begin{equation*}
    \|R_{q, n+1}\|_N \lesssim \|R_{q, n+1}\|_{N+\alpha} \lesssim \tau_q^{-1} \sup_{k \in \mathbb Z_{q,n}} \|\psi_{k, n+1}\|_{N+\alpha}.
\end{equation*}
It follows from lemma \ref{psi_estim} that 
\begin{equation*}
    \|R_{q, n+1}\|_N \leq C \delta_{q+1, n}\bigg(\frac{\lambda_q}{\lambda_{q+1}}\bigg)^{\frac12 - \beta} (\lambda_{q+1} \ell_q)^{-\alpha} \lambda_{q+1}^{-2\alpha} \lambda_q^N \leq (C \lambda_{q+1}^{-\alpha })\delta_{q+1, n+1} \lambda_q^{-\alpha} \lambda_q^N, \ \ N \in \{0,1,..., L_t\}
\end{equation*}
for a constant $C>0$ independent of $a > a_0$ and $q$.  
As before one can choose $a_0$ sufficiently large such that 
\begin{equation*}
    C \lambda_{q+1}^{-\alpha} \leq 1, \ \ \forall \alpha > 0
\end{equation*}
in order for \eqref{NewIter_1} to hold. Analogously, we obtain \eqref{NewIter_3} from lemma \ref{psi_estim}.

Furthermore, applying proposition \ref{prop-commu} we have
\begin{eqnarray*}
    \|\bar D_t R_{q, n+1}\|_{N+\alpha} &\lesssim& \sup_{k \in \mathbb Z_{q,n}} \big( \|\bar D_t(\partial_t \tilde \chi_k \psi_{k, n+1})\|_{N+\alpha} + \|[\bar u_q \cdot \nabla, \div^{-1} \nabla^\perp] \partial_t \tilde \chi_k \psi_{k, n+1}\|_{N+\alpha} \big) \\ 
    &\lesssim & \sup_{k \in \mathbb Z_{q,n}} \big( \tau_q^{-2} \|\psi_{k, n+1}\|_{N+\alpha} + \tau_q^{-1} \|\bar D_t \psi_{k, n+1}\|_{N+\alpha, \,\, \supp \partial_t \tilde \chi_k} \\ 
    && + \tau_q^{-1} \|\bar u_q\|_{1+\alpha}\|\psi_{k, n+1}\|_{N+\alpha} + \tau_q^{-1} \|\bar u_q\|_{N+1+\alpha} \|\psi_{k,n+1}\|_\alpha \big).
\end{eqnarray*}
 It thus follows from lemmas \ref{smoli_estim} and \ref{psi_estim}, and the choice of a sufficiently large $a_0>0$ that
\begin{equation*}
    \|\bar D_t R_{q, n+1}\|_{N} \leq C \tau_q^{-1} \delta_{q+1, n} \bigg(\frac{\lambda_q}{\lambda_{q+1}}\bigg)^{\frac12 - \beta} (\lambda_{q+1} \ell_q)^{-\alpha} \lambda_{q+1}^{-2\alpha} \lambda_q^N\leq \delta_{q+1, n+1} \tau_q^{-1} \lambda_q^{N - \alpha}, \ \ N\geq \{0,1,...,L_t\},
\end{equation*}
and the estimates propagate.
\end{proof}

\subsection{Estimates for the total Newton perturbation}

We now obtain estimates for the total Newton perturbation: 
\begin{equation*}
    \theta_{q+1}^{(t)} = \sum_{n = 1}^\Gamma \theta_{q+1, n}^{(t)}\,,\qquad w_{q+1}^{(t)} = T[\theta_{q+1}^{(t)}]\,.
\end{equation*}
In light of proposition \ref{NewIter} and lemma \ref{psi_estim}, we have the following estimates.

\begin{lem} \label{w_t_estim}
    The estimates: 
    \begin{equation} \label{w_t_estim_0}
        \|\Lambda^{-1}\theta_{q+1}^{(t)}\|_0 \lesssim \frac{\delta_{q+1} \lambda_q \ell_q^{-\alpha}}{\mu_{q+1}},
    \end{equation}
    \begin{equation} \label{w_t_estim_1}
        \|\theta_{q+1}^{(t)}\|_N \lesssim \frac{\delta_{q+1} \lambda_q^{N + 2} \ell_q^{-\alpha}}{\mu_{q+1}}, \,\|w_{q+1}^{(t)}\|_N \lesssim \frac{\delta_{q+1} \lambda_q^{N + 2} \ell_q^{-\alpha}}{\mu_{q+1}}, \,\,\, \forall N \in \{0,1,...,L_t-2\},
    \end{equation}
    \begin{equation} \label{w_t_estim_2}
        \|\bar D_t \theta_{q+1}^{(t)}\|_N \lesssim \delta_{q+1} \lambda_q^{N+2} \ell_q^{-\alpha}, \, \|\bar D_t w_{q+1}^{(t)}\|_{N} \lesssim \delta_{q+1} \lambda_q^{N+2} \ell_q^{-\alpha}, \,\,\, \forall N \in  \{0,1,...,L_t-2\},
    \end{equation}
    \begin{equation} \label{w_t_estim_3}
        \|\theta_{q+1}^{(t)}\|_{N+L_t-2} \lesssim \frac{\delta_{q+1} \lambda_q^{L_t} \ell_q^{-N-\alpha}}{\mu_{q+1}}, \, \|w_{q+1}^{(t)}\|_{N+L_t-2 } \lesssim \frac{\delta_{q+1} \lambda_q^{L_t} \ell_q^{-N-\alpha}}{\mu_{q+1}}, \,\,\, \forall N \geq 0,
    \end{equation}
    \begin{equation} \label{w_t_estim_4}
        \|\bar D_t \theta_{q+1}^{(t)}\|_{N+L_t-2} \lesssim \delta_{q+1} \lambda_q^{L_t} \ell_q^{-N - \alpha}, \, \|\bar D_t w_{q+1}^{(t)}\|_{N+L_t-2} \lesssim \delta_{q+1} \lambda_q^{L_t} \ell_q^{-N - \alpha}, \,\,\, \forall N \geq 0,
    \end{equation}
    hold true for implicit constants which depend on $\Gamma$, $M$, $\alpha$ and $N$. 
    In addition, the temporal support satisfies 
    \begin{eqnarray}\label{temporal-qplus1}
        \supp_t \theta_{q+1}^{(t)} \cup \supp_t w_{q+1}^{(t)} &\subset& [-2 + (\delta_q^{\frac12} \lambda_q^{\frac{3}2})^{-1} - 2\Gamma \tau_q, -1 - (\delta_q^{\frac12} \lambda_q^{\frac{3}2})^{-1} + 2\Gamma \tau_q] \\ 
        && \cup [1 + (\delta_q^{\frac12} \lambda_q^{\frac{3}2})^{-1} - 2 \Gamma \tau_q, 2 - (\delta_q^{\frac12} \lambda_q^{\frac{3}2})^{-1} + 2 \Gamma \tau_q]. \nonumber
    \end{eqnarray}
\end{lem}

\begin{proof}
Recalling that $\delta_{q+1, n} \leq \delta_{q+1}$ for all $n$, we only need to analyze $\theta_{q+1, n+1}^{(t)}$ and $w_{q+1, n+1}^{(t)}$. Since $\Lambda^{-1} \nabla$ is an operator of zero order, the estimate \eqref{w_t_estim_0} follows from
    \begin{equation*}
        \|\Lambda^{-1} \theta_{q+1, n+1}^{(t)}\|_0 \lesssim \sup_{k \in \mathbb Z_{q,n}} \|\div (\Lambda^{-1} \nabla \psi_{k, n+1})\|_\alpha \lesssim \sup_{k \in \mathbb Z_{q,n}} \|\psi_{k,n+1}\|_{1+\alpha}
    \end{equation*}
   and lemma \ref{psi_estim}.
    
   On the other hand, \eqref{w_t_estim_1} and \eqref{w_t_estim_3} also follow from  
    \begin{equation*}
        \|\theta_{q+1, n+1}^{(t)}\|_N \lesssim \sup_{k \in \mathbb Z_{q,n}} \|\psi_{k, n+1}\|_{N+2+\alpha},\qquad \|w_{q+1, n+1}^{(t)}\|_N \lesssim \sup_{k \in \mathbb Z_{q,n}} \|\psi_{k, n+1}\|_{N+2+\alpha},  \ \ N \geq 0
    \end{equation*}
    and lemma \ref{psi_estim}. 

    Regarding the material derivative estimates, we write 
    \begin{equation*}
        \bar D_t \theta_{q+1, n+1}^{(t)} = \sum_{k \in \mathbb Z_{q,n}} \big( \partial_t \tilde \chi_k \Delta \psi_{k, n+1} + \tilde \chi_k \Delta \bar D_t \psi_{k, n+1} + \tilde \chi_k [\bar u_q\cdot,\Delta]\nabla \psi_{k, n+1} \big).
    \end{equation*}
   Thus we have
    \begin{eqnarray*}
        \|\bar D_t \theta_{q+1, n+1}^{(t)}\|_{N+\alpha} \lesssim && \sup_{k \in \mathbb Z_{q,n}} \big( \tau_q^{-1} \|\psi_{k, n+1}\|_{N+2+\alpha} + \|\bar D_t \psi_{k, n+1}\|_{N + 2 + \alpha} \\ 
        && + \|\bar u_q\|_{N+2+\alpha} \|\psi_{k, n+1}\|_{1+\alpha} + \|\bar u_q\|_{1+\alpha} \|\psi_{k, n+1}\|_{N+2 + \alpha} \big),
    \end{eqnarray*}
    with the second term on the right-hand-side being the dominant one. 
    The estimates in \eqref{w_t_estim_2} and \eqref{w_t_estim_4} for $\theta_{q+1}^{(t)}$ follow from the estimates of lemmas \ref{smoli_estim} and \ref{psi_estim}.
    For the remaining estimates in \eqref{w_t_estim_2} and \eqref{w_t_estim_4}, we have 
    \begin{equation*}
        \|\bar D_t w_{q+1}^{(t)}\|_{N+\alpha} \lesssim \|\bar D_t \theta_{q+1}^{(t)}\|_{N + \alpha} + \|\bar u_q\|_{N+1+\alpha} \|\theta_{q+1}^{(t)}\|_\alpha + \|\bar u_q\|_{1+\alpha} \|\theta_{q+1}^{(t)}\|_{N+\alpha},
    \end{equation*}
    where the first term is the dominant one.

    The claimed property on the temporal support \eqref{temporal-qplus1} follows straightforwardly from the definition of $\theta_{q+1}^{(t)}$ and proposition \ref{NewIter}.
\end{proof}

In the construction of the Nash perturbation, we will make use of the the backward flow of the velocity field 
\begin{equation*}
    \tilde u_{q, \Gamma} := \bar u_q + P_{\lesssim \ell_q^{-1}} w_{q+1}^{(t)} = P_{\lesssim \ell_q^{-1}}(u_q + w_{q+1}^{(t)}) = P_{\lesssim \ell_q^{-1}} u_{q, \Gamma}.
\end{equation*}
We define also 
\begin{equation*}
    \tilde \theta_{q, \Gamma} := \bar \theta_q + P_{\lesssim \ell_q^{-1}} \theta_{q+1}^{(t)} = P_{\lesssim \ell_q^{-1}}(\theta_q + \theta_{q+1}^{(t)}) = P_{\lesssim \ell_q^{-1}}\theta_{q, \Gamma}.
\end{equation*}
We remark that this is a point of departure from the corresponding definition in \cite{GR23}, where the extra mollification operator $P_{\lesssim \ell_q^{-1}}$ is omitted. It is preferable from the point of view of the estimates of the following section to have good control over the localization in Fourier space -- this justifies the definition. 

As a direct consequence of lemma~\ref{w_t_estim} we have the following corollary. We note that $L_t - 2 \leq L_\theta$.

\begin{cor} \label{Gamma_velo_estim}
    The estimates: 
    \begin{equation}\notag
        \|\tilde \theta_{q, \Gamma} \|_N +\|\tilde u_{q, \Gamma} \|_N \lesssim \delta_{q}^{\frac12} \lambda_q^{\frac{1}2+N}, \,\,\, \forall N \in \{1,2,...,L_t-2\},
    \end{equation}
    \begin{equation}\notag
        \|\tilde \theta_{q, \Gamma}\|_{N + L_t-2} + \|\tilde u_{q, \Gamma}\|_{N + L_t-2} \lesssim \delta_{q}^{\frac12} \lambda_q^{\frac{1}2+L_t-2} \ell_q^{-N}, \,\,\, \forall N \geq 0,
    \end{equation}
    are satisfied with implicit constants depending on $\Gamma$, $M$, $\alpha$, and $N$.
\end{cor}

    


\subsection{The perturbed flow} \label{Pert.Flow.Sec} 
Let $\tilde \Phi_t$ be the backward flow generated by $\tilde u_{q, \Gamma}$:
\begin{equation} \label{Flow_t_gam}
    \begin{cases}
        \partial_s \tilde \Phi_t(x,s) + \tilde u_{q, \Gamma}(x,s) \cdot \nabla \tilde \Phi_t (x,s) = 0, \\ 
        \tilde \Phi_t(x, t) = x.
    \end{cases}
\end{equation}
The corresponding Lagraingian flow $\tilde X_t$ is given by
\begin{equation} \label{Lagr_t_gam}
    \begin{cases}
        \frac{d}{ds} \tilde X_t(\alpha, s) = \tilde u_{q, \Gamma}(\tilde X_t(\alpha, s),s), \\ 
        \tilde X_t(\alpha, t) = \alpha.
    \end{cases}
\end{equation}


Consistent with previous notation, $\tilde D_{t, \Gamma}$ is used for the material derivative 
\begin{equation*}
    \tilde D_{t, \Gamma} = \partial_t + \tilde u_{q, \Gamma} \cdot \nabla.
\end{equation*}
\begin{cor} \label{Flow_gam_estim}
Let $\tilde \Phi_t$ and $ \tilde X_t$ be defined by \eqref{Flow_t_gam} and \eqref{Lagr_t_gam} for $t \in \mathbb R$, respectively. We have for any $|s - t| < \tau \leq  \|\tilde u_{q, \Gamma}\|_1^{-1}$, 
    \begin{equation} \label{Flow_gam_estim_1}
        \|(\nabla \tilde \Phi_t)^{-1}(\cdot, s)\|_N + \|\nabla \tilde \Phi_t (\cdot, s)\|_N \lesssim \lambda_q^N, \,\,\, \forall N \in\{0,1,..., L_t-3\},
    \end{equation}
    \begin{equation} \label{Flow_gam_estim_2}
        \|\tilde D_{t, \Gamma} (\nabla \tilde \Phi_t)^{-1}(\cdot, s)\|_N + \|\tilde D_{t, \Gamma} \nabla \tilde \Phi_t (\cdot, s)\|_N \lesssim \delta_q^{\frac12} \lambda_q^{\frac{3}2+N}, \,\,\, \forall N \in\{0,1,..., L_t-3\},
    \end{equation}
    \begin{equation} \label{Flow_gam_estim_3}
         \|D \tilde X_t(\cdot, s)\|_N  \lesssim \lambda_q^N, \,\,\, \forall N \in \{0,1,..., L_t-3\},
    \end{equation}
    \begin{equation} \label{Flow_gam_estim_4}
        \|(\nabla \tilde \Phi_t)^{-1}(\cdot, s)\|_{N+L_t-3} + \|\nabla \tilde \Phi_t(\cdot, s)\|_{N+L_t-3}  \lesssim \lambda_q^{L_t-3} \ell_q^{-N}, \, \, \, \forall N \geq 0,
    \end{equation}
    \begin{equation} \label{Flow_gam_estim_5}
        \|\tilde D_{t, \Gamma} (\nabla \tilde \Phi_t)^{-1}(\cdot, s)\|_{N+L_t-3} + \|\tilde D_{t, \Gamma} \nabla \tilde \Phi_t(\cdot, s)\|_{N+L_t-3}  \lesssim \delta_q^{\frac12} \lambda_q^{\frac{3}2+L_t-3} \ell_q^{-N}, \, \, \, \forall N \geq 0,
    \end{equation}
    \begin{equation} \label{Flow_gam_estim_6}
        \|D \tilde X_t(\cdot, s)\|_{N+L_t-3} \lesssim \lambda_q^{L_t-3} \ell_q^{-N}, \, \, \, \forall N \geq 0,
    \end{equation}
   with the implicit constants depending only on $\Gamma$, $M$, $\alpha$, and $N$.   
\end{cor}
The proof is the same as that of lemma \ref{Flow_estim}; see, for instance, lemma 3.2 in \cite{GR23}.

\begin{rem} \label{remark_flow_bd}
Corollary \ref{Flow_gam_estim} holds with $\tau = \tau_q$, since
\begin{equation*}
    \tau_q \|\tilde u_{q, \Gamma}\|_1 \lesssim \lambda_{q+1}^{-\alpha}.
\end{equation*}
Moreover, proposition \ref{transport_estim} shows that
\begin{equation*}
    \|\I - \nabla \tilde \Phi\|_0 \lesssim \lambda_{q+1}^{-\alpha}, \ \ \mbox{for} \ \tau = \tau_q.
\end{equation*}
Therefore, $\|\nabla \tilde \Phi\|_0$ has an upper bound which is independent of the parameters in the construction, since for any constant $C > 0$ independent of $a > a_0$ and $q$, $a_0$ can be chosen sufficiently large so that 
\begin{equation*}
    C \lambda_{q+1}^{- \alpha} \leq 1.
\end{equation*}
\end{rem}

We conclude this section with a simple stability estimate on the perturbed flow, which will be used in the following section.

\begin{lem} \label{flow_stabil}
    For $t \in \mathbb R$, let $\tilde \Phi_t$ and $\Phi_t$
    be, respectively, the backward flows of $\tilde u_{q, \Gamma}$ and $\bar u_q$, as defined in \eqref{Flow_t_gam} and \eqref{Flow_t}. If $|s-t| < \tau \leq (\|\tilde u_{q, \Gamma}\|_1 + \|\bar u_q\|_1)^{-1}$ and $N \in \{0,1,...,L_t-4\}$, 
    \begin{equation}
        \|\nabla \Phi_t(\cdot, s) - \nabla \tilde \Phi_t(\cdot, s)\|_N + \|(\nabla \Phi_t(\cdot, s))^{-1} - (\nabla \tilde \Phi_t(\cdot, s))^{-1}\|_N \lesssim \tau \frac{\delta_{q+1} \lambda_q^{3} \ell_q^{-\alpha}}{\mu_{q+1}}\lambda_q^{N}, 
    \end{equation}
    while if $N \geq 0$,
    \begin{equation}
        \|\nabla \Phi_t(\cdot, s) - \nabla \tilde \Phi_t(\cdot, s)\|_{N + L_t-4} + \|(\nabla \Phi_t(\cdot, s))^{-1} - (\nabla \tilde \Phi_t(\cdot, s))^{-1}\|_{N+ L_t-4} \lesssim \tau \frac{\delta_{q+1} \lambda_q^{3} \ell_q^{-\alpha}}{\mu_{q+1}}\lambda_q^{L_t-4} \ell_q^{-N},
    \end{equation}
   with implicit constants depending on $\Gamma$, $M$, $\al$ and $N$. 
\end{lem}

We refer the reader to lemma 3.12 in \cite{GR23} for a proof.

\section{The Nash step}
\label{sec-Nash}

In this section we perform the main spatially oscillatory perturbation, the Nash perturbation, to conclude that the inductive assumptions (\ref{induct-u}), (\ref{induct-R}), (\ref{induct-DR}) and (\ref{induct-sup}) with $q$ replaced by $q+1$ also hold. 

\subsection{Mollification along the flow}
\label{sec-molli}
To propagate the material derivative estimate (\ref{induct-DR}) for the new stress error, we require estimates on the second material derivative of the previous stress error. As it is already standard in Nash iteration schemes, we use the mollification along the flow, which was first introduced and analyzed in \cite{Isett}.

Let $\widetilde X_t$ be the Lagrangian flow defined through (\ref{Lagr_t_gam}) and $\rho$ be a standard temporal mollifier. Fix the material mollification scale 
\[\ell_{t,q}=\delta_q^{-\frac12}\lambda_q^{-\frac12}\lambda_{q+1}^{-1}.\]
We thus have 
\[\delta_q^{\frac12}\lambda_q^{\frac{3}2}<\ell_{t,q}^{-1}<\delta_{q+1}^{\frac12}\lambda_{q+1}^{\frac{3}2}\]
and 
\[\ell_{t,q}<\mu_{q+1}^{-1}<\tau_q, \ \ \mbox{for sufficiently small} \ \ \alpha>0.\]
Define the regularized stresses
\[\bar R_{q,n}=\int_{-\ell_{t,q}}^{\ell_{t,q}}R_{q,n}(\widetilde X_t(x,t+s), t+s) \rho_{\ell_{t,q}}(s)\, ds.\]
The proof of the lemma below in a very similar setting is given in \cite{GR23}.

\begin{lem} \label{le-bar-R}
For the regularized stress $\bar R_{q,n}$ we have
\begin{equation} \label{bar_R_1}
    \|\bar R_{q, n}\|_{N} \lesssim \delta_{q+1, n} \lambda_q^{N - \alpha}, \,\,\, \forall N \in \{0, 1,..., L_t-2\},
\end{equation} 
\begin{equation} \label{bar_R_2}
    \|\tilde D_{t,\Gamma} \bar R_{q, n}\|_N \lesssim \delta_{q+1,n} \tau_q^{-1} \lambda_q^{N - \alpha}, \,\,\, \forall N \in \{0, 1,..., L_t-2\},
\end{equation}
\begin{equation} \label{bar_R_3}
    \|\tilde D_{t,\Gamma}^2 \bar R_{q, n}\|_N \lesssim \delta_{q+1,n} \tau_q^{-1} \ell_{t,q}^{-1}\lambda_q^{N - \alpha}, \,\,\, \forall N \in \{0, 1,..., L_t-2\},
\end{equation}
\begin{equation} \label{bar_R_4}
    \|\bar R_{q, n}\|_{N+L_t-2} \lesssim \delta_{q+1, n} \lambda_q^{L_t-2 -\alpha} \ell_q^{-N}, \,\,\, \forall N \geq 0
\end{equation}
\begin{equation} \label{bar_R_5}
    \|\tilde D_{t,\Gamma} \bar R_{q,n}\|_{N + L_t-2} \lesssim \delta_{q+1, n} \tau_q^{-1} \lambda_q^{L_t - 2 - \alpha} \ell_{q}^{-N}, \,\,\, \forall N \geq 0,
\end{equation}
\begin{equation} \label{bar_R_6}
    \|\tilde D_{t,\Gamma}^2 \bar R_{q,n}\|_{N + L_t-2} \lesssim \delta_{q+1, n} \tau_q^{-1}\ell_{t,q}^{-1} \lambda_q^{L_t - 2 - \alpha} \ell_{q}^{-N}, \,\,\, \forall N \geq 0,
\end{equation}
where the implicit constants depend on $\Gamma$, $M$, $\alpha$ and $N$.
\end{lem}

\subsection{Construction of the Nash perturbation and the new solution to the SQG-Reynolds system}
\label{sec-construct-Nash}
Let $\widetilde \Phi_k$ be the backward flow of $\tilde u_{q,\Gamma}$ starting from time $t=t_k$, i.e. $\widetilde \Phi_k$ satisfies
\begin{equation}\notag
\begin{cases}
\partial_t\widetilde \Phi_k+\tilde u_{q,\Gamma}\cdot\nabla \widetilde \Phi_k=0,\\
\widetilde \Phi_k|_{t=t_k}=x.
\end{cases}
\end{equation}
Define the amplitude functions (cf. \eqref{def.a.coeff}) 
\begin{equation}\label{def-bar-a}
\bar a_{\xi, k, n}=2\lambda_{q+1}^{\frac{1}2}\delta_{q+1,n}^{\frac12}|(\nabla\widetilde \Phi_k)^{T}\xi|^{\frac{3}2}\chi_k\gamma_{\xi}\Big((\nabla\widetilde \Phi_k)^{-T}(\nabla\widetilde\Phi_k)^{-1}-(\nabla\widetilde \Phi_k)^{-T}\frac{\bar R_{q,n}}{\delta_{q+1,n}} (\nabla\widetilde \Phi_k)^{-1}\Big),
\end{equation} 
and the Nash perturbation
\begin{equation}\label{def-theta-Nash}
\theta_{q+1}^{(p)}=\sum_{n=0}^{\Gamma-1}\sum_{k\in \mathbb Z_{q,n}}\sum_{\xi\in F}g_{\xi, k, n+1}(\mu_{q+1}\cdot)P_{\approx \lambda_{q+1}}\left(\bar a_{\xi, k, n}\cos(\lambda_{q+1}\widetilde \Phi_k\cdot\xi) \right).
\end{equation}
In the expression above, $P_{\approx \lambda_{q+1}}$ is a projection on an annulus in Fourier space with radii $\approx \lambda_{q+1}$. More precisely, let $A \subset \mathbb R^2$ be an annulus centered at the origin such that all of the vectors $2\xi$ and $\xi/2$ ($\xi \in F$) are contained in $A$. Let, then, $\chi:\mathbb R^2 \rightarrow \mathbb R$ be a smooth function with support in a slightly larger annulus $A'$ which, moreover, satisfies $\chi(x) = 1$ for all $\xi \in A$. Then, for any function $f:\mathbb T^2 \rightarrow \mathbb R$, we let 
\begin{equation*}
    P_{\approx \lambda_{q+1}} := \sum_{k \in \mathbb Z^2} \chi(\lambda_{q+1}^{-1}k) \hat f(k) e^{i k \cdot x}.
\end{equation*}
We will also make use of a projection on a slightly enlarged annulus. Let $A''$ and $A'''$ be two annuli centered at the origin such that $A' \subset A'' \subset A'''$; and $\tilde \chi:\mathbb R^2 \rightarrow \mathbb R$ such that $\tilde \chi = 1$ on $A''$ and which vanishes outside $A'''$. For $f:\mathbb T^2 \rightarrow \mathbb R$, let 
\begin{equation*}
    \tilde P_{\approx \lambda_{q+1}} := \sum_{k\in \mathbb Z^2}\tilde \chi(\lambda_{q+1}^{-1}k) \hat f(k) e^{i k \cdot x}.
\end{equation*}
These definitions readily generalize to vector fields, tensor fields, etc. Moreover, we remark that 
\begin{equation*}
    \tilde P_{\approx \lambda_{q+1}} f = f,
\end{equation*}
for any $f:\mathbb T^2 \rightarrow \mathbb R$ such that $\supp \hat f \subset \lambda_{q+1} A''$. By choosing $a_0$ sufficiently large, we can ensure that the latter holds in particular for any function whose Fourier series is supported in a neighbourhood of radius $10 \ell_q^{-1}$ around $\lambda_{q+1}A'$.

In view of lemma \ref{osc_prof}, the terms in (\ref{def-theta-Nash}) have pair-wise disjoint temporal supports.  Indeed, if $j\neq k$, we have that either $|j-k| > 2$, in which case 
\begin{equation*}
    \supp_t \bar a_{\xi,k,n} \cap \supp_t \bar a_{\eta, j, m} = \emptyset,
\end{equation*}
or $j$ and $k$ have distinct parities, in which case
\[\supp_t g_{\xi,k,n+1}\cap \supp_t g_{\eta, j, m+1}=\emptyset.\]
On the other hand, if $j=k$ and 
\[\supp_t g_{\xi,k,n+1}\cap \supp_t g_{\eta, j, m+1}\neq\emptyset,\]
it follows from lemma \ref{osc_prof} that $(\xi, n)=(\eta, m)$.

We define $\theta_{q+1}$ and $u_{q+1}$:
\begin{equation}\label{theta-q1}
\begin{split}
\theta_{q+1}& =\theta_q+\theta_{q+1}^{(t)} +\theta_{q+1}^{(p)},\\
u_{q+1}&=u_q+w_{q+1}^{(t)} +w_{q+1}^{(p)},
\end{split} 
\end{equation}
with 
\[w_{q+1}^{(t)}=\nabla^{\perp}\Lambda^{-1}\theta_{q+1}^{(t)}, \ \ w_{q+1}^{(p)}=\nabla^{\perp}\Lambda^{-1}\theta_{q+1}^{(p)}.\]
The SQG-Reynolds system at the $(q+1)$-th stage is, thus, 
\begin{equation}\notag
\partial_t\theta_{q+1}+u_{q+1}\cdot\nabla \theta_{q+1}=\nabla^{\perp}\cdot\div R_{q+1},
\end{equation}
with the new stress $R_{q+1}$ decomposed as
\begin{equation}\label{R-q1-decom}
R_{q+1}=R_{q+1,L}+R_{q+1,O}+R_{q+1,R},
\end{equation}
where the linear error $R_{q+1,L}$ and residual error $R_{q+1,R}$ are respectively given by
\begin{equation}\notag
\begin{split}
R_{q+1, L}&:=\div^{-1}\nabla^{\perp}\Delta^{-1}\left( \tilde D_{t,\Gamma} \theta_{q+1}^{(p)}+T[\theta_{q+1}^{(p)}]\cdot \nabla \tilde \theta_{q,\Gamma}\right),\\
R_{q+1,R}&:=R_{q,\Gamma}+P_{q+1,\Gamma}+\div^{-1}\nabla^{\perp}\cdot\Delta^{-1} \left(T[\theta_{q+1}^{(p)}]\cdot\nabla(\theta_{q,\Ga}-\tilde\theta_{q,\Ga})+\left(T[\theta_{q,\Ga}]-T[\tilde\theta_{q,\Ga}]\right)\cdot\nabla\theta_{q+1}^{(p)}\right).
\end{split}
\end{equation}
The oscillation error $R_{q+1, O}$ will be precisely defined in subsection~\ref{sec.est.osc} and satisfies
\begin{equation*}
    \na^\perp \cdot \div R_{q+1, O}=\nabla^{\perp}\cdot \div S_{q,\Gamma}+T[\theta_{q+1}^{(p)}]\cdot\nabla\theta_{q+1}^{(p)}.
\end{equation*}

Recall that the temporal supports of $\theta_{q+1}^{(t)}$, $u_{q+1}^{(t)}$, $R_{q,\Gamma}$ and $S_{q,\Gamma}$ are contained in 
\[[-2+(\delta_q^{\frac12}\lambda_q^{\frac{3}2})^{-1}-2\Gamma\tau_q, -1-(\delta_q^{\frac12}\lambda_q^{\frac{3}2})^{-1}+2\Gamma\tau_q] \cup [1+(\delta_q^{\frac12}\lambda_q^{\frac{3}2})^{-1}-2\Gamma\tau_q, 2-(\delta_q^{\frac12}\lambda_q^{\frac{3}2})^{-1}+2\Gamma\tau_q]. \]
In view of the definition (\ref{def-theta-Nash}), $\theta_{q+1}^{(p)}$ and $w_{q+1}^{(p)}$ also have temporal supports contained in the set above. Hence, it holds for $\theta_{q+1}$, $u_{q+1}$ and $R_{q+1}$. Moreover, for any $\alpha>0$, there exists a large enough constant $a_0>0$ depending on $\alpha$, $b$ and $\beta$ such that 
\[(\delta_{q+1}^{\frac12}\lambda_{q+1}^{\frac{3}2})^{-1}+2\Gamma (\delta_q^{\frac12}\lambda_q^{\frac{3}2})^{-1}\lambda_{q+1}^{-\alpha}<(\delta_q^{\frac12}\lambda_q^{\frac{3}2})^{-1}.\]
Therefore it follows that (\ref{induct-sup}) is valid with $q$ replaced by $q+1$.

\subsection{Estimates of the Nash perturbation}
\label{sec-error-Nash}
We first collect estimates the amplitude functions of the Nash perturbation.

\begin{lem}\label{le-amp-reg}
For $\bar a_{\xi, k, n}$ defined in (\ref{def-bar-a}) we have the estimates
\begin{equation} \label{bar_a_estim_1}
    \|\bar a_{\xi, k, n}\|_N \lesssim \delta_{q+1, n}^{\frac12} \la_{q+1}^{\frac{1}2} \lambda_q^N, \,\,\, \forall N \in \{0,1,..., L_t-3\}
\end{equation}
\begin{equation} \label{bar_a_estim_2}
     \|\tilde D_{t,\Gamma} \bar a_{\xi, k, n}\|_N  \lesssim \delta_{q+1, n}^{\frac12} \la_{q+1}^{\frac{1}2} \tau_q^{-1} \lambda_q^N, \,\,\, \forall N \in \{0,1,..., L_t-3\}
\end{equation}
\begin{equation} \label{bar_a_estim_3}
    \|\bar a_{\xi, k, n}\|_{N+L_t-3} \lesssim \delta_{q+1, n}^{\frac12} \la_{q+1}^{\frac{1}2} \lambda_q^{L_t-3} \ell_q^{-N}, \,\,\, \forall N \geq 0
\end{equation}
\begin{equation} \label{bar_a_estim_4}
    \|\tilde D_{t,\Gamma} \bar a_{\xi, k, n}\|_{N+L_t-3} \lesssim \delta_{q+1, n}^{\frac12} \la_{q+1}^{\frac{1}2} \lambda_q^{L_t-3}  \tau_q^{-1} \ell_q^{-N}, \,\,\, \forall N \geq 0,
\end{equation}
where the implicit constants depend on $\Gamma$, $M$, $\alpha$, and $N$.
\end{lem}

\begin{lem}\label{le-est-Nash}
There exists a constant $M_0>0$ depending only on $\beta$ and $L_\theta$ such that
\begin{equation}
    \|\Lambda^{-1} \theta_{q+1}^{(p)}\|_0 \leq \frac{1}{2} M_0 \delta_{q+1}^{\frac{1}{2}} \lambda_{q+1}^{-\frac{1}{2}},
\end{equation}
\begin{equation}\label{est-u-q-p}
\|\theta_{q+1}^{(p)}\|_{N}+ \|w_{q+1}^{(p)}\|_{N}\leq \frac12M_0\delta_{q+1}^{\frac12}\lambda_{q+1}^{\frac{1}2+N}, \ \ \forall \ N\in\{0, 1, ..., L_\theta\},
\end{equation}
and, thus, 
\begin{equation}
    \| \Lambda^{-1} (\theta_{q+1}^{(p)} + \theta_{q+1}^{(t)})\|_0 \leq M_0 \delta_{q+1}^{\frac{1}{2}} \lambda_{q+1}^{-\frac{1}{2}},
\end{equation}
\begin{equation}\label{est-u-q-total}
\|\theta_{q+1}^{(t)}+\theta_{q+1}^{(p)}\|_{N}+\|w_{q+1}^{(t)}+w_{q+1}^{(p)}\|_{N}\leq M_0\delta_{q+1}^{\frac12}\lambda_{q+1}^{\frac{1}2+N}, \ \ \forall \ N\in\{0, 1, ..., L_\theta\}.
\end{equation}
\end{lem}

\begin{proof}

We argue first for $\theta_{q+1}^{(p)}$. In view of the definition (\ref{def-theta-Nash}), we deduce
\begin{equation*}
    \|\theta_{q+1}^{(p)}\|_0 \lesssim \sup_{\xi, k, n} |g_{\xi, k, n+1}| \|\bar a_{\xi, k, n}\|_0 \lesssim \delta_{q+1}^{1/2} \lambda_{q+1}^{1/2},
\end{equation*}
where the implicit constant depends only on the functions $\gamma_\xi$, and the temporal profiles $g_{\xi, k, n}$. Since $\theta_{q+1}^{(p)}$ is localized at frequencies $\approx \lambda_{q+1}$, there exists a constant, which depends only on $\beta$ (through its dependence on the profiles $g_{\xi, k, n}$) and the number $L_\theta$ of controlled derivatives, such that
\begin{equation*}
    \|\Lambda^{-1} \theta_{q+1}^{(p)}\|_0 \leq C \delta_{q+1}^{1/2} \lambda_{q+1}^{-1/2},
\end{equation*}
and, for $0 \leq N \leq L_\theta $,
\begin{equation*}
    \|\theta_{q+1}^{(p)}\|_N + \|w_{q+1}^{(p)}\|_N \leq C \delta_{q+1}^{1/2} \lambda_{q+1}^{1/2 + N}.
\end{equation*}
We can define, then, $M_0 = 2 C$.

Regarding the Newton perturbation, it follows from lemma \ref{w_t_estim} and the definition of $\mu_{q+1}$ that
\begin{equation*}
    \|\Lambda^{-1} \theta_{q+1}^{(t)}\|_0 \lesssim \delta_{q+1}^{1/2} \lambda_{q+1}^{-1/2} \lambda_{q+1}^{-3\alpha} \leq \frac12M_0\delta_{q+1}^{\frac12}\lambda_{q+1}^{-\frac{1}2},
\end{equation*}
and
\begin{equation}\notag
\|\theta_{q+1}^{(t)}\|_{N} + \|w_{q+1}^{(t)}\|_N \lesssim \frac{\delta_{q+1} \lambda_q^{2} \ell_q^{-N-\alpha}}{\mu_{q+1}}\lesssim \delta_{q+1}^{\frac12}\lambda_{q+1}^{\frac{1}2}\left(\frac{\lambda_q}{\lambda_{q+1}}\right)(\lambda_q\lambda_{q+1})^{\frac12N}
\leq \frac12M_0\delta_{q+1}^{\frac12}\lambda_{q+1}^{\frac{1}2}\lambda_{q+1}^N
\end{equation}
for a sufficiently large choice of $a_0$.  
\end{proof}

\begin{lem}\label{le-est-u-0}
We have the estimates
\begin{equation}\label{est-u-N}
\|\theta_{q+1}\|_N+\|u_{q+1}\|_N\leq M\delta_{q+1}^{\frac12}\lambda_{q+1}^{\frac{1}2+N}, \ \ \forall \ N\in\{0, 1, ..., L_\theta\}
\end{equation}
\begin{equation}\label{est-u-incre}
\|\theta_{q+1}-\theta_q\|_0 + \lambda_{q+1} \|\Lambda^{-1}(\theta_{q+1} - \theta_q)\|_0
\leq 2M\delta_{q+1}^{\frac12}\lambda_{q+1}^{\frac12}.
\end{equation}
\end{lem}
\begin{proof}
Recall $\theta_{q+1}=\theta_q+\theta_{q+1}^{(t)}+\theta_{q+1}^{(p)}$ and $u_{q+1}=u_q+w_{q+1}^{(t)}+w_{q+1}^{(p)}$ and, hence, (\ref{est-u-incre}) follows from (\ref{est-u-q-total}). Moreover we deduce from (\ref{est-u-q-total})
\begin{eqnarray*}
\|\theta_{q+1}\|_N + \|u_{q+1}\|_N &\leq& \|\theta_{q}\|_N + \|u_q\|_N +\|\theta_{q+1}^{(t)}+\theta_{q+1}^{(p)}\|_{N} + \|w_{q+1}^{(t)} + w_{q+1}^{(p)}\|_N \\ 
&\leq& M\delta_q^{\frac12}\lambda_q^{\frac{1}2+N}+M_0\delta_{q+1}^{\frac12}\lambda_{q+1}^{\frac{1}2+N} \\ 
&\leq& M\delta_{q+1}^{\frac12}\lambda_{q+1}^{\frac{1}2+N},
\end{eqnarray*}
provided $a_0$ is chosen sufficiently large.
\end{proof}

The previous two lemmas fix the constant $M_0$ in the statement of the main proposition, and shows the propagation of estimates \eqref{induct-u}, as well as the validity of \eqref{eq.prop.main}.

\subsection{Estimates for the new stress \texorpdfstring{$R_{q+1}$}{rqpo}}
\subsubsection{Estimates for the linear error \texorpdfstring{$R_{q+1, L}$}{rqpol}}
Recall 
\[
\begin{split}
R_{q+1, L}
&=\underbrace{\div^{-1}\nabla^{\perp}\Delta^{-1} \tilde D_{t,\Gamma} \theta_{q+1}^{(p)}}_{\text{transport error}}+\underbrace{\div^{-1}\nabla^{\perp}\Delta^{-1} \left(T[\theta_{q+1}^{(p)}]\cdot \nabla \tilde\theta_{q,\Gamma}\right)}_{\text{Nash error}},
\end{split}
\]
We begin by presenting estimates on material derivatives. 

\begin{lem}\label{le-mat-d}
We have the estimates on $\tilde \theta_{q,\Gamma}$, $\tilde u_{q,\Gamma}$ and $\bar a_{\xi, k, n}$,
\begin{equation}\label{est-mat-bar-th1}
\|\tilde D_{t,\Gamma}\nabla \tilde \theta_{q,\Gamma}\|_N\lesssim \delta_q\lambda_q^{N+3}, \ \ \ \forall \ N\in \{0, 1, ..., L_t-4\},
\end{equation}
\begin{equation}\label{est-mat-bar-th2}
\|\tilde D_{t,\Gamma}\nabla \tilde \theta_{q,\Gamma}\|_{N+L_t-4}\lesssim \delta_q\lambda_q^{L_t-1}\ell_q^{-N}, \ \ \ \forall \ N\geq 0,
\end{equation}
\begin{equation}\label{est-mat-bar-a1}
\|\tilde D^2_{t,\Gamma} \bar a_{\xi, k, n}\|_N\lesssim \delta_{q+1,n}^{\frac12}\lambda_{q+1}^{\frac{1}2}\tau_q^{-1}\ell_{t,q}^{-1}\lambda_q^N, \ \ \ \forall \ N\in \{0, 1, ..., L_t-4\},
\end{equation}
\begin{equation}\label{est-mat-bar-a2}
\|\tilde D^2_{t,\Gamma} \bar a_{\xi, k, n}\|_{N+L_t-4}\lesssim \delta_{q+1,n}^{\frac12}\lambda_{q+1}^{\frac{1}2}\lambda_q^{L_t-4}\tau_q^{-1}\ell_{t,q}^{-1}\ell_q^{-N}, \ \ \ \forall \ N\geq 0
\end{equation}
with implicit constants depending on $\Gamma, M, \alpha$ and $N$.
\end{lem}
\begin{proof}
To estimate $\tilde D_{t,\Gamma}\nabla\tilde\theta_{q,\Gamma}$ we first write
\begin{eqnarray*}
    \tilde D_{t, \Gamma} \nabla \tilde \theta_{q, \Gamma} &=& \tilde u_{q, \Gamma} \cdot \nabla \nabla \tilde \theta_{q, \Gamma} + \nabla P_{\lesssim \ell_q^{-1}}(\partial_t \theta_q + \partial_t \theta_{q+1}^{(t)}) \\ 
    &=& \tilde u_{q, \Gamma} \cdot \nabla \nabla \tilde \theta_{q, \Gamma} + \nabla P_{\lesssim \ell_q^{-1}}(\nabla^{\perp} \div R_{q} - u_q \cdot \nabla \theta_q + \bar D_{t} \theta_{q+1}^{(t)} - \bar u_q \cdot \nabla \theta_{q+1}^{(t)}),
\end{eqnarray*}
from which we obtain 
\begin{eqnarray*}
    \|\tilde D_{t, \Gamma} \na \tilde \theta_{q, \Gamma}\|_N &\lesssim& \|\tilde u_{q, \Gamma} \cdot \nabla \nabla \tilde \theta_{q, \Gamma}\|_N + \|R_{q, 0}\|_{N+3} + \|\bar D_t \theta_{q+1}^{(t)}\|_{N+1} \\
    &&+ \|P_{\lesssim \ell_q^{-1}}( u_q \cdot \nabla \theta_q)\|_{N+1} + \|\bar u_q \cdot \nabla \theta_{q+1}^{(t)}\|_{N+1}.
\end{eqnarray*}
Estimates \eqref{est-mat-bar-th1} and \eqref{est-mat-bar-th2} follow, then, from the inductive assumptions \eqref{induct-u}, lemmas \ref{smoli_estim} and \ref{w_t_estim}, and corollary \ref{Gamma_velo_estim}. To aid the reader, we point out that the first and fourth terms are the dominant ones.


The estimates on $\bar D^2_{t,\Gamma} \bar a_{\xi, k, n}$ can be established analogously to the proof of lemma 4.7 from \cite{GR23}.
\end{proof}

We are ready to establish the estimates for the transport term.
\begin{lem}\label{le-tran}
We have
\begin{equation}\label{est-tran1}
\|\div^{-1}\nabla^{\perp}\Delta^{-1} \tilde D_{t,\Gamma} \theta_{q+1}^{(p)}\|_N\lesssim  \delta_{q+1}\frac{\lambda_q}{\lambda_{q+1}^{1-4\alpha}}\lambda_{q+1}^{N}, \ \ \forall \ N\geq 0, 
\end{equation}
\begin{equation}\label{est-tran2}
\|\tilde D_{t,\Gamma}\div^{-1}\nabla^{\perp}\Delta^{-1} \tilde D_{t,\Gamma} \theta_{q+1}^{(p)}\|_N\lesssim  \ell_{t,q}^{-1}\delta_{q+1}\frac{\lambda_q}{\lambda_{q+1}^{1-4\alpha}}\lambda_{q+1}^{N}, \ \ \forall \ N\geq 0
\end{equation}
with implicit constants depending on $\Gamma, \alpha, M$ and $N$.
\end{lem}

\begin{proof}
We begin by writing 
\begin{equation}\notag
\begin{split}
\div^{-1}\nabla^{\perp}\Delta^{-1} &\tilde D_{t,\Gamma} \theta_{q+1}^{(p)}\\
=&\div^{-1}\nabla^{\perp}\Delta^{-1} \tilde D_{t,\Gamma} \sum_{\xi,k,n}g_{\xi, k, n+1}P_{\approx \lambda_{q+1}}\left(\bar a_{\xi, k, n}\cos(\lambda_{q+1}\widetilde \Phi_k\cdot\xi) \right)\\
=& \underbrace{\div^{-1}\nabla^{\perp}\Delta^{-1} \sum_{\xi,k,n}P_{\approx \lambda_{q+1}}\left(\tilde D_{t,\Gamma}(g_{\xi, k, n+1} \bar a_{\xi, k, n})\cos(\lambda_{q+1}\widetilde \Phi_k\cdot\xi) \right)}_{T_1}\\
&\qquad + \underbrace{\div^{-1}\nabla^{\perp}\Delta^{-1} \sum_{\xi,k,n}[\tilde D_{t,\Gamma}, P_{\approx \lambda_{q+1}}]\left(g_{\xi, k, n+1}\bar a_{\xi, k, n}\cos(\lambda_{q+1}\widetilde \Phi_k\cdot\xi) \right)}_{T_2},
\end{split}
\end{equation}
where, in $T_1$, we use the fact that $\widetilde \Phi_k$ is the flow of $\tilde u_{q, \Gamma}$. For the first term, since the operator $\div^{-1} \nabla^\perp \Delta^{-1}$ is of order $-2$, we obtain 
\begin{equation*}
    \|T_1\|_N \lesssim \lambda_{q+1}^{N-2}\sup_{\xi, k, n}\big( \mu_{q+1} \|\bar a_{\xi, k, n}\|_0 + \|\tilde D_{t, \Gamma} \bar a_{\xi, k, n}\|_0\big).
\end{equation*}
For the second term, since $\tilde u_{q, \Gamma}$ is localized to frequencies $\lesssim \ell_q^{-1}$, we can write 
\begin{equation*}
    T_2 = \div^{-1}\nabla^{\perp}\Delta^{-1} \sum_{\xi,k,n} \tilde P_{\approx \lambda_{q+1}}[\tilde D_{t,\Gamma}, P_{\approx \lambda_{q+1}}]\left(g_{\xi, k, n+1}\bar a_{\xi, k, n}\cos(\lambda_{q+1}\widetilde \Phi_k\cdot\xi) \right),
\end{equation*}
where $\tilde P_{\approx \lambda_{q+1}}$ is the previously defined frequency projection on a slightly larger annulus of radii $\approx \lambda_{q+1}$. It follows, then, using proposition \ref{prop-commu-loc} that 
\begin{equation*}
    \|T_2\|_N \lesssim \lambda_{q+1}^{N-2} \|\nabla \tilde u_{q, \Gamma}\|_0 \sup_{\xi, k, n}\|\bar a_{\xi, k, n}\|_0.
\end{equation*}
We conclude that 
\begin{equation*}
    \|\div^{-1} \nabla^{\perp} \Delta^{-1} \tilde D_{t, \Gamma} \theta_{q+1}^{(p)}\|_N \lesssim \lambda_{q+1}^{N-2} \sup_{\xi, k, n}\big(\mu_{q+1} \|\bar a_{\xi, k, n}\|_0 + \|\tilde D_{t, \Gamma} \bar a_{\xi, k, n}\|_0 + \|\nabla \tilde u_{q, \Gamma}\|_0 \|\bar a_{\xi, k, n}\|_0\big) \lesssim \lambda_{q+1}^{N-3/2} \delta_{q+1}^{1/2} \mu_{q+1}, 
\end{equation*}
from which the first claimed estimate follows. 

We turn, now, to estimating the material derivative of the transport term. We write 
\begin{eqnarray*}
    \tilde D_{t,\Gamma}\div^{-1}\nabla^{\perp}\Delta^{-1} \tilde D_{t,\Gamma} \theta_{q+1}^{(p)} &=& \underbrace{[\tilde D_{t,\Gamma},\div^{-1}\nabla^{\perp}\Delta^{-1} P_{\approx \lambda_{q+1}}]\sum_{\xi, k, n} \tilde D_{t, \Gamma}(g_{\xi, k, n+1} \bar a_{\xi, k, n}) \cos(\lambda_{q+1} \widetilde \Phi_k\cdot\xi)}_{T_{11}} \\
    && + \underbrace{\div^{-1}\nabla^{\perp}\Delta^{-1} P_{\approx \lambda_{q+1}} \sum_{\xi, k, n} \tilde D^2_{t, \Gamma}(g_{\xi, k, n+1} \bar a_{\xi, k, n}) \cos(\lambda_{q+1} \widetilde \Phi_k\cdot\xi)}_{T_{12}} \\ 
    && + \underbrace{[\tilde D_{t,\Gamma},\div^{-1}\nabla^{\perp}\Delta^{-1} \tilde P_{\approx \lambda_{q+1}}] [\tilde D_{t, \Gamma}, P_{\approx \lambda_{q+1}}] \sum_{\xi, k, n} g_{\xi, k, n+1} \bar a _{\xi, k, n} \cos(\lambda_{q+1} \widetilde \Phi_k\cdot\xi)}_{T_{21}} \\ 
    && + \underbrace{\div^{-1}\nabla^{\perp}\Delta^{-1} \tilde P_{\approx \lambda_{q+1}} \tilde D_{t, \Gamma} [\tilde D_{t, \Gamma}, P_{\approx \lambda_{q+1}}]\sum_{\xi, k, n} g_{\xi, k, n+1} \bar a _{\xi, k, n} \cos(\lambda_{q+1} \widetilde \Phi_k\cdot\xi)}_{T_{22}}
\end{eqnarray*}
Using proposition \ref{prop-commu-loc}, we estimate 
\begin{equation*}
    \|T_{11}\|_N \lesssim \lambda_{q+1}^{N-2} \|\nabla \tilde u_{q, \Gamma}\|_0 \sup_{\xi, k, n} \big( \mu_{q+1} \|\bar a_{\xi, k, n}\|_0 + \|\tilde D_{t, \Gamma} \bar a_{\xi, k, n}\|_0 \big),
\end{equation*}
\begin{equation*}
    \|T_{12}\|_N \lesssim \lambda_{q+1}^{N-2} \sup_{\xi, k, n} \big( \mu_{q+1}^{2} \|\bar a_{\xi, k, n}\|_0 + \mu_{q+1} \|\tilde D_{t, \Gamma} \bar a_{\xi, k, n}\|_0 + \|\tilde D_{t, \Gamma}^2 \bar a_{\xi, k, n}\|_0\big),
\end{equation*}
\begin{equation*}
    \|T_{21}\|_N \lesssim \lambda_{q+1}^{N-2} \|\nabla \tilde u_{q, \Gamma}\|_0^2 \sup_{\xi, k, n}\|\bar a_{\xi, k, n}\|_0.
\end{equation*}
For the remaining term, we write 
\begin{eqnarray*}
    T_{22}&=& \div^{-1} \nabla^\perp \Delta^{-1} \tilde P_{\approx \lambda_{q+1}} \bigg[ [\partial_t \tilde u_{q, \Gamma} \cdot \nabla, P_{\approx \lambda_{q+1}}] \sum_{\xi, k, n} g_{\xi, k, n+1} \bar a_{\xi, k, n} \cos (\lambda_{q+1} \widetilde \Phi_k\cdot\xi) \\  
    && + [\tilde u_{q, \Gamma} \cdot \nabla, P_{\approx \lambda_{q+1}}]  \sum_{\xi, k, n}\tilde D_{t, \Gamma}( g_{\xi, k, n+1} \bar a_{\xi, k, n}) \cos (\lambda_{q+1} \widetilde \Phi_k\cdot\xi) \\ 
    && + \left[\tilde u_{q, \Gamma} \cdot \nabla, [\tilde D_{t, \Gamma}, P_{\approx \lambda_{q+1}}]\right]  \sum_{\xi, k, n} g_{\xi, k, n+1} \bar a_{\xi, k, n} \cos (\lambda_{q+1} \widetilde \Phi_k\cdot\xi) \bigg].
\end{eqnarray*}
The first two terms above can be estimated by using proposition \ref{prop-commu-loc} as before, while for the final term, we can simply estimate each term in the commutator. Using, then, 
\begin{equation*}
    \|\partial_t \tilde u_{q, \Gamma}\|_0 \lesssim \|\tilde D_{t, \Gamma} \tilde u_{q,\Gamma}\|_0 + \|\tilde u_{q, \Gamma} \cdot \nabla \tilde u_{q, \Gamma}\|_0,
\end{equation*}
we obtain
\begin{eqnarray*}
    \|T_{22}\|_N &\lesssim& \lambda_{q+1}^{N-2} \sup_{\xi, k, n} \big( (\|\tilde D_{t, \Gamma} \tilde u_{q, \Gamma}\|_1 + \|\tilde u_{q, \Gamma}\|_0 \|\tilde u_{q, \Gamma}\|_2 +\mu_{q+1}\|\tilde u_{q, \Gamma}\|_1 +\lambda_{q+1} \|\tilde u_{q, \Gamma}\|_0 \|\tilde u_{q, \Gamma}\|_1  ) \|\bar a_{\xi, k, n}\|_0 \\
    && + \|\tilde u_{q, \Gamma}\|_1\|\tilde D_{t,\Gamma} \bar a_{\xi, k, n}\|_0\big).
\end{eqnarray*}
It is clear by inspection that the largest term is the one involving two material derivatives which appears in the estimate for $T_{12}$. The conclusion follows.
\end{proof}

We obtain the estimates for the Nash error term as follows.
\begin{lem}\label{le-Nash}
The estimates
\begin{equation}\label{est-Nash1}
\left\|\div^{-1}\nabla^{\perp}\Delta^{-1} \left(T[\theta_{q+1}^{(p)}]\cdot \nabla \tilde\theta_{q,\Gamma}\right)\right\|_N\lesssim \frac{\delta_q^{\frac12}\delta_{q+1}^{\frac12}\lambda_q^{\frac{3}2}}{\lambda_{q+1}^{\frac{3}2}}\lambda_{q+1}^{N}, \ \ \forall \ N\geq 0, 
\end{equation}
\begin{equation}\label{est-Nash2}
\left\|\tilde D_{t,\Gamma}\div^{-1}\nabla^{\perp}\Delta^{-1} \left(T[\theta_{q+1}^{(p)}]\cdot \nabla \tilde\theta_{q,\Gamma}\right)\right\|_N\lesssim \frac{\mu_{q+1}\delta_q^{\frac12}\delta_{q+1}^{\frac12}\lambda_q^{\frac{3}2}}{\lambda_{q+1}^{\frac{3}2}}\lambda_{q+1}^{N}, \ \ \forall \ N\geq 0
\end{equation}
hold for implicit constants depending on $\Gamma, \alpha, M$ and $N$.
\end{lem}
\begin{proof}
As in the proof of the previous lemma, since $\tilde \theta_{q,\Ga}$ is localized to frequencies $\lesssim \ell_q^{-1} \ll \lambda_{q+1}$, we can write 
\begin{eqnarray*}
    \left\|\div^{-1}\nabla^{\perp}\Delta^{-1} \left(T[\theta_{q+1}^{(p)}]\cdot \nabla \tilde \theta_{q,\Gamma}\right)\right\|_N &=& \left\|\div^{-1}\nabla^{\perp}\Delta^{-1} \tilde P_{\approx \lambda_{q+1}} \left(T[\theta_{q+1}^{(p)}]\cdot \nabla \tilde \theta_{q,\Gamma}\right)\right\|_N \\ 
    &\lesssim& \lambda_{q+1}^{N-2} \sup_{\xi, k, n}\|\bar a_{\xi, k, n}\|_0 \|\nabla \tilde \theta_{q,\Gamma}\|_0,
\end{eqnarray*}
from which the first claimed estimate follows.

To estimate the material derivative, we write 
\begin{eqnarray*}
    \tilde D_{t, \Gamma} \div^{-1}\nabla^{\perp}\Delta^{-1} \left(T[\theta_{q+1}^{(p)}]\cdot \nabla \tilde \theta_{q,\Gamma}\right) &=& \underbrace{[\tilde D_{t, \Gamma},\div^{-1}\nabla^{\perp}\Delta^{-1} \tilde P_{\approx \lambda_{q+1}}]T[\theta_{q+1}^{(p)}] \cdot \nabla \tilde \theta_{q, \Gamma}}_{T_1} \\ 
    &&+ \underbrace{\div^{-1}\nabla^{\perp}\Delta^{-1} \tilde P_{\approx \lambda_{q+1}}\tilde D_{t, \Gamma} (T[\theta_{q+1}^{(p)}] \cdot  \nabla \tilde \theta_{q, \Gamma})}_{T_2},
\end{eqnarray*}
from which we can further expand 
\begin{eqnarray*}
    T_2 &=& \div^{-1}\nabla^{\perp}\Delta^{-1} \tilde P_{\approx \lambda_{q+1}} T[\theta_{q+1}^{(p)}] \cdot \tilde D_{t, \Gamma} \nabla \tilde \theta_{q, \Gamma}\\
    &&+\div^{-1}\nabla^{\perp}\Delta^{-1} \tilde P_{\approx \lambda_{q+1}}\nabla \tilde \theta_{q, \Gamma} \cdot \nabla [\tilde D_{t, \Gamma}, T P_{\approx \lambda_{q+1}}] \sum_{\xi, k, n} g_{\xi, k, n+1} \bar a_{\xi,k,n} \cos (\lambda_{q+1} \widetilde \Phi_k\cdot\xi) \\
    &&+ \div^{-1}\nabla^{\perp}\Delta^{-1} \tilde P_{\approx \lambda_{q+1}}\nabla \tilde \theta_{q, \Gamma} \cdot T P_{\approx \lambda_{q+1}} \sum_{\xi, k, n} \tilde D_{t, \Gamma}(g_{\xi, k, n+1} \bar a_{\xi,k,n}) \cos (\lambda_{q+1} \widetilde \Phi_k\cdot\xi).
\end{eqnarray*}
We can now employ proposition \ref{prop-commu-loc} to obtain 
\begin{eqnarray*}
    \bigg\|\tilde D_{t, \Gamma} \div^{-1}\nabla^{\perp}\Delta^{-1} \left(T[\theta_{q+1}^{(p)}]\cdot \nabla \tilde \theta_{q,\Gamma}\right)\bigg\|_N &\lesssim& \lambda_{q+1}^{N-2} \sup_{\xi, k, n} \big( \|\bar a_{\xi, k, n}\|_0(\|\tilde u_{q, \Gamma}\|_1 \| \tilde \theta_{q,\Gamma}\|_1 + \|\tilde D_{t, \Gamma} \nabla \tilde \theta_{q,\Gamma}\|_0 \\
    && + \mu_{q+1} \|\tilde \theta_{q,\Gamma}\|_1) + \| \tilde \theta_{q,\Gamma}\|_1 \|\tilde D_{t, \Gamma} \bar a_{\xi, k, n}\|_0\big).
\end{eqnarray*}
The wanted estimate follows.
\end{proof}

\subsubsection{Estimates of $R_{q+1, O}$}\label{sec.est.osc}
Recall the oscillation error satisfies
\[\na^\perp \cdot \div R_{q+1, O}=\nabla^{\perp}\cdot \div S_{q,\Gamma}+T[\theta_{q+1}^{(p)}]\cdot\nabla\theta_{q+1}^{(p)}.\]
The goal of this subsection is to apply the bilinear microlocal lemma \ref{le-bilinear-odd} and show how the first order of the quadratic self-interaction cancels $S_{q,\Gamma}$ up to a small error. 
Recall 
\[A_{\xi, k, n}=\frac14\frac{1}{\lambda_{q+1}|(\nabla\Phi_k)^{T}\xi|^{3}}a_{\xi,k,n}^2(\nabla\Phi_k)^{T}\xi\otimes\xi(\nabla\Phi_k)^{}.\]
Denote similarly
\[\bar A_{\xi, k, n}=\frac14\frac{1}{\lambda_{q+1}|(\nabla\widetilde\Phi_k)^{T}\xi|^{3}}\bar a_{\xi,k,n}^2(\nabla\widetilde\Phi_k)^{T}\xi\otimes\xi(\nabla\widetilde\Phi_k)^{}.\]
The advantage of introducing the Newton step is to avoid interactions among different directions by diverting them using temporal oscillations. The crucial outcome is that $\{g_{\xi, k, n+1}\bar a_{\xi, k,n}\}_{\xi, k,n}$ have pair-wise disjoint supports. 
Denote
\[\theta_{\xi,k,n}^{(p)}=g_{\xi, k, n+1}P_{\approx \lambda_{q+1}}\left(\bar a_{\xi, k, n}\cos(\lambda_{q+1}\widetilde \Phi_k\cdot\xi) \right).\]
We then have, applying the bilinear microlocal lemma \ref{le-bilinear-odd},
\begin{equation}\notag
\begin{split}
T[\theta_{q+1}^{(p)}]\cdot\nabla\theta_{q+1}^{(p)}=&\sum_{\xi, k,n}T[\theta_{\xi,k,n}^{(p)}]\cdot\nabla \theta_{\xi,k,n}^{(p)}\\
=& \, \nabla^{\perp}\cdot\div \left(\sum_{\xi, k,n}g^2_{\xi,k,n+1}\bar A_{\xi,k,n}+g^2_{\xi,k,n+1}\delta B_{\xi,k,n}\right),
\end{split}
\end{equation}
where the error term $\delta B_{\xi,k,n}$ is explicit and is given in the proof of lemma~\ref{le-bilinear-odd}.

It follows that
\begin{equation}\notag
T[\theta_{q+1}^{(p)}]\cdot\nabla\theta_{q+1}^{(p)}+\nabla^{\perp}\cdot \div S_{q,\Gamma}=\nabla^{\perp}\cdot\div \sum_{\xi, k,n}g^2_{\xi,k,n+1}\left(\bar A_{\xi,k,n}-A_{\xi,k,n}\right)+\nabla^{\perp}\cdot\div \sum_{\xi, k,n}g^2_{\xi,k,n+1}\delta B_{\xi,k,n},
\end{equation}
and, hence, we can define
\begin{equation}\notag
\begin{split}
R_{q+1, O}:=
&\underbrace{\sum_{\xi, k,n}g^2_{\xi,k,n+1}\left(\bar A_{\xi,k,n}-A_{\xi,k,n}\right)}_{\text{flow error}}
+\underbrace{\sum_{\xi, k,n}g^2_{\xi,k,n+1}\delta B_{\xi,k,n}}_{\text{main oscillation error}}.
\end{split}
\end{equation}

\begin{lem} \label{le_est_A}
We have the estimates for $\bar A_{\xi,k,n}$, 
\begin{equation}\notag
    \|\bar A_{\xi, k, n}\|_N \lesssim \delta_{q+1, n} \lambda_q^N, \,\,\, \forall N \in \{0,1,..., L_t-3\},
\end{equation}
\begin{equation}\notag
    \|\tilde D_t \bar A_{\xi, k, n} \|_N  \lesssim \delta_{q+1, n} \tau_q^{-1} \lambda_q^N, \,\,\, \forall N \in \{0,1,..., L_t-3\},
\end{equation}
\begin{equation}\notag
    \|\bar A_{\xi, k, n}\|_{N+L_t-3} \lesssim \delta_{q+1, n} \lambda_q^{L_t-3} \ell_q^{-N},  \,\,\, \forall N \geq 0,
\end{equation}
\begin{equation}\notag
    \|\tilde D_t \bar A_{\xi, k, n} \|_{N+ L_t-3} \lesssim \delta_{q+1, n} \lambda_q^{L_t-3} \tau_q^{-1} \ell_q^{-N}, \,\,\, \forall N \geq 0,
\end{equation}
with implicit constants depending on $\Gamma$, $M$, $\alpha$, and $N$. Moreover the difference $\bar A_{\xi,k,n}-A_{\xi,k,n}$ satisfies
\begin{equation}\label{A-diff}
\|\bar A_{\xi,k,n}-A_{\xi,k,n}\|_0\lesssim \delta_{q+1,n}\frac{\delta_{q+1}^{\frac12}\lambda_q^{\frac12}}{\delta_q^{\frac12}\lambda_{q+1}^{\frac12}}.
\end{equation}
\end{lem}
\begin{proof}
The first four estimates on $\bar A_{\xi,k,n}$ can be proven analogously to the estimates in lemma \ref{a_estim}. In the process we apply corollary \ref{Flow_gam_estim} and lemma \ref{le-amp-reg}.

We only show the details for \eqref{A-diff}. In view of the definitions of $A_{\xi,k,n}$ and $\bar A_{\xi,k,n}$ we can write
\begin{equation}\notag
\begin{split}
\bar A_{\xi,k,n}- A_{\xi,k,n}=&\ \frac{1}{4\lambda_{q+1}}(\bar a_{\xi,k,n}^2-a_{\xi,k,n}^2)\frac{(\nabla\widetilde \Phi_k)^{T}}{|(\nabla\widetilde \Phi_k)^{T}\xi|^{\frac{3}2}}\xi\otimes \xi \frac{(\nabla\widetilde \Phi_k)}{|(\nabla\widetilde \Phi_k)^{T}\xi|^{\frac{3}2}}\\
&+\frac{1}{4\lambda_{q+1}}a_{\xi,k,n}^2\left[\frac{(\nabla\widetilde \Phi_k)^{T}}{|(\nabla\widetilde \Phi_k)^{T}\xi|^{\frac{3}2}}-\frac{(\nabla \Phi_k)^{T}}{|(\nabla \Phi_k)^{T}\xi|^{\frac{3}2}} \right]
\xi\otimes\xi \frac{(\nabla\widetilde \Phi_k)}{|(\nabla\widetilde \Phi_k)^{T}\xi|^{\frac{3}2}}\\
&+\frac{1}{4\lambda_{q+1}}a_{\xi,k,n}^2\frac{(\nabla \Phi_k)^{T}}{|(\nabla \Phi_k)^{T}\xi|^{\frac{3}2}}\xi\otimes\xi \left[\frac{(\nabla\widetilde \Phi_k)}{|(\nabla\widetilde \Phi_k)^{T}\xi|^{\frac{3}2}}-\frac{(\nabla \Phi_k)}{|(\nabla \Phi_k)^{T}\xi|^{\frac{3}2}}\right].
\end{split}
\end{equation}
Thus it follows that
\begin{equation}\notag
\begin{split}
\|\bar A_{\xi,k,n}- A_{\xi,k,n}\|_0\lesssim&\ \lambda_{q+1}^{-1}\|\bar a_{\xi,k,n}^2-a_{\xi,k,n}^2 \|_0+\lambda_{q+1}^{-1}\|a_{\xi,k,n}^2 \|_0\|(\nabla\widetilde \Phi_k)-(\nabla \Phi_k)\|_0\\
&+\lambda_{q+1}^{-1}\|a_{\xi,k,n}^2 \|_0\left\||(\nabla\widetilde\Phi_k)^{T}\xi|^{\frac{3}2}-|(\nabla\Phi_k)^{T}\xi|^{\frac{3}2}\right\|_0\\
\lesssim&\ \lambda_{q+1}^{-1}\|\bar a_{\xi,k,n}^2-a_{\xi,k,n}^2 \|_0+\lambda_{q+1}^{-1}\|a_{\xi,k,n}^2 \|_0\|(\nabla\widetilde \Phi_k)-(\nabla \Phi_k)\|_0\\
\end{split}
\end{equation}
where we used the mean value inequality in the last step. According to the definitions (\ref{def.a.coeff}) and (\ref{def-bar-a}), applying the mean value inequality again gives 
\begin{equation}\notag
\begin{split}
\|\bar a_{\xi,k,n}^2-a_{\xi,k,n}^2 \|_0\lesssim&\ \lambda_{q+1}\delta_{q+1,n}\left\|(\nabla\widetilde \Phi_k)^{-T} \frac{\delta_{q+1,n}\mbox{Id}-\bar R_{q,n}}{\delta_{q+1,n}}(\nabla\widetilde \Phi_k)^{-1}-(\nabla \Phi_k)^{-T}\frac{\delta_{q+1,n}\mbox{Id}- R_{q,n}}{\delta_{q+1,n}}(\nabla \Phi_k)^{-1}\right\|_0\\
&+\lambda_{q+1}\delta_{q+1,n}\left\||(\nabla\widetilde\Phi_k)^{T}\xi|^{\frac{3}2}-|(\nabla\Phi_k)^{T}\xi|^{\frac{3}2}\right\|_0\\
\lesssim&\ \lambda_{q+1}\|(\nabla\widetilde \Phi_k)^{-T}-(\nabla\Phi_k)^{-T}\|_0\|\delta_{q+1,n}\mbox{Id}-\bar R_{q,n}\|_0+\lambda_{q+1}\|R_{q,n}-\bar R_{q,n}\|_0\\
&+\lambda_{q+1}\|\nabla\widetilde \Phi_k-\nabla\Phi_k\|_0\|\delta_{q+1,n}\mbox{Id}- R_{q,n}\|_0+\lambda_{q+1}\delta_{q+1,n}\|(\nabla\widetilde \Phi_k)^{-1}-(\nabla \Phi_k)^{-1}\|_0.
\end{split}
\end{equation}
On the other hand, we have the standard mollification estimate
\[\|R_{q,n}-\bar R_{q,n}\|_0\lesssim \|\bar D_{t,\Gamma}R_{q,n}\|_0\ell_{t,q}\lesssim \delta_{q+1,n}\tau_q^{-1}\ell_{t,q}.\]
Combining the last three estimates and applying lemmas \ref{a_estim}, \ref{flow_stabil} and \ref{le-bar-R},
we deduce
\begin{equation}\notag
\begin{split}
\|\bar A_{\xi,k,n}- A_{\xi,k,n}\|_0\lesssim&\ \|\nabla\widetilde \Phi_k-\nabla\Phi_k\|_0\|\delta_{q+1,n}\mbox{Id}-\bar R_{q,n}\|_0+\|R_{q,n}-\bar R_{q,n}\|_0\\
&+\|\nabla\widetilde \Phi_k-\nabla\Phi_k\|_0\|\delta_{q+1,n}\mbox{Id}- R_{q,n}\|_0+\delta_{q+1,n}\|(\nabla\widetilde \Phi_k)^{-1}-(\nabla \Phi_k)^{-1}\|_0\\
\lesssim&\ \delta_{q+1,n}\left(\tau_q\frac{\delta_{q+1}\lambda_q^{3}\ell_q^{-\alpha}}{\mu_{q+1}}+\tau_q^{-1}\ell_{t,q} \right)
\lesssim \delta_{q+1,n}\left(\frac{\delta_{q+1}^{\frac12}\lambda_q^{\frac12}}{\delta_q^{\frac12}\lambda_{q+1}^{\frac12}}+\frac{\lambda_q}{\lambda_{q+1}}\lambda_{q+1}^{\alpha} \right)
\end{split}
\end{equation}
where we recalled
\begin{equation}\notag
\begin{split}
\tau_q=\delta_q^{-\frac12}\lambda_q^{-\frac{3}2}\lambda_{q+1}^{-\alpha},\quad
\mu_{q+1}=\delta_{q+1}^{\frac12}\lambda_q\lambda_{q+1}^{\frac12}\lambda_{q+1}^{4\alpha},\quad
\ell_{t,q}=\delta_q^{-\frac12}\lambda_q^{-\frac12}\lambda_{q+1}^{-1}.
\end{split}
\end{equation}
Note 
\[\frac{\lambda_q}{\lambda_{q+1}}< \frac{\delta_{q+1}^{\frac12}\lambda_q^{\frac12}}{\delta_q^{\frac12}\lambda_{q+1}^{\frac12}}, \ \ \mbox{for} \ \ \beta<\frac12.\]
Thus for small enough $\alpha$ depending on $\beta$ and $b$, the following inequality
\[\frac{\lambda_q}{\lambda_{q+1}}\lambda_{q+1}^\alpha\leq \frac{\delta_{q+1}^{\frac12}\lambda_q^{\frac12}}{\delta_q^{\frac12}\lambda_{q+1}^{\frac12}}\]
holds. The inequality (\ref{A-diff}) then follows immediately.
\end{proof}
We are now in a position to estimate the flow error, which we do in the following lemma.

\begin{lem}\label{le-flow-error-est}
The following estimates for the flow error
\begin{equation}\notag
\begin{split}
\left\|\sum_{\xi,k,n}g_{\xi,k,n}^2(\bar A_{\xi,k,n}-A_{\xi,k,n})\right\|_N
\lesssim&\ \delta_{q+1}\frac{\delta_{q+1}^{\frac12}\lambda_q^{\frac12}}{\delta_q^{\frac12}\lambda_{q+1}^{\frac12}}\lambda_{q+1}^N, \ \ \forall \ N\geq 0,\\
\end{split}
\end{equation}
\begin{equation}\notag
\begin{split}
\left\|\tilde D_{t,\Gamma}\sum_{\xi,k,n}g_{\xi,k,n}^2(\bar A_{\xi,k,n}-A_{\xi,k,n})\right\|_N
\lesssim&\ \mu_{q+1}\delta_{q+1}\lambda_{q+1}^N, \ \ \forall \ N\geq 0\\
\end{split}
\end{equation}
hold with implicit constants depending on $\Gamma, M, N$ and $\alpha$. 
\end{lem}
\begin{proof}
The first estimate of the lemma is true for $N=0$ by \eqref{A-diff}. For any $N\geq 0$, lemmas \ref{a_estim} and \ref{le_est_A} give
\begin{equation}\notag
\begin{split}
\|g^2_{\xi,k,n} (\bar A_{\xi,k,n}-A_{\xi,k,n})\|_{N+1} &\lesssim \|\bar A_{\xi,k,n}\|_{N+1}+\|A_{\xi,k,n}\|_{N+1} \\
&\lesssim \delta_{q+1}\lambda_q^{N+1}
\lesssim\ \delta_{q+1}\lambda_q\lambda_{q+1}^{N}\lesssim \delta_{q+1}\frac{\delta_{q+1}^{\frac12}\lambda_q^{\frac12}}{\delta_q^{\frac12}\lambda_{q+1}^{\frac12}}\lambda_{q+1}\lambda_{q+1}^{N}
\end{split}
\end{equation}
where in the last step we used 
\[\lambda_q\leq \frac{\delta_{q+1}^{\frac12}\lambda_q^{\frac12}}{\delta_q^{\frac12}\lambda_{q+1}^{\frac12}}\lambda_{q+1}.\]
We conclude the proof of the first estimate of the lemma.

The material derivative of the flow error can be estimated similarly to give
\begin{equation}\notag
\|\tilde D_{t,\Gamma}\sum_{\xi,k,n}g_{\xi,k,n}^2(\bar A_{\xi,k,n}-A_{\xi,k,n})\|_{N}\lesssim  \mu_{q+1}\delta_{q+1}\lambda_{q+1}^{N}.
\end{equation}

\end{proof}

\begin{lem} \label{le_est_delta_B}
We have the estimates for the main oscillation error, 
\begin{equation}\notag
    \|g^2_{\xi,k,n+1} \delta B_{\xi, k, n}\|_N \lesssim \delta_{q+1}\frac{\lambda_q}{\lambda_{q+1}}\lambda_{q+1}^N, \,\,\, \forall N\geq 0,
\end{equation}
\begin{equation}\notag
    \|\tilde D_{t,\Gamma} g^2_{\xi,k,n+1} \delta B_{\xi, k, n} \|_N  \lesssim \delta_{q+1}\mu_{q+1}\frac{\lambda_q}{\lambda_{q+1}}\lambda_{q+1}^N , \,\,\, \forall N\geq 0,
\end{equation}
with implicit constants depending on $n$, $\Gamma$, $M$, and $N$. 
\end{lem}
\begin{proof}
The proof is similar to that of proposition 4.5 in \cite{IM}.
Recall first the explicit form of the error term as given in the proof of lemma~\ref{le-bilinear-odd}. We note that 
    \begin{align*}
        \left\|R_{\widetilde \Phi_k}(\cdot,h)\right\|_N &\lesssim |h|^2 \|\na^2 \widetilde\Phi_k\|_N\\
        &\lesssim \begin{cases}
            |h|^2 \la_q^{N+1} \,, \qquad \forall N \in \{0,...,L_t-4\}\\ 
            |h|^2 \la_q^{L_t-3} \ell_q^{-N+L_t - 4}\,,\qquad \forall N \geq L_t - 4
        \end{cases}\\
        &\lesssim |h|^2 \la_q \la_{q+1}^N
    \end{align*}
    where we have used the estimates of corollary~\ref{Flow_gam_estim}. Similarly, we have
    \begin{align*}
        \left\|\na_h R_{\widetilde \Phi_k}(\cdot,h)\right\|_N &\lesssim |h| \|\na^2 \widetilde\Phi_k\|_N + |h|^2 \|\na^3 \widetilde\Phi_k\|_N\\
        &\lesssim \begin{cases}
            |h| \la_q^{N+1} + |h|^2 \la_q^{N+2} \,, \qquad \forall N \in \{0,...,L_t-5\}\\ 
            |h| \la_q^{L_t-4} \ell_q^{-N+L_t - 5} + |h|^2 \la_q^{L_t-3} \ell_q^{-N + L_t - 5}\,,\qquad \forall N \geq L_t - 5
        \end{cases}\\
        &\lesssim |h| \la_q (1+|h|\la_{q+1}) \la_{q+1}^N
    \end{align*}
    
    By applying proposition~\ref{comp_estim} with $\Psi(y) = e^{\pm iy}$ and $u(x) = \lambda_{q+1} R_{\widetilde \Phi_k}(x,h)\cdot \xi$, one has
    \begin{eqnarray*}
         [e^{\pm i\lambda_{q+1}R_{\widetilde \Phi_k}(\cdot,h)\cdot\xi}]_N &\lesssim& \lambda_{q+1} \|R_{\tilde \Phi_k}(\cdot, h)\|_N + \lambda_{q+1}^N\| \|R_{\tilde \Phi_k}(\cdot,h)\|_1^N \\
         &\lesssim& |h|^2 \lambda_q \lambda_{q+1}^{N+1} + |h|^{2N} \lambda_q^N \lambda_{q+1}^{2N}\\
         &\lesssim& (|h|^2 \lambda_{q+1}^2 + |h|^{2N} \lambda_{q+1}^{2N}) \lambda_{q+1}^{N}\\
         &\lesssim&(1+|h|^{2N}\lambda_{q+1}^{2N})\lambda_{q+1}^N,
    \end{eqnarray*}
    for all $N>0$. In the last line, we use the elementary inequality $|h|^k \lambda_{q+1}^k \leq 1 + |h|^N \lambda_{q+1}^N$, which holds for all $0 \leq k \leq N$. The same estimate trivially holds also for $N=0$.
    Similarly, again applying proposition~\ref{comp_estim} now with $\Psi(y)= e^{\pm i y}$ and $u(x) = \lambda_{q+1}\nabla\widetilde\Phi_k ^T(x)\xi\cdot h$, one estimates 
    \begin{align*}                   [e^{\pm i\lambda_{q+1}\nabla\widetilde\Phi_k ^T\xi\cdot h}]_N 
        \lesssim \lambda_{q+1}|h|\|\nabla \tilde \Phi_k\|_N +(\lambda_{q+1}|h|\|\nabla \tilde \Phi_k\|_1)^N \lesssim \lambda_{q+1}^{N+1}|h| + \lambda_{q+1}^{2N}|h|^N \lesssim (1 + \lambda_{q+1}^N|h|^N)\lambda_{q+1}^N.
    \end{align*}
    In order to continue, we rewrite 
    $$Y^\xi_{\la_{q+1}}(x,h) = \int_0^1 \frac{d}{dr} \left( \bar a_{\xi,k,n}(x-rh)e^{i\la_{q+1} R_{\widetilde \Phi_k}(x,rh)\cdot\xi} \right)\,dr.$$
    Now one can estimate $Y^\xi_{\la_{q+1}}$, using corollary~\ref{Flow_gam_estim} and lemma~\ref{le-amp-reg}, as
    \begin{align*}
        [Y^\xi_{\la_{q+1}}(\cdot,h)]_N &\lesssim \sup_{0\leq r\leq 1} \bigg[ \frac{d}{dr}\left(e^{i\lambda_{q+1}R_{\widetilde \Phi_k}(\cdot,rh) \cdot \xi}\bar a_{\xi,k,n}(\cdot-rh)\right) \bigg]_N\\
        &\lesssim \sup_{0\leq r\leq 1} \bigg[e^{i\lambda_{q+1}R_{\widetilde \Phi_k}(\cdot,rh)\cdot\xi} i \lambda_{q+1}\na_{h}R_{\widetilde \Phi_k}(\cdot,rh)h \xi\,\, \bar a_{\xi,k,n}(\cdot-rh) \bigg]_N\\
        &\qquad + \sup_{0\leq r\leq 1} \bigg[ e^{i\lambda_{q+1}R_{\widetilde \Phi_k}(\cdot,rh)\cdot\xi}\, \na \bar a_{\xi,k,n}(\cdot-rh)\cdot h \bigg]_N\\
        &\lesssim \sum_{N_1+N_2+N_3=N}\lambda_{q+1}|h|\Big(1 +  |h|^{2N_1}\la_{q+1}^{2N_1}\Big)\la_{q+1}^{N_1} \cdot |h|\lambda_q(1+|h|\lambda_{q+1})\lambda_{q+1}^{N_2} \cdot \de_{q+1,n}^\frac12 \la_{q+1}^{\frac12 +N_3}\\
        &\qquad + \sum_{N_1+N_3=N} |h| \Big(1 + |h|^{2N_1} \la_{q+1}^{2N_1}\Big)\la_{q+1}^{N_1} \la_q \de_{q+1,n+1}^\frac12 \la_{q+1}^{\frac12 +N_3}\\
        &\lesssim \lambda_q |h| \delta_{q+1,n}^{\frac12} \lambda_{q+1}^{\frac12}\Big(1+|h|^2\lambda_{q+1}^2\Big)\Big(1 + |h|^{2N}\lambda_{q+1}^{2N}\Big)\lambda_{q+1}^N \\
        & \lesssim \lambda_q|h| \delta_{q+1,n}^{\frac12} \lambda_{q+1}^{\frac12}\Big(1+|h|^{2(N+1)}\lambda_{q+1}^{2(N+1)}\Big)\lambda_{q+1}^N.
    \end{align*}
    Of course, the same holds for $Y_{\lambda_{q+1}}^{-\xi}$. With $\eta, \zeta \in \{\xi,-\xi\}$, we estimate $Y^{\eta,\zeta}_{\la_{q+1}}(x,h_1,h_2)$ as
    \begin{align*}
        [Y^{\eta,\zeta}_{\la_{q+1}}(\cdot,h_1,h_2)]_N &\lesssim \sum_{N_1+N_2=N} ([Y^\eta_{\la_{q+1}}(\cdot,h_1)]_{N_1}+[Y^\zeta_{\la_{q+1}}(\cdot,h_2)]_{N_1})\|\bar a_{\xi,k,n}\|_{N_2} + [Y^\eta_{\la_{q+1}}(\cdot,h_1)]_{N_1} [Y^\zeta_{\la_{q+1}}(\cdot,h_2)]_{N_2}\\
        &\lesssim \delta_{q+1,n} \lambda_{q+1} \lambda_q(|h_1| + |h_2| + \lambda_q |h_1||h_2|)\\
        & \qquad (1+|h_1|^{2(N+1)}\lambda_{q+1}^{2(N+1)} + |h_2|^{2(N+1)}\lambda_{q+1}^{2(N+1)} + |h_1|^{2(N+1)}|h_2|^{2(N+1)}\lambda_{q+1}^{4(N+1)} ) \lambda_{q+1}^N
    \end{align*}
    Finally, we can estimate $\de B^{\eta,\zeta}_{\la_{q+1}}(x)$, using crucially~\eqref{est.K}, as
    \begin{align*}
        \| \de B^{\eta,\zeta}_{\la_{q+1}} \|_N &\lesssim \sum_{N_1+...+N_4=N} \|e^{i\lambda\widetilde\Phi_k \cdot (\eta+\zeta)}\|_{N_1} \int \|e^{i\lambda\nabla\widetilde\Phi_k^T\eta\cdot h_1}\|_{N_2}\|e^{i\lambda\nabla\widetilde\Phi_k^T\zeta\cdot h_2}\|_{N_3} |K_{\lambda_{q+1}}(h_1,h_2)|\\
        & \qquad \qquad \qquad \|Y^{\eta,\zeta}_{\la_{q+1}}(x,h_1,h_2)\|_{N_4}\, dh_1dh_2\\
        &\lesssim \frac{\la_q \de_{q+1,n}}{\la_{q+1}} \la_{q+1}^N
    \end{align*}
    Summing over $\eta,\zeta \in \{\xi,-\xi\}$ proves the first estimate of the lemma.

    The second estimate on the material derivative is proved following similar steps. We note that 
    \begin{eqnarray*}
        \tilde D_{t,\Gamma} R_{\widetilde \Phi_k}(x,h)&=& \sum_{j,k=1}^2 h^j h^k \int_0^1 (1-s)(\tilde D_{t,\Gamma} \partial_i \partial_j \widetilde \Phi_k)(x-sh) ds \\ 
        && \qquad + h^j h^k \int_0^1 (1-s)(\tilde u_{q,\Gamma}(x-sh) - \tilde u_{q, \Gamma}(x))\cdot \nabla \partial_i\partial_j \widetilde \Phi_k(x-sh) ds.
    \end{eqnarray*}
    Therefore, using the mean-value inequality, we obtain the estimate
    \begin{align*}
        \left\|\tilde D_{t, \Gamma}R_{\widetilde \Phi_k}(\cdot,h)\right\|_N 
        &\lesssim |h|^2 \|\tilde D_{t, \Gamma} \na^2 \widetilde\Phi_k\|_N + |h|^3\left(\|\tilde u_{q,\Gamma}\|_1 \|\nabla^3 \widetilde \Phi_k\|_N + \|\tilde u_{q,\Gamma}\|_{N+1} \|\nabla^3 \widetilde \Phi_k\|_0\right) \\
        &\lesssim |h|^2 \left(\|\tilde D_{t, \Gamma} \na\widetilde\Phi_k\|_{N+1} +  \|\tilde u_{q, \Gamma}\|_{N+1}\| \na^2\widetilde\Phi_k\|_{0} +\|\tilde u_{q, \Gamma}\|_{1}\| \na^2\widetilde\Phi_k\|_{N}  \right)\\
        & \qquad +|h|^3 \left(\|\tilde u_{q,\Gamma}\|_1 \|\nabla^3 \widetilde \Phi_k\|_N + \|\tilde u_{q,\Gamma}\|_{N+1} \|\nabla^3 \widetilde \Phi_k\|_0\right)\\
        &\lesssim |h|^2 \de_q^{\frac12} \la_q^{\frac52} (1+|h| \lambda_{q+1})\la_{q+1}^N,
    \end{align*}
    where we have used the estimates of corollary~\ref{Flow_gam_estim}. Similarly, we have
    \begin{align*}
        \left\|\na_h \tilde D_{t, \Gamma} R_{\widetilde \Phi_k}(\cdot,h)\right\|_N &\lesssim |h| \|\tilde D_{t, \Gamma} \na^2 \widetilde\Phi_k\|_N + |h|^2 \| \tilde D_{t, \Gamma} \na^2 \widetilde\Phi_k\|_{N+1}\\
        &\qquad +|h|^2(\|\tilde u_{q,\Gamma}\|_{N+1} \|\na^3 \widetilde \Phi_k\|_0 +\|\tilde u_{q,\Gamma}\|_{1} \|\na^3 \widetilde \Phi_k\|_N )\\
        &\qquad +|h|^3(\|\tilde u_{q,\Gamma}\|_{N+1} \|\na^4 \widetilde \Phi_k\|_0 +\|\tilde u_{q,\Gamma}\|_{1} \|\na^4 \widetilde \Phi_k\|_N )\\
        &\lesssim |h| \de_q^{\frac12} \la_q^{\frac52}(1+|h|^2\la_{q+1}^2) \la_{q+1}^N
    \end{align*}
Following the same reasoning, we have
\begin{equation*}
    \|\tilde D_{t,\Gamma}\bar a_{\xi,k,n}(\cdot - h)\|_{N} \lesssim \delta_{q+1,n}^{\frac12}\lambda_{q+1}^{\frac12} \delta_q^{\frac12} \lambda_q^{\frac32}(1+|h|\lambda_{q+1})\lambda_{q+1}^N,
\end{equation*}
and 
\begin{equation*}
    \|\nabla_h \tilde D_{t,\Gamma}a(\cdot -h)\|_N \lesssim \delta_{q+1,n}^{\frac12}\lambda_{q+1}^{\frac12} \delta_q^{\frac12} \lambda_q^{\frac52}(1+|h|\lambda_{q+1})\lambda_{q+1}^{N}.
\end{equation*}
We continue, as before, by using Taylor's formula 
\begin{align*}
    \tilde D_{t,\Ga} Y^\xi_{\la_{q+1}}(x,h) &= \tilde D_{t,\Ga} \int_0^1 \frac{d}{dr} \left( \bar a_{\xi,k,n}(x-rh)e^{i\la_{q+1} R_{\widetilde \Phi_k}(x,rh)\cdot\xi} \right)\,dr\\
    &= \int_0^1 \frac{d}{dr} \tilde D_{t,\Ga} \left( \bar a_{\xi,k,n}(x-rh)e^{i\la_{q+1} R_{\widetilde \Phi_k}(x,rh)\cdot\xi} \right)\,dr
\end{align*}
    Now one can estimate $\tilde D_{t,\Ga} Y^\xi_{\la_{q+1}}$, using the above estimates, corollary~\ref{Flow_gam_estim} and lemma~\ref{le-amp-reg}, as
    \begin{align*}
        \|\tilde D_{t,\Ga} Y^\xi_{\la_{q+1}}(\cdot,h)\|_N &\lesssim \sup_{0\leq r\leq 1} \left\|\frac{d}{dr} \tilde D_{t,\Ga}\left(e^{i\lambda_{q+1}R_{\widetilde \Phi_k}(\cdot,rh) \cdot \xi}\bar a_{\xi,k,n}(\cdot-rh)\right) \right\|_N\\
        &\lesssim \sup_{0\leq r\leq 1} \left\| \frac{d}{dr} \left(e^{i\lambda_{q+1}R_{\widetilde \Phi_k}(\cdot,rh)\cdot\xi} i \lambda_{q+1} \tilde D_{t,\Ga}R_{\widetilde \Phi_k}(\cdot,rh)\cdot\xi\,\, \bar a_{\xi,k,n}(\cdot-rh)\right) \right\|_N\\
        &\qquad + \sup_{0\leq r\leq 1} \left\| \frac{d}{dr} \left( e^{i\lambda_{q+1}R_{\widetilde \Phi_k}(\cdot,rh)\cdot\xi}\,  \tilde D_{t,\Ga} \bar a_{\xi,k,n}(\cdot-rh)\right) \right\|_N\\
        &\lesssim \sup_{0\leq r\leq 1} \left\|e^{i\lambda_{q+1}R_{\widetilde \Phi_k}(\cdot,rh)\cdot\xi}\lambda_{q+1}^2|h| \nabla_h R_{\widetilde \Phi_k}(\cdot,rh) \tilde D_{t,\Ga}R_{\widetilde \Phi_k}(\cdot,rh)\cdot\xi\,\, \bar a_{\xi,k,n}(\cdot-rh) \right\|_N \\
        &\qquad + \sup_{0\leq r\leq 1}\left\|e^{i\lambda_{q+1}R_{\widetilde \Phi_k}(\cdot,rh)\cdot\xi} \lambda_{q+1}|h| \nabla_h \tilde D_{t,\Ga}R_{\widetilde \Phi_k}(\cdot,rh)\cdot\xi\,\, \bar a_{\xi,k,n}(\cdot-rh) \right\|_N\\
        &\qquad + \sup_{0\leq r\leq 1}\left\|e^{i\lambda_{q+1}R_{\widetilde \Phi_k}(\cdot,rh)\cdot\xi} \lambda_{q+1} |h| \tilde D_{t,\Ga}R_{\widetilde \Phi_k}(\cdot,rh)\cdot\xi\,\, \na\bar a_{\xi,k,n}(\cdot-rh) \right\|_N \\
        &\qquad + \sup_{0\leq r\leq 1} \left\| e^{i\lambda_{q+1}R_{\widetilde \Phi_k}(\cdot,rh)\cdot\xi}\, \lambda_{q+1}|h| \nabla_h R_{\widetilde \Phi_k}(\cdot,rh) \tilde D_{t,\Ga} \bar a_{\xi,k,n}(\cdot-rh) \right\|_N \\
        &\qquad + \sup_{0\leq r\leq 1} \left\| e^{i\lambda_{q+1}R_{\widetilde \Phi_k}(\cdot,rh)\cdot\xi}\, |h| \nabla_h \tilde D_{t,\Ga} \bar a_{\xi,k,n}(\cdot-rh) \right\|_N \\
        &\lesssim \lambda_{q+1}^2|h|\Big(\left\|e^{i\lambda_{q+1}R_{\widetilde \Phi_k}(\cdot,rh)\cdot\xi}\right\|_N \left\|\nabla_h R_{\widetilde \Phi_k}(\cdot,rh)\right\|_0 \left\|\tilde D_{t,\Ga}R_{\widetilde \Phi_k}(\cdot,rh)\right\|_0\left\|\bar a_{\xi,k,n} \right\|_0 \\
        & \qquad + \left\|e^{i\lambda_{q+1}R_{\widetilde \Phi_k}(\cdot,rh)\cdot\xi}\right\|_0 \left\|\nabla_h R_{\widetilde \Phi_k}(\cdot,rh)\right\|_N \left\|\tilde D_{t,\Ga}R_{\widetilde \Phi_k}(\cdot,rh)\right\|_0\left\|\bar a_{\xi,k,n} \right\|_0 \\
        & \qquad + \left\|e^{i\lambda_{q+1}R_{\widetilde \Phi_k}(\cdot,rh)\cdot\xi}\right\|_0 \left\|\nabla_h R_{\widetilde \Phi_k}(\cdot,rh)\right\|_0 \left\|\tilde D_{t,\Ga}R_{\widetilde \Phi_k}(\cdot,rh)\right\|_N\left\|\bar a_{\xi,k,n} \right\|_0 \\
         & \qquad + \left\|e^{i\lambda_{q+1}R_{\widetilde \Phi_k}(\cdot,rh)\cdot\xi}\right\|_0 \left\|\nabla_h R_{\widetilde \Phi_k}(\cdot,rh)\right\|_0 \left\|\tilde D_{t,\Ga}R_{\widetilde \Phi_k}(\cdot,rh)\right\|_0\left\|\bar a_{\xi,k,n} \right\|_N \Big)\\
         & +\lambda_{q+1}|h| \Big( \left\|e^{i\lambda_{q+1}R_{\widetilde \Phi_k}(\cdot,rh)\cdot\xi}\right\|_N\left\|\na_h\tilde D_{t,\Ga}R_{\widetilde \Phi_k}(\cdot,rh)\right\|_0 \left\|\bar a_{\xi,k,n} \right\|_0\\
         & \qquad + \left\|e^{i\lambda_{q+1}R_{\widetilde \Phi_k}(\cdot,rh)\cdot\xi}\right\|_0\left\|\na_h\tilde D_{t,\Ga}R_{\widetilde \Phi_k}(\cdot,rh)\right\|_N \left\|\bar a_{\xi,k,n} \right\|_0\\
         &\qquad + \left\|e^{i\lambda_{q+1}R_{\widetilde \Phi_k}(\cdot,rh)\cdot\xi}\right\|_0\left\|\na_h\tilde D_{t,\Ga}R_{\widetilde \Phi_k}(\cdot,rh)\right\|_0 \left\|\bar a_{\xi,k,n} \right\|_N \Big) \\ 
         & + \lambda_{q+1}|h|\Big(\left\|e^{i\lambda_{q+1}R_{\widetilde \Phi_k}(\cdot,rh)\cdot\xi}\right\|_N  \left\|\tilde D_{t,\Ga}R_{\widetilde \Phi_k}(\cdot,rh)\right\|_0\left\|\nabla \bar a_{\xi,k,n} \right\|_0 \\ 
         & \qquad + \left\|e^{i\lambda_{q+1}R_{\widetilde \Phi_k}(\cdot,rh)\cdot\xi}\right\|_0  \left\|\tilde D_{t,\Ga}R_{\widetilde \Phi_k}(\cdot,rh)\right\|_N\left\|\nabla \bar a_{\xi,k,n} \right\|_0 \\ 
         & \qquad + \left\|e^{i\lambda_{q+1}R_{\widetilde \Phi_k}(\cdot,rh)\cdot\xi}\right\|_0  \left\|\tilde D_{t,\Ga}R_{\widetilde \Phi_k}(\cdot,rh)\right\|_0\left\|\nabla \bar a_{\xi,k,n} \right\|_N\Big) \\ 
         &+ \lambda_{q+1}|h| \Big(\left\|e^{i\lambda_{q+1}R_{\widetilde \Phi_k}(\cdot,rh)\cdot\xi}\right\|_N \left\|\nabla_h R_{\widetilde \Phi_k}(\cdot,rh)\right\|_0 \left\|\tilde D_{t,\Gamma}\bar a_{\xi,k,n} \right\|_0 \\
         & \qquad + \left\|e^{i\lambda_{q+1}R_{\widetilde \Phi_k}(\cdot,rh)\cdot\xi}\right\|_0 \left\|\nabla_h R_{\widetilde \Phi_k}(\cdot,rh)\right\|_N \left\|\tilde D_{t,\Gamma}\bar a_{\xi,k,n} \right\|_0 \\
         & \qquad + \left\|e^{i\lambda_{q+1}R_{\widetilde \Phi_k}(\cdot,rh)\cdot\xi}\right\|_0 \left\|\nabla_h R_{\widetilde \Phi_k}(\cdot,rh)\right\|_0 \left\|\tilde D_{t,\Gamma}\bar a_{\xi,k,n} \right\|_N\Big) \\
         & + |h| \Big( \left\|e^{i\lambda_{q+1}R_{\widetilde \Phi_k}(\cdot,rh)\cdot\xi}\right\|_N\left\|\na_h \tilde D_{t,\Gamma}\bar a_{\xi,k,n} \right\|_0 + \left\|e^{i\lambda_{q+1}R_{\widetilde \Phi_k}(\cdot,rh)\cdot\xi}\right\|_0\left\|\na_h \tilde D_{t,\Gamma}\bar a_{\xi,k,n} \right\|_N\Big)\\
        &\lesssim \delta_{q+1,n}^{\frac12} \lambda_{q+1}^{\frac12}\delta_q^{\frac12}\lambda_q^{\frac32} |h|\lambda_q (1+ \lambda_{q+1}^{2N+5}|h|^{2N+5})\lambda_{q+1}^N.
    \end{align*}
    This leads to the following estimate $Y^{\eta,\zeta}_{\la_{q+1}}(x,h_1,h_2)$:
    \begin{align*}
        \|\tilde D_{t,\Ga} Y^{\eta,\zeta}_{\la_{q+1}}(\cdot,h_1,h_2)\|_N &\lesssim \delta_{q+1,n}\lambda_{q+1} \delta_q^{\frac12}\lambda_q^{\frac32} \lambda_q(|h_1|+|h_2|+\lambda_q|h_1||h_2|)\\
        &\qquad (1+|h_1|^{2N+5}\lambda_{q+1}^{2N+5}+|h_2|^{2N+5}\lambda_{q+1}^{2N+5} + |h_1|^{2N+5}|h_2|^{2N+5}\lambda_{q+1}^{4N+10}).
    \end{align*}
    In view of  
    \begin{align*}
        \| \tilde D_{t, \Gamma} e^{i\lambda_{q+1}\nabla\widetilde\Phi_k ^T\xi\cdot h} \|_N \lesssim \lambda_{q+1}|h|\left(\|\tilde D_{t,\Gamma} \nabla \widetilde \Phi_k\|_N + \|e^{i\lambda_{q+1}\nabla\widetilde\Phi_k ^T\xi\cdot h}\|_N \|\tilde D_{t,\Gamma} \nabla \widetilde \Phi_k\|_0\right) \lesssim \la_{q+1}|h| \de_q^{\frac12} \la_q^{\frac32}(1+ \la_{q+1}^{N} |h|^N)\la_{q+1}^N  \,,
    \end{align*}
    we finally obtain 
    \begin{equation*}
        \|\tilde D_{t,\Gamma} \de B^{\eta,\zeta}_{\la_{q+1}}\|_N \lesssim \frac{\delta_{q+1,n} \delta_q^{1/2}\lambda_q^{3/2}\lambda_q}{\lambda_{q+1}}. 
    \end{equation*}
    This, together with the fact that the temporal frequency of the profile functions $\{g_{\xi,k,n}\}$ is $\mu_{q+1}> \delta_q^{1/2}\lambda_q^{3/2}$, implies the final claimed estimate.
\end{proof}

\subsubsection{Estimates of $R_{q+1, R}$}
Recall
\[R_{q+1,R}=R_{q,\Gamma}+P_{q+1,\Gamma}+\div^{-1}\nabla^{\perp} \Delta^{-1}\div\left(T[\theta_{q+1}^{(p)}](\theta_{q,\Ga}-\tilde\theta_{q,\Ga})+T[\theta_{q,\Ga}-\tilde\theta_{q,\Ga}]\theta_{q+1}^{(p)}\right)\]
with
\begin{equation}\notag
\begin{split}
P_{q+1,\Gamma}=&\  \div^{-1} \nabla^{\perp} \Delta^{-1}\div\left(T[\theta_{q+1}^{(t)}] \theta_{q+1}^{(t)} \right)+R_q-R_{q,0}\\
&+\div^{-1} \nabla^{\perp} \Delta^{-1}\div\left(T[\theta_{q+1}^{(t)}] (\theta_{q}-\bar\theta_{q})+T[\theta_{q}-\bar\theta_{q}]\theta_{q+1}^{(t)} \right).
\end{split}
\end{equation}

\begin{lem}\label{le-est-glue}
The final gluing error satisfies the estimates
\begin{equation}\notag
\begin{split}
\|R_{q,\Gamma}\|_N\lesssim&\ \frac{\delta_{q+1}\lambda_q}{\lambda_{q+1}}\lambda_{q+1}^N, \ \ \forall N \geq 0,\\
\|\tilde D_{t,\Gamma}R_{q,\Gamma}\|_N\lesssim&\ \tau_q^{-1}\frac{\delta_{q+1}\lambda_q}{\lambda_{q+1}}\lambda_{q+1}^N, \ \ \forall N \geq 0,
\end{split}
\end{equation}
with implicit constants depending on $\Gamma, M,N$ and $\alpha$.
\end{lem}
\begin{proof}
It follows from proposition \ref{NewIter} and the fact that
\[\delta_{q+1,\Gamma}=\delta_{q+1}\left(\frac{\lambda_q}{\lambda_{q+1}}\right)^{(\frac12-\beta)\Gamma}\leq \delta_{q+1}\frac{\lambda_q}{\lambda_{q+1}} \]
that, for all $N \geq 0$,
\[\|R_{q,\Gamma}\|_N \lesssim \frac{\delta_{q+1}\lambda_q}{\lambda_{q+1}}\lambda_{q+1}^N.\]
For the material derivative we write
\[\tilde D_{t,\Gamma}R_{q,\Gamma}=\bar D_t R_{q,\Gamma}+ \tilde w_{q+1}^{(t)}\cdot\nabla R_{q,\Gamma}.\]
It follows from proposition \ref{NewIter} and the definition of $\delta_{q+1,\Gamma}$ that
\[\|\bar D_t R_{q,\Gamma}\|_N \lesssim \tau_q^{-1}\frac{\delta_{q+1}\lambda_q}{\lambda_{q+1}}\lambda_{q+1}^N.\]
On the other hand, applying proposition \ref{NewIter} and lemma \ref{w_t_estim}, we obtain
\begin{equation}\notag
\|\tilde w_{q+1}^{(t)}\cdot\nabla R_{q,\Gamma}\|_N\lesssim \|\tilde w_{q+1}^{(t)}\|_N\|R_{q,\Gamma}\|_1+\|\tilde w_{q+1}^{(t)}\|_0\|R_{q,\Gamma}\|_{N+1} \lesssim \frac{\delta_{q+1}\lambda_q^{2}\ell_q^{-\alpha}}{\mu_{q+1}}\frac{\delta_{q+1}\lambda_q^2}{\lambda_{q+1}}\lambda_{q+1}^N.
\end{equation}
The wanted estimate follows upon noting that
\[\frac{\delta_{q+1}\lambda_q^{3}\ell_q^{-\alpha}}{\mu_{q+1}}\leq \tau_q^{-1}.\]
\end{proof}

\begin{lem}\label{le-Newton-est}
The Newton error satisfies the estimates
\begin{equation}
\|\div^{-1} \nabla^{\perp} \Delta^{-1}\div\left(T[\theta_{q+1}^{(t)}] \theta_{q+1}^{(t)} \right)\|_N\lesssim \frac{\delta_{q+1}\lambda_q}{\lambda_{q+1}}\lambda_{q+1}^N, \ \ \forall \ N\geq 0,
\end{equation}
\begin{equation}
\|\tilde D_{t,\Gamma}\div^{-1} \nabla^{\perp} \Delta^{-1}\div\left(T[\theta_{q+1}^{(t)}]\theta_{q+1}^{(t)} \right)\|_N\lesssim \mu_{q+1}\frac{\delta_{q+1}\lambda_q}{\lambda_{q+1}}\lambda_{q+1}^N, \ \ \forall \ N\geq 0,
\end{equation}
where the implicit constants depend on $\Gamma, M$ and $N$.
\end{lem}
\begin{proof}
Recall that
\begin{equation}\notag
\theta_{q+1}^{(t)}=\sum_{n=0}^{\Gamma-1}\theta_{q+1,n+1}^{(t)}=\sum_{n=0}^{\Gamma-1}\sum_{k\in \mathbb Z_{q,n}}\tilde \chi_k(t)\theta_{k,n+1}
=\sum_{n=0}^{\Gamma-1}\sum_{k\in \mathbb Z_{q,n}}\tilde \chi_k(t)\Delta\psi_{k,n+1}=:\Delta \psi_{q+1}^{(t)},
\end{equation}
with 
\[\psi_{q+1}^{(t)}=\sum_{n=0}^{\Gamma-1}\sum_{k\in \mathbb Z_{q,n}}\tilde \chi_k(t)\psi_{k,n+1}.\]
Since $\div^{-1}\nabla^{\perp}$ is a zero order Fourier multiplier, we infer, using lemma \ref{le-bilinear-S}, that
\begin{equation}\notag
\begin{split}
\|\div^{-1} \nabla^{\perp} \Delta^{-1}\div\left(T[\theta_{q+1}^{(t)}] \theta_{q+1}^{(t)} \right)\|_N\lesssim&\ \|\Delta^{-1}\div\left(T[\Delta\psi_{q+1}^{(t)}] \theta_{q+1}^{(t)}+T[\theta_{q+1}^{(t)}] \Delta\psi_{q+1}^{(t)} \right)\|_{N+\alpha}\\
\lesssim&\ \|\psi_{q+1}^{(t)}\|_{N+1+\alpha}\|\theta_{q+1}^{(t)}\|_\alpha+\|\psi_{q+1}^{(t)}\|_{1+\alpha}\|\theta_{q+1}^{(t)}\|_{N+\alpha}\\
\lesssim&\ \frac{\delta_{q+1}\lambda_{q}\ell_q^{-\alpha}}{\mu_{q+1}}\cdot\frac{\delta_{q+1}\lambda_q^2\ell_q^{-2\alpha}}{\mu_{q+1}} \lambda_{q+1}^{N+1}
\lesssim\ \frac{\delta_{q+1}\lambda_q}{\lambda_{q+1}}\lambda_{q+1}^N,
\end{split}
\end{equation}
where, for the final line, we use lemmas \ref{psi_estim} and \ref{w_t_estim}, as well as the definition of $\mu_{q+1}$. 

To estimate the material derivative, we write 
\begin{eqnarray*}
    \tilde D_{t, \Gamma} \div^{-1} \nabla^{\perp} \Delta^{-1}\div\left(T[\theta_{q+1}^{(t)}] \theta_{q+1}^{(t)} \right) &=& \underbrace{\tilde u_{q, \Gamma} \cdot \nabla \div^{-1} \nabla^{\perp} \Delta^{-1}\div\left(T[\theta_{q+1}^{(t)}] \theta_{q+1}^{(t)} \right)}_{T_1}  \\
    && + \underbrace{\div^{-1} \nabla^{\perp} \Delta^{-1}\div \left( T[\bar D_t \theta_{q+1}^{(t)}]\theta_{q+1}^{(t)} + T[\theta_{q+1}^{(t)}] \bar D_t \theta_{q+1}^{(t)} \right)}_{T_2} \\ 
    && -  \underbrace{\div^{-1} \nabla^{\perp} \Delta^{-1}\div \left( T[\bar u_q \cdot \nabla \theta_{q+1}^{(t)}]\theta_{q+1}^{(t)} + T[\theta_{q+1}^{(t)}] \bar u_{q} \cdot \nabla \theta_{q+1}^{(t)} \right)}_{T_3},
\end{eqnarray*}
and estimate each term separately. We have 
\begin{equation*}
    \|T_1\|_N \lesssim \|\tilde u_{q, \Gamma}\|_N \|\div^{-1} \nabla^{\perp} \Delta^{-1}\div\left(T[\theta_{q+1}^{(t)}] \theta_{q+1}^{(t)} \right)\|_1 + \|\tilde u_{q, \Gamma}\|_0 \|\div^{-1} \nabla^{\perp} \Delta^{-1}\div\left(T[\theta_{q+1}^{(t)}] \theta_{q+1}^{(t)} \right)\|_{N+1}. 
\end{equation*}
Using corollary \ref{Gamma_velo_estim}, and arguing as above for the bilinear term, we obtain 
\begin{equation*}
    \|T_1\|_N \lesssim \delta_q^{1/2} \lambda_q^{1/2} \frac{\delta_{q+1} \lambda_q^2}{\lambda_{q+1}} \lambda_{q+1}^N \lesssim \tau_q^{-1}\frac{\delta_{q+1} \lambda_q}{\lambda_{q+1}} \lambda_{q+1}^N.
\end{equation*}
For the second term, we once again use lemma \ref{le-bilinear-S} together with the fact that $\div^{-1} \nabla^\perp$ is a zero order operator: 
\begin{eqnarray*}
    \|T_2\|_N &\lesssim& \|\Delta^{-1}\div\left(T[\Delta\psi_{q+1}^{(t)}] \bar D_t \theta_{q+1}^{(t)}+T[\bar D_t \theta_{q+1}^{(t)}]\Delta\psi_{q+1}^{(t)} \right)\|_{N+\alpha} \\ 
    & \lesssim & \|\psi_{q+1}^{(t)}\|_{N+1+\alpha}\|\bar D_t \theta_{q+1}^{(t)}\|_\alpha+\|\psi_{q+1}^{(t)}\|_{1+\alpha}\|\bar D_t \theta_{q+1}^{(t)}\|_{N+\alpha} \\ 
    & \lesssim & \mu_{q+1} \frac{\delta_{q+1} \lambda_q}{\lambda_{q+1}} \lambda_{q+1}^N.
\end{eqnarray*}
Similarly, 
\begin{eqnarray*}
    \|T_3\|_N &\lesssim& \|\Delta^{-1}\div\left(T[\Delta\psi_{q+1}^{(t)}] \bar u_q \cdot \nabla \theta_{q+1}^{(t)}+T[\bar u_q \cdot \nabla \theta_{q+1}^{(t)}]\Delta\psi_{q+1}^{(t)} \right)\|_{N+\alpha} \\ 
    & \lesssim & \|\psi_{q+1}^{(t)}\|_{N+1+\alpha}\|\bar u_q \cdot \nabla \theta_{q+1}^{(t)}\|_\alpha+\|\psi_{q+1}^{(t)}\|_{1+\alpha}\|\bar u_q \cdot \nabla \theta_{q+1}^{(t)}\|_{N+\alpha} \\ 
    & \lesssim & \tau_q^{-1} \frac{\delta_{q+1} \lambda_q}{\lambda_{q+1}} \lambda_{q+1}^N.
\end{eqnarray*}
The conclusion follows.
\end{proof}

\begin{lem}\label{le-molli-est}
The spatial mollification error has the estimates
\begin{equation}
\|R_q-R_{q,0}\|_N\lesssim \frac{\delta_{q+1}\lambda_q}{\lambda_{q+1}}\lambda_{q+1}^N, \ \ N\in \{0,1, ... , L_R\},
\end{equation}
\begin{equation}
\|\tilde D_{t,\Gamma}(R_q-R_{q,0})\|_N\lesssim \delta_{q+1}\delta_q^{\frac12}\lambda_q^{\frac32}\lambda_{q+1}^N, \ \ N\in \{0,1, ... , L_t\},
\end{equation}
with the implicit constants depending on $\Gamma, M, N$ and $\alpha$.
\end{lem}
\begin{proof}
Applying the mollifier estimate in proposition \ref{prop-molli} for $N\leq L_R-2$, we obtain
\[\|R_q-R_{q,0}\|_N\lesssim \ell_q^2\|R_q\|_{N+2}\lesssim \ell_q^2\delta_{q+1}\lambda_q^{N+2}\lesssim \frac{\delta_{q+1}\lambda_q}{\lambda_{q+1}}\lambda_q^N.\]
In the case $N=L_R-1,L_R$, we have
\[\|R_q-R_{q,0}\|_N\lesssim \|R_q\|_{N}+\|R_{q,0}\|_N \lesssim \delta_{q+1}\lambda_q^{N}\lesssim \frac{\delta_{q+1}\lambda_q}{\lambda_{q+1}}\lambda_{q+1}^N,\]
as long as $L_R \geq 2$. Regarding the material derivative, we first write
\begin{equation}\notag
\tilde D_{t,\Gamma}(R_q-R_{q,0})= D_tR_q+(\bar u_q-u_q)\cdot\nabla R_q-\bar D_t R_{q,0}+\tilde w_{q+1}^{(t)}\cdot \nabla(R_q-R_{q,0}).
\end{equation}
Then, for $N \leq L_t$,
\begin{equation}\notag
\|D_tR_q\|_N+\|\bar D_tR_{q,0}\|_N\lesssim \delta_{q+1}\delta_q^{\frac12}\lambda_q^{\frac32}\lambda_{q+1}^N;
\end{equation}
and for $N \leq \min\{L_\theta - 2, L_R -1\}$,
\begin{equation}\notag
\begin{split}
\|(\bar u_q-u_q)\cdot\nabla R_q\|_N\lesssim&\ \|\bar u_q-u_q\|_N\|\nabla R_q\|_0+\|\bar u_q-u_q\|_0\|\nabla R_q\|_N\\
\lesssim&\ \ell_q^2\delta_q^{\frac12}\lambda_q^{2+\frac12}\delta_{q+1}\lambda_q \lambda_{q+1}^N \\
\lesssim&\ \delta_{q+1}\delta_q^{\frac12}\lambda_q^{\frac32}\lambda_{q+1}^N;
\end{split}
\end{equation}
\begin{equation}\notag
\begin{split}
\|\tilde w_{q+1}^{(t)}\cdot \nabla(R_q-R_{q,0})\|_N\lesssim &\ \|\tilde w_{q+1}^{(t)}\|_N \|\nabla(R_q-R_{q,0})\|_0+\|\tilde w_{q+1}^{(t)}\|_0 \|\nabla(R_q-R_{q,0})\|_N\\
\lesssim&\ \frac{\delta_{q+1}\lambda_q^{2}\ell_q^{-\alpha}}{\mu_{q+1}}\frac{\delta_{q+1} \lambda_q}{\lambda_{q+1}}\lambda_{q+1}^{N+1}\\
\lesssim&\ \delta_{q+1}\delta_q^{\frac12}\lambda_q^{\frac32}\lambda_{q+1}^N.
\end{split}
\end{equation}
Since we choose $L_t \leq \min\{L_\theta - 2, L_R -1\}$, the lemma is proven.
\end{proof}

In order to estimate the remaining term in $P_{q+1,\Ga}$, we note that
\[ \Delta^{-1}\div\left(T[\theta_{q+1}^{(t)}](\theta_{q}-\bar\theta_{q})+T[\theta_{q}-\bar\theta_{q}](\theta_{q+1}^{(t)})\right) = S[\psi_{q+1}^{(t)}, \theta_{q}-\bar\theta_{q}]\,. \]
Similarly we have that
\[ \Delta^{-1}\div\left(T[\theta_{q+1}^{(p)}](\theta_{q,\Ga}-\tilde\theta_{q,\Ga})+T[\theta_{q,\Ga}-\tilde\theta_{q,\Ga}](\theta_{q+1}^{(p)})\right) = S[\Delta^{-1}\theta_{q+1}^{(p)}, \theta_{q,\Ga}-\tilde\theta_{q,\Ga}]\,. \]

\begin{lem}
The following estimates 
    \begin{equation}
\begin{split}
\|\div^{-1}\nabla^{\perp} &S[\psi_{q+1}^{(t)}, \theta_{q}-\bar\theta_{q}]\|_N
\lesssim  \ \delta_{q+1}^{\frac12}\delta_q^{\frac12}\bigg(\frac{\lambda_q}{\lambda_{q+1}}\bigg)^{\frac{3}{2}}\lambda_{q+1}^{N}, \ \ N\in \{0,1, ... , L_R\},
\end{split}
\end{equation}
\begin{equation}
\begin{split}
\|\tilde D_{t, \Gamma}\div^{-1}\nabla^{\perp} &S[\psi_{q+1}^{(t)}, \theta_{q}-\bar\theta_{q}]\|_N
\lesssim 
\delta_{q+1}^{\frac32}\lambda_q\lambda_{q+1}^\frac12\lambda_{q+1}^{N+2\alpha},  \ \ N\in \{0,1, ... , L_t\},
\end{split}
\end{equation}
\begin{equation}
\begin{split}
\|\div^{-1}\nabla^{\perp} &S[\Delta^{-1}\theta_{q+1}^{(p)}, \theta_{q,\Ga}-\tilde\theta_{q,\Ga}]\|_N
\lesssim  \ \delta_{q+1}^{\frac12}\delta_q^{\frac12}\bigg(\frac{\lambda_q}{\lambda_{q+1}}\bigg)^{\frac{3}{2}}\lambda_{q+1}^{N+2\alpha}, \ \ N\in \{0,1, ... , L_R\},
\end{split}
\end{equation}
\begin{equation}
\begin{split}
\|\tilde D_{t, \Gamma}\div^{-1}\nabla^{\perp} &S[\Delta^{-1}\theta_{q+1}^{(p)}, \theta_{q,\Ga}-\tilde\theta_{q,\Ga}]\|_N
\lesssim 
\delta_{q+1}^{\frac32}\lambda_q\lambda_{q+1}^\frac12\lambda_{q+1}^{N+6\alpha},  \ \ N\in \{0,1, ... , L_t\}
\end{split}
\end{equation}
hold with implicit constants depending on $\Gamma, M,N$ and $\alpha$.
\end{lem}

\begin{proof}
For $N \leq L_\theta - 2$, proposition \ref{prop-molli} implies 
\begin{equation*}
    \|\theta_q - \bar \theta_q\|_N \lesssim \ell_q^2 \|\theta_q\|_{N+2} \lesssim \delta_q^{1/2} \lambda_q^{1/2} \frac{\lambda_q}{\lambda_{q+1}} \lambda_q^N\lesssim \delta_q^{1/2} \lambda_q^{1/2} \frac{\lambda_q}{\lambda_{q+1}} \lambda_{q+1}^N.
\end{equation*}
Using lemma \ref{le-bilinear-S} and the fact that $\div^{-1} \nabla^\perp$ is a zero-order Fourier multiplier, we obtain, for $N \leq L_\theta - 3$,
\begin{eqnarray*}
    \|\div^{-1}\nabla^{\perp} S[\psi_{q+1}^{(t)}, \theta_{q}-\bar\theta_{q}]\|_{N+\alpha} &\lesssim& \|S[\psi_{q+1}^{(t)}, \theta_{q}-\bar\theta_{q}]\|_{N+\alpha} \\ 
    &\lesssim& \|\psi_{q+1}^{(t)}\|_{N+1+\alpha} \|\theta_q - \bar \theta_q\|_{\alpha} + \|\psi_{q+1}^{(t)}\|_{1+\alpha} \|\theta_q - \bar \theta_q\|_{N+\alpha} \\ 
    &\lesssim& \frac{\delta_{q+1} \lambda_q \ell_q^{-\alpha}}{\mu_{q+1}} \delta_q^{1/2}\lambda_q^{1/2} \frac{\lambda_q}{\lambda_{q+1}} \lambda_{q+1}^{N+\alpha} \\ 
    &\lesssim& \delta_{q+1}^{\frac12}\delta_q^{\frac12}\bigg(\frac{\lambda_q}{\lambda_{q+1}}\bigg)^{\frac{3}{2}}\lambda_{q+1}^{N},
\end{eqnarray*}
and the first claimed estimate is proven, in view of the choice $L_\theta \geq L_R + 3$.

For the material derivative estimate, we write 
\begin{eqnarray*}
    \tilde D_{t, \Gamma} \div^{-1}\nabla^{\perp} S[\psi_{q+1}^{(t)}, \theta_{q}-\bar\theta_{q}] &=& \underbrace{\tilde u_{q, \Gamma} \cdot \nabla \div^{-1} \nabla^\perp S[\psi_{q+1}^{(t)}, \theta_q - \bar \theta_q]}_{T_1} \\ 
    && + \underbrace{\div^{-1}\nabla^{\perp} S[\bar D_t \psi_{q+1}^{(t)}, \theta_{q}-\bar\theta_{q}]}_{T_2} \\ 
    && + \underbrace{\div^{-1}\nabla^{\perp} S[\psi_{q+1}^{(t)}, D_t \theta_{q}-P_{\lesssim \ell_q^{-1}}D_t\theta_{q}]}_{T_3} \\ 
    && - \underbrace{\div^{-1}\nabla^{\perp} S[\bar u_q \cdot \nabla \psi_{q+1}^{(t)}, \theta_{q}-\bar\theta_{q}]}_{T_4} \\ 
    && - \underbrace{\div^{-1}\nabla^{\perp} S[\psi_{q+1}^{(t)}, u_q \cdot \nabla \theta_{q}-P_{\lesssim \ell_q^{-1}} (u_q \cdot \nabla \theta_{q})]}_{T_5}.
\end{eqnarray*}
We estimate each term separately: for $N \leq L_\theta - 4$,
\begin{eqnarray*}
    \|T_1\|_{N+\alpha} &\lesssim& \|\tilde u_{q, \Gamma}\|_{N+\alpha} \|S[\psi_{q+1}^{(t)}, \theta_q - \bar \theta_q]\|_{1+\alpha} + |\tilde u_{q, \Gamma}\|_{\alpha} \|S[\psi_{q+1}^{(t)}, \theta_q - \bar \theta_q]\|_{N+1+\alpha} \\ 
    &\lesssim& \delta_q^{1/2} \lambda_q^{1/2} \delta_{q+1}^{1/2} \delta_q^{1/2} \bigg(\frac{\lambda_q}{\lambda_{q+1}}\bigg)^{3/2} \lambda_{q+1}^{N+1+\alpha} \\ 
    & \lesssim & \delta_{q+1}^{1/2} \delta_q \lambda_q^{2} \lambda_{q+1}^{-1/2} \lambda_{q+1}^{N+\alpha}, 
\end{eqnarray*}
where we use the previously obtained estimate together with corollary \ref{Gamma_velo_estim}. Using lemma \ref{psi_estim}, we obtain 
\begin{eqnarray*}
    \|T_2\|_{N+\alpha} &\lesssim& \|\bar D_t \psi_{q+1}^{(t)}\|_{N+1+\alpha} \|\theta_q - \bar \theta_q\|_\alpha + \|\bar D_t \psi_{q+1}^{(t)}\|_{1+\alpha} \|\theta_q - \bar \theta_q\|_{N+\alpha} \\ 
    &\lesssim& \delta_{q+1}\delta_q^{1/2} \lambda_q^{1/2} \frac{\lambda_q}{\lambda_{q+1}} \lambda_{q+1}^{N+1+2\alpha} \\ 
    & \lesssim & \delta_{q+1} \delta_q^{1/2} \lambda_q^{3/2} \lambda_{q+1}^{N+2\alpha},
\end{eqnarray*}
which holds for $N \leq L_\theta - 3$. For the third term, we note that $D_t \theta_q = \nabla^\perp \div R_q$, and, thus, for $N \leq L_R - 3$,
\begin{eqnarray*}
    \|T_3\|_{N+\alpha} &\lesssim& \|\psi_{q+1}^{(t)}\|_{N+1+\alpha} \|R_q - R_{q,0}\|_{2+\alpha} + \|\psi_{q+1}^{(t)}\|_{1+\alpha} \|R_q - R_{q,0}\|_{N+2+\alpha} \\
    &\lesssim & \frac{\delta_{q+1}\lambda_q}{\mu_{q+1}} \frac{\delta_{q+1} \lambda_q}{\lambda_{q+1}} \lambda_{q+1}^{N+2+2\alpha} \\ 
    & \lesssim & \delta_{q+1}^{3/2} \lambda_q \lambda_{q+1}^{1/2} \lambda_{q+1}^N,
\end{eqnarray*}
where we use lemmas \ref{psi_estim} and \ref{le-molli-est}. For the fourth term, we first estimate 
\begin{eqnarray*}
    \|\bar u_q \cdot \nabla \psi_{q+1}^{(t)}\|_{N+\alpha} &\lesssim& \|\bar u_q\|_{N+\alpha} \|\psi_{q+1}^{(t)}\|_{1+\alpha} + \|\bar u_q\|_{\alpha} \|\psi_{q+1}^{(t)}\|_{N+1+\alpha} \\
    &\lesssim& \delta_q^{1/2} \lambda_q^{1/2} \frac{\delta_{q+1} \lambda_q}{\mu_{q+1}} \lambda_{q+1}^{N+\alpha}\\
    &\lesssim& \delta_q^{1/2} \delta_{q+1}^{1/2} \bigg(\frac{\lambda_q}{\lambda_{q+1}}\bigg)^{1/2} \lambda_{q+1}^N,
\end{eqnarray*}
from which we conclude that, for $N \leq L_\theta - 3$,
\begin{equation*}
    \|T_4\|_{N+\alpha} \lesssim \delta_q^{1/2} \delta_{q+1}^{1/2} \bigg(\frac{\lambda_q}{\lambda_{q+1}}\bigg)^{1/2} \delta_q^{1/2} \lambda_q^{1/2} \frac{\lambda_q}{\lambda_{q+1}} \lambda_{q+1}^{N+1+\alpha} \lesssim \delta_q \delta_{q+1}^{1/2} \lambda_q^2 \lambda_{q+1}^{-1/2} \lambda_{q+1}^{N+\alpha}.
\end{equation*}
For the final term, we use once again proposition \ref{prop-molli}, to obtain 
\begin{equation*}
    \|u_q \cdot \nabla \theta_q - P_{\lesssim \ell_q^{-1}}(u_q \cdot \nabla \theta_q)\|_N \lesssim \ell_q^{2}\|u_q \cdot \nabla \theta_q\|_{N+2} \lesssim \ell_q^2(\|u_q\|_{N+2}\|\theta_q\|_1 + \|u_q\|_{0}\|\theta_q\|_{N+3})\lesssim \delta_q \lambda_q^2 \frac{\lambda_q}{\lambda_{q+1}} \lambda_{q+1}^N,
\end{equation*}
for $N \leq L_\theta - 3$, which implies 
\begin{equation*}
    \|T_5\|_{N+\alpha} \lesssim \frac{\delta_{q+1}\lambda_q}{\mu_{q+1}}\delta_q \lambda_q^2 \frac{\lambda_q}{\lambda_{q+1}} \lambda_{q+1}^{N+2\alpha} \lesssim \delta_{q+1}^{1/2}\delta_q \lambda_q^3 \lambda_{q+1}^{-3/2}\lambda_{q+1}^N,
\end{equation*}
for $N \leq L_\theta - 4$. The first material derivative estimate follows, since we choose $L_t \leq \min\{L_\theta - 4, L_R - 3\}$.

We now turn to obtaining the two remaining estimates. Since we have sharp control over $L_t - 2 > 3$ derivatives of $\theta_{q+1}^{(t)}$, we have, for $N \leq 1$,
\begin{equation*}
    \|\theta_{q+1}^{(t)} - P_{\lesssim \ell_q^{-1}} \theta_{q+1}^{(t)}\|_N \lesssim \ell_q^2 \frac{\delta_{q+1} \lambda_q^{N+4}\ell_q^{-\alpha}}{\mu_{q+1}} \lesssim \frac{\delta_{q+1}^{1/2}\lambda_q^2}{\lambda_{q+1}^{3/2}}\lambda_{q+1}^N,
\end{equation*}
whereas for $N \geq 2$, 
\begin{equation*}
    \|\theta_{q+1}^{(t)} - P_{\lesssim \ell_q^{-1}} \theta_{q+1}^{(t)}\|_N \lesssim \|\theta_{q+1}^{(t)}\|_N \lesssim \delta_{q+1}^{1/2} \frac{\lambda_q}{\lambda_{q+1}^{1/2}} \bigg(\frac{\lambda_q}{\lambda_{q+1}}\bigg)^{N/2} \lambda_{q+1}^N \lesssim \frac{\delta_{q+1}^{1/2}\lambda_q^2}{\lambda_{q+1}^{3/2}}\lambda_{q+1}^N,
\end{equation*}
from which we conclude that this estimate holds for all values of $N$. We remark also that this is a better estimate than the one used for $\theta_q - \bar \theta_q$. Moreover, since $\theta_{q+1}^{(p)}$ is localized at frequency $\approx \lambda_{q+1}$, it holds that 
\begin{equation*}
    \|\Delta^{-1}\theta_{q+1}^{(p)}\|_N \lesssim \delta_{q+1}^{1/2} \lambda_{q+1}^{-3/2}\lambda_{q+1}^N.
\end{equation*}
We have, then, for $N \leq L_\theta - 3$,
\begin{eqnarray*}
    \|\div^{-1}\nabla^{\perp}S[\Delta^{-1}\theta_{q+1}^{(p)}, \theta_{q,\Ga}-\tilde\theta_{q,\Ga}]\|_{N+\alpha} &\lesssim& \delta_{q+1}^{1/2}\lambda_{q+1}^{-1/2} \delta_q^{1/2} \lambda_q^{1/2} \frac{\lambda_q}{\lambda_{q+1}} \lambda_{q+1}^{N+2\alpha},
\end{eqnarray*}
and the wanted estimate follows. 

For the final material derivative estimate, we write 
\begin{eqnarray*}
    \tilde D_{t, \Gamma}\div^{-1}\nabla^{\perp} S[\Delta^{-1}\theta_{q+1}^{(p)}, \theta_{q,\Ga}-\tilde\theta_{q,\Ga}] & = & \underbrace{\tilde u_{q, \Gamma} \cdot \nabla \div^{-1}\nabla^{\perp} S[\Delta^{-1}\theta_{q+1}^{(p)}, \theta_{q,\Ga}-\tilde\theta_{q,\Ga}]}_{T_1} \\ 
    && + \underbrace{\div^{-1}\nabla^{\perp} S[\tilde D_{t, \Gamma}\Delta^{-1}\theta_{q+1}^{(p)}, \theta_{q,\Ga}-\tilde\theta_{q,\Ga}]}_{T_2} \\ 
    && - \underbrace{\div^{-1}\nabla^{\perp} S[\tilde u_{q, \Gamma} \cdot \nabla \Delta^{-1}\theta_{q+1}^{(p)}, \theta_{q,\Ga}-\tilde\theta_{q,\Ga}]}_{T_3} \\ 
    && + \underbrace{\div^{-1}\nabla^{\perp} S[\Delta^{-1}\theta_{q+1}^{(p)}, D_t \theta_q - P_{\lesssim \ell_q^{-1}} D_t \theta_q]}_{T_4} \\ 
    && - \underbrace{\div^{-1}\nabla^{\perp} S[\Delta^{-1}\theta_{q+1}^{(p)}, u_q \cdot \nabla \theta_q - P_{\lesssim \ell_q^{-1}} (u_q \cdot \nabla \theta_q)]}_{T_5} \\ 
    && + \underbrace{\div^{-1}\nabla^{\perp} S[\Delta^{-1}\theta_{q+1}^{(p)}, \bar D_t \theta_{q+1}^{(t)} - P_{\lesssim \ell_q^{-1}} \bar D_t \theta_{q+1}^{(t)}]}_{T_6} \\ 
    && - \underbrace{\div^{-1}\nabla^{\perp} S[\Delta^{-1}\theta_{q+1}^{(p)}, \bar u_q \cdot \nabla \theta_{q+1}^{(t)} - P_{\lesssim \ell_q^{-1}}( \bar u_q \cdot \nabla \theta_{q+1}^{(t)})]}_{T_7}.
\end{eqnarray*}
We estimate each term separately. For $N \leq L_\theta - 4$, we have, as before, 
\begin{equation*}
    \|T_1\|_{N+\alpha} \lesssim \delta_q \delta_{q+1}^{1/2} \lambda_q^2 \lambda_{q+1}^{-1/2} \lambda_{q+1}^{N+3\alpha}.
\end{equation*}
For $T_2$, we argue similarly to the proof of lemma \ref{le-tran}: 
\begin{eqnarray*}
    \tilde D_{t, \Gamma} \Delta^{-1} \theta_{q+1}^{(p)} &=& \sum_{\xi, k, n} \Delta^{-1} P_{\approx \lambda_{q+1}} \big(\tilde D_{t, \Gamma}(g_{\xi, k, n+1} \bar a_{\xi, k, n}) \cos(\lambda_{q+1} \tilde \Phi_k \cdot \xi)\big) \\
    && + [\tilde D_{t, \Gamma}, \Delta^{-1}P_{\approx \lambda_{q+1}}] g_{\xi, k, n+1} \bar a_{\xi, k, n}\cos(\lambda_{q+1} \tilde \Phi_k \cdot \xi),
\end{eqnarray*}
and it follows that 
\begin{equation*}
    \|\tilde D_{t, \Gamma} \Delta^{-1} \theta_{q+1}^{(p)}\|_N \lesssim \mu_{q+1} \delta_{q+1}^{1/2} \lambda_{q+1}^{-3/2} \lambda_{q+1}^N \lesssim \delta_{q+1}\frac{\lambda_q}{\lambda_{q+1}}\lambda_{q+1}^{N+4\alpha}.
\end{equation*}
Consequently, for $N \leq L_\theta - 3$,
\begin{equation*}
    \|T_2\|_{N+\alpha} \lesssim \delta_{q+1} \lambda_q \delta_q^{1/2}\lambda_q^{1/2} \frac{\lambda_q}{\lambda_{q+1}} \lambda_{q+1}^{N+6\alpha} \lesssim \delta_{q+1}\delta_q^{1/2} \lambda_q^{5/2}\lambda_{q+1}^{-1} \lambda_{q+1}^{N+6\alpha}.
\end{equation*}
For the third term, we have 
\begin{equation*}
    \|\tilde u_{q, \Gamma} \cdot \nabla \Delta^{-1} \theta_{q+1}^{(p)}\|_{N} \lesssim \|\tilde u_{q, \Gamma}\|_N \|\Delta^{-1} \theta_{q+1}^{(p)}\|_1 + \|\tilde u_{q, \Gamma}\|_0 \|\Delta^{-1} \theta_{q+1}^{(p)}\|_{N+1} \lesssim \delta_q^{1/2} \delta_{q+1}^{1/2} \lambda_q^{1/2} \lambda_{q+1}^{-1/2} \lambda_{q+1}^N,
\end{equation*}
and, thus, for $N \leq L_\theta -3$,
\begin{equation*}
    \|T_3\|_{N+\alpha} \lesssim \delta_q \delta_{q+1}^{1/2} \lambda_q^2 \lambda_{q+1}^{-1/2} \lambda_{q+1}^{N+2\alpha}.
\end{equation*}
Terms $T_4$ and $T_5$ are treated similarly to $T_3$ and $T_5$ of the previous material derivative estimate. The estimates are: for $N \leq L_R - 3$,
\begin{equation*}
    \|T_4\|_{N+\alpha} \lesssim \delta_{q+1}^{3/2} \lambda_q \lambda_{q+1}^{1/2} \lambda_{q+1}^{N+2\alpha},
\end{equation*}
and, for $N \leq L_\theta - 4$, 
\begin{equation*}
    \|T_5\|_{N+\alpha} \lesssim \delta_{q+1}^{1/2} \delta_q \lambda_q^3 \lambda_{q+1}^{-3/2} \lambda_{q+1}^{N+2\alpha}.
\end{equation*}
From lemma \ref{w_t_estim} and proposition \ref{prop-molli}, we obtain 
\begin{equation*}
    \|\bar D_t \theta_{q+1}^{(t)} - P_{\lesssim \ell_q^{-1}} \bar D_t \theta_{q+1}^{(t)}\|_N \lesssim \delta_{q+1} \lambda_q^3 \lambda_{q+1}^{-1} \lambda_{q+1}^{N+\alpha}, \,\,\, \forall N \geq 0,
\end{equation*}
and, thus, for all $N \geq 0$,
\begin{equation*}
    \|T_6\|_{N+\alpha} \lesssim \delta_{q+1}^{3/2} \lambda_q^3 \lambda_{q+1}^{-3/2} \lambda_{q+1}^{N+3\alpha}.
\end{equation*}
The mollification estimate implies that for all $N \geq 0$, 
\begin{equation*}
    \|\bar u_q \cdot \nabla \theta_{q+1}^{(t)} - P_{\lesssim \ell_q^{-1}}(\bar u_q \cdot \nabla \theta_{q+1}^{(t)})\|_N \lesssim (\lambda_q \lambda_{q+1})^{-1} \frac{\delta_q^{1/2}\delta_{q+1} \lambda_q^{5 + 1/2}}{\mu_{q+1}} \lambda_{q+1}^{N+\alpha} \lesssim \delta_q^{1/2}\delta_{q+1}^{1/2}\lambda_q^{7/2} \lambda_{q+1}^{-3/2} \lambda_{q+1}^{N + \alpha},
\end{equation*}
and, thus, finally, 
\begin{equation*}
    \|T_7\|_{N+\alpha} \lesssim \delta_{q+1} \delta_q^{1/2} \lambda_q^{7/2} \lambda_{q+1}^{-2} \lambda_{q+1}^{N+3\alpha}.
\end{equation*}
This concludes the proof.
\end{proof}

\subsection{Conclusion} Lemma \ref{le-est-u-0} shows that the inductive propagation of the estimates concerning the density $\theta_{q+1}$ and velocity $u_{q+1}$. It remains to check the propagation of the estimates on the stress error $R_{q+1}$. In the following, we denote the material derivative corresponding to the vector field $u_{q+1}$ by  
\begin{equation*}
    D_{t,q+1} = \pa_t + u_{q+1}\cdot \na.
\end{equation*}
\begin{cor}
    The following hold:
    \begin{equation} \label{Spatial_estim_qplus1}
        \|R_{q+1}\|_N \leq \delta_{q+2} \lambda_{q+1}^{N-2\alpha}, \,\,\, \forall N \in \{0,1,...,L_R\},
    \end{equation}
    \begin{equation} \label{Material_estim_qplus1}
        \|D_{t, q+1} R_{q+1}\|_N \leq \delta_{q+2} \delta_{q+1}^{1/2} \lambda_{q+1}^{N+\frac32-2\alpha}, \,\,\, \forall N \in \{0,1,...,L_t\}.
    \end{equation}
\end{cor}
\begin{proof}
Collecting the estimates obtained in the previous sections we have 
    \begin{equation*}
        \|R_{q+1}\|_N \lesssim \bigg(\frac{\delta_{q+1}\lambda_q}{\lambda_{q+1}}+\frac{\delta_q^{\frac12}\delta_{q+1}^{\frac12}\lambda_q^{\frac32}}{\lambda_{q+1}^{\frac32}} + \frac{\delta_{q+1}^{\frac32} \lambda_q^{\frac12}}{\delta_q^{\frac12}\lambda_{q+1}^{\frac12}}\bigg) \lambda_{q+1}^{4\alpha} \lambda_{q+1}^N.
    \end{equation*}
Since $\beta < 1/2$, the dominant term is the third one. Therefore, 
\begin{equation*}
    \|R_{q+1}\|_N \lesssim \frac{\delta_{q+1}^{3/2} \lambda_q^{1/2}}{\delta_q^{1/2}\lambda_{q+1}^{1/2}} \lambda_{q+1}^{N+4\alpha},
\end{equation*}
and, by choosing $a_0$ sufficiently large, the implicit constants can be bounded by $\lambda_{q+1}^{\alpha}$, which implies 
\begin{equation*}
    \|R_{q+1}\|_N \leq \frac{\delta_{q+1}^{3/2} \lambda_q^{1/2}}{\delta_q^{1/2}\lambda_{q+1}^{1/2}} \lambda_{q+1}^{N+5\alpha}, \,\,\, \forall N \in \{0,1,..., L_R\}.
\end{equation*}
In view of 
\begin{equation*}
    b < \frac{1+2\beta}{4\beta},
\end{equation*}
the coefficient $\alpha > 0$ can be chosen sufficiently small in terms of $\beta$ and $b$ such that 
\begin{equation*}
    \frac{\delta_{q+1}^{3/2} \lambda_q^{1/2}}{\delta_q^{1/2}\lambda_{q+1}^{1/2}} \lambda_{q+1}^{7\alpha} \leq \delta_{q+2},
\end{equation*}
and, thus, \eqref{Spatial_estim_qplus1} is achieved.

    We are left to prove the material derivative estimate (\ref{Material_estim_qplus1}) corresponding to $u_{q+1}$. 
    Recall 
    \[u_{q+1}=u_q+w_{q+1}^{(t)}+w_{q+1}^{(p)}=\tilde u_{q,\Gamma}+(w_{q+1}^{(t)}-\tilde w_{q+1}^{(t)})+(u_q-\bar u_q) +w_{q+1}^{(p)}. \]
    Consequently,
    \begin{align*}
        \|D_{t,q+1} R_{q+1}\|_N \lesssim \|\tilde D_{t,\Ga} R_{q+1}\|_N+ \|(u_q - \bar u_q) \cdot\na R_{q+1}\|_N + \|(w_{q+1}^{(t)}-\tilde w_{q+1}^{(t)})\cdot\na R_{q+1}\|_N + \|w_{q+1}^{(p)} \cdot \nabla R_{q+1}\|_{N}.
    \end{align*}
    Bringing together the estimates obtained in the previous sections we deduce, for all $N \in \{0,1,..., L_t\}$,
    \begin{eqnarray*}
        \|\tilde D_{t, \Gamma} R_{q+1}\|_N &\lesssim& \bigg(\ell_{t,q}^{-1}\frac{\delta_{q+1}\lambda_q}{\lambda_{q+1}}\lambda_{q+1}^{4\alpha}+\mu_{q+1}\frac{\delta_{q}^{1/2}\delta_{q+1}^{1/2}\lambda_q^{3/2}}{\lambda_{q+1}^{3/2}}+\mu_{q+1}\delta_{q+1}  \\ 
        && + \tau_q^{-1}\frac{\delta_{q+1}\lambda_q}{\lambda_{q+1}}+\mu_{q+1}\frac{\delta_{q+1}\lambda_q}{\lambda_{q+1}}+\delta_{q+1}\delta_q^{1/2}\lambda_q^{3/2}\\
        &&+ \delta_{q+1}^{3/2}\lambda_q \lambda_{q+1}^{1/2}\lambda_{q+1}^{6\alpha}\bigg) \lambda_{q+1}^N\\
        &\lesssim& \delta_{q+1}^{3/2}\lambda_q \lambda_{q+1}^{1/2}\lambda_{q+1}^{N+6\alpha}.
    \end{eqnarray*}
    On the other hand, it follows from proposition \ref{prop-molli} and the inductive assumption (\ref{induct-u}) that
    \begin{equation*}
        \|\bar u_q - u_q\|_N\lesssim \ell_q^2\|u_q\|_{N+2} \lesssim \ell_q^2\delta_q^{1/2}\lambda_q^{N+2+1/2} \lesssim \delta_q^{1/2}\lambda_q^{1/2} \frac{\lambda_q}{\lambda_{q+1}}\lambda_{q+1}^N, \,\,\, \forall N \in \{0,1,...,L_\theta-2\},
    \end{equation*}
    and hence, for $N \leq L_t$,
    \begin{eqnarray*}
        \|(\bar u_q - u_q) \cdot \nabla R_{q+1}\|_N &\lesssim& \|\bar u_q - u_q\|_N \|R_{q+1}\|_1 + \|\bar u_q - u_q\|_0\|R_{q+1}\|_{N+1} \\
        &\lesssim & \delta_{q+1}^{3/2}\lambda_q^2 \lambda_{q+1}^{-1/2}\lambda_{q+1}^{N+5\alpha}
    \end{eqnarray*}
    where, in the last step, we used previously obtained estimate for $R_{q+1}$.
    Similarly, applying proposition \ref{prop-molli} and lemma \ref{w_t_estim}, we obtain
    \begin{eqnarray*}
        \|(w_{q+1}^{(t)}-\tilde w_{q+1}^{(t)}) \cdot \nabla R_{q+1}\|_N &\lesssim& \|w_{q+1}^{(t)}-\tilde w_{q+1}^{(t)}\|_N \|R_{q+1}\|_1 + \|w_{q+1}^{(t)}-\tilde w_{q+1}^{(t)}\|_0\|R_{q+1}\|_{N+1} \\
        &\lesssim & \delta_{q+1}^2 \delta_q^{-1/2}\lambda_q^{5/2}\lambda_{q+1}^{-1} \lambda_{q+1}^{N+5\alpha}.
    \end{eqnarray*}
    In the end, it follows from lemma \ref{le-est-Nash}, 
    \begin{eqnarray*}
        \|w_{q+1}^{(p)} \cdot \nabla R_{q+1}\|_N &\lesssim& \|w_{q+1}^{(p)}\|_N \|R_{q+1}\|_1 + \|w_{q+1}^{(p)}\|_0 \|R_{q+1}\|_{N+1} \\ 
        &\lesssim& \delta_{q+1}^2 \delta_q^{-1/2}\lambda_q^{1/2}\lambda_{q+1} \lambda_{q+1}^{N+5\alpha}.
    \end{eqnarray*}
    We conclude that, by absorbing the implicit constant into an extra $\lambda_{q+1}^{\alpha}$ factor, it holds that  
    \begin{equation*}
        \|D_{t, q+1} R_{q+1}\|_N \leq  \delta_{q+1}^2 \delta_q^{-1/2} \lambda_q^{1/2}\lambda_{q+1}\lambda_{q+1}^{N+7 \alpha}.
    \end{equation*}
    The estimate (\ref{Material_estim_qplus1}), then, follows for sufficiently small $\alpha > 0$
\end{proof}

\appendix
\section{Transport estimates} 
We recall standard estimates for solutions to the transport equation 
\begin{equation} \label{transport}
    \begin{cases}
    \partial_t f + u \cdot \nabla f = g, \\ 
    f \big | _{t = t_0} = f_0,
    \end{cases}
\end{equation}

\begin{prop}\cite[Proposition B.1]{BDLSV} \label{transport_estim}
Assume $|t-t_0| \|u\|_1 \leq 1$. Any solution $f$ of \eqref{transport} satisfies 
\begin{equation*}
    \|f(\cdot, t)\|_0 \leq \|f_0\|_0 + \int_{t_0}^t \|g(\cdot, \tau)\|_0 d \tau, 
\end{equation*}
\begin{equation*}
    \|f(\cdot, t)\|_\alpha \leq 2 \big( \|f_0\|_\alpha + \int_{t_0}^t \|g(\cdot, \tau)\|_\alpha d\tau\big),
\end{equation*}
for $\alpha \in [0,1]$.
More generally, for any $N \geq 1$ and $\alpha \in [0,1)$,
\begin{equation*}
[f(\cdot, t)]_{N+\alpha} \lesssim [f_0]_{N+\alpha} + |t-t_0| [u]_{N+\alpha} [f_0]_1 + \int_{t_0}^t \big( [g(\cdot, \tau)]_{N+\alpha} + (t-\tau) [u]_{N+\alpha} [g(\cdot, \tau)]_1\big) d\tau,
\end{equation*}
where the implicit constant depends on $N$ and $\alpha$.
Consequently, the backwards flow $\Phi$ of $u$ starting at time $t_0$ satisfies 
\begin{equation*}
    \|D\Phi(\cdot, t) - \I\|_0 \lesssim |t-t_0|[u]_1, 
\end{equation*}
\begin{equation*}
    [\Phi(\cdot, t)]_N \lesssim  |t - t_0| [u]_N, \,\,\,  \forall N \geq 2. 
\end{equation*}
\end{prop}

\section{Harmonic analysis}\label{sec.bha}
\subsection{A Littlewood-Paley partition of unity}

Let $d \geq 2$ and $\psi:\mathbb R^d \rightarrow \mathbb R$ be a smooth, spherically-symmetric function such that $\supp \psi \subset B_{3/2}(0)$ and $\psi(x) = 1$ for all $x \in B_1(0)$. For $j \geq 0$, denote 
\begin{equation*}
    \chi_j(\xi) = \psi\bigg( \frac{\xi}{2^j} \bigg) - \psi\bigg( \frac{\xi}{2^{j-1}} \bigg).
\end{equation*}
Then, for $f:\mathbb T^d \rightarrow \mathbb R$, we define the Littlewood-Paley projections 
\begin{equation*}
    \Delta_j f = \sum_{k \in \mathbb Z^d} \chi_j(k) \hat{f}(k) e^{i k \cdot x}, \,\,\, j \geq 0,
\end{equation*}
and 
\begin{equation*}
    \Delta_{-1} f = \hat f(0).
\end{equation*}
It will be notationally convenient to extend the definition to all $j \in \mathbb Z$ by 
\begin{equation*}
    \Delta_{j} f = 0 \,\,\, \forall j < -1.
\end{equation*}
We will also use the low frequency projections: 
\begin{equation*}
    S_j f = \sum_{i \leq j} \Delta_i f = \sum_{k \in \mathbb Z^2} \psi\bigg(\frac{k}{2^j}\bigg) \hat{f}(k) e^{ik\cdot x}.
\end{equation*}
It is not difficult to verify that for all $f \in C^\infty(\mathbb T^d)$, the following hold: 
\begin{itemize}
    \item $f(x) = \sum_{j=-1}^\infty \Delta_j f(x)=\hat{f}(0) + \sum_{j=0}^\infty \Delta_j f(x)$;

    \item $\supp \widehat{\Delta_j f} \subset \mathbb Z^d \cap \big(B_{2^{j+1}}(0) \setminus B_{2^{j-1}}(0)\big)$, for all $j \geq 0$;

    \item $\Delta_j \Delta_k f = 0$, whenever $|j-k| > 1$. 
\end{itemize}

\subsection{Basic Littlewood-Paley theory}

We recall here a few standard lemmas (see, for example, \cite{MS} for the proofs of these results on the Euclidean space). Since these are most commonly stated on $\mathbb R^d$ instead of $\mathbb T^d$, we also provide proofs for the convenience of the reader. 

\begin{defn}
We say $T$ is a Fourier multiplier operator of order $s \in \mathbb R$ if 
\begin{equation*}
    Tf(x) = \sum_{k \in \mathbb Z^d \setminus \{0\}} m(k) \hat f(k) e^{ik\cdot x},
\end{equation*}
for a multiplier $m \in C^\infty (\mathbb R^d \setminus \{0\})$ which is $s$-homogeneous.
\end{defn}

\begin{lem} \label{Bernstein}
    Let $f \in C^\infty(\mathbb T^d)$, and $T$ be an operator of order $s \in \mathbb R$. Then, it holds that 
    \begin{equation*}
        \|T \Delta_j f\|_0 \lesssim 2^{sj} \|\Delta_j f\|_0, 
    \end{equation*}
    with an implicit constant depending only on the operator $T$. 
\end{lem}

\begin{proof}
    Let $\bar\chi_0:\mathbb R^d\setminus \{0\} \rightarrow \mathbb R$ be a smooth, compactly supported function satisfying $\bar \chi_0(\xi) = 1$, for all $\xi \in \supp \chi_0$. With $\bar \chi_j(\xi) = \bar \chi_0(2^{-j}\xi)$, for $j \geq 0$, we have 
    \begin{equation*}
        T\Delta_j f(x) = \sum_{k \in \mathbb Z^d} m(k) \bar \chi_j (k) \widehat{\Delta_j f}(k) e^{i k\cdot x}.
    \end{equation*}
    It follows, then, that 
    \begin{equation*} 
        T \Delta_j f(x) = \int_{\mathbb R^d} \Delta_j f(x-y) \widecheck {m \bar \chi_j}(y),
    \end{equation*}
    where we identify $\Delta_j f$ with its periodic extension. Then,
    \begin{equation*}
        \|T \Delta_j f\|_0 \lesssim \| \widecheck {m \bar \chi_j}\|_{L^1(\mathbb R^d)} \|\Delta_j f\|_0.
    \end{equation*}
    The result follows once we note that 
    \begin{equation*}
        \widecheck {m \bar \chi_j}(x) = 2^{j(s+d)}\widecheck {m \bar \chi_0} (2^j x),
    \end{equation*}
    which implies 
    \begin{equation*}
        \|\widecheck {m \bar \chi_j}\|_{L^1} = 2^{sj} \|\widecheck {m \bar \chi_0}\|_{L^1}.
    \end{equation*}
\end{proof}

\begin{rem}
    Similar scaling arguments can be used to show that 
    \begin{equation*}
        \|\Delta_j f\|_N \lesssim \|f\|_N, \,\,\, \forall j \geq -1, N\geq 0;
    \end{equation*}
    \begin{equation*}
        \|S_j f\|_N \lesssim \|f\|_N, \,\,\, \forall j \geq -1, N\geq 0.
    \end{equation*}
\end{rem}

We gather in the following corollary immediate applications of lemma \ref{Bernstein}.

\begin{cor} \label{Bern_cor}
    The following 
    \begin{equation*}
        \|\Delta_j f\|_N \lesssim 2^{Nj}\|\Delta_j f\|_0,
    \end{equation*}
    \begin{equation*}
        \|\Lambda \Delta_j f\|_0 \lesssim 2^j \|\Delta_j f\|_0 \lesssim \|\Lambda \Delta_j f\|_0,
    \end{equation*}
    hold for all $f \in C^\infty(\mathbb T^d)$.
\end{cor}

\begin{proof}
    The first result, as well as the first inequality of the second result follow immediately from lemma \ref{Bernstein}. For the final inequality we note that
    \begin{equation*}
        \|\Delta_j f\|_0 = \|\Lambda^{-1} \Delta_j \Lambda f\|_0 \lesssim 2^{-j} \|\Lambda \Delta_j f\|_0.
    \end{equation*}
\end{proof}

\begin{lem} \label{Besov_char} Let $0<\alpha<1$. There exists a constant $C>0$, depending only on $\alpha$, such that 
\begin{equation} \label{Besov_equiv}
    \frac{1}{C} \|f\|_\alpha \leq \sup_{j \geq -1} 2^{j\alpha} \|\Delta_j f\|_0 \leq C \|f\|_\alpha,
\end{equation}
for all $f \in C^\infty(\mathbb T^d)$.
\end{lem}

\begin{proof}
    Arguing similarly to the proof of the previous lemma, we note that 
    \begin{equation*}
        \Delta_j f(x) = \int_{\mathbb R^d} f(x-y) \check \chi_j(y) dy,
    \end{equation*}
    where $f$ is identified with its periodic extension. Since 
    \begin{equation*}
        \int_{\mathbb R^d} \check \chi_j = \chi_j(0) = 0,
    \end{equation*}
    it follows that, for all $j \geq 0$,
    \begin{eqnarray*}
        |\Delta_j f(x)| &=& \bigg| \int_{\mathbb R^d} \big( f(x-y) - f(x) \big) \check \chi_j(y) dy \bigg| \\ 
        & \lesssim & \|f\|_\alpha \int_{\mathbb R^d} |y|^\alpha |\check \chi_j(y)| dy \\ 
        & = & 2^{-j \alpha} \|f\|_\alpha \int_{\mathbb R^d} |y|^\alpha |\check \chi_0 (y)| dy.
    \end{eqnarray*}
    This implies $\sup_{j \geq 0} 2^{j\alpha} \|\Delta_j f\|_0 \lesssim \|f\|_\alpha$, and the second inequality in \eqref{Besov_equiv} follows once we note that $|\hat f(0)| \lesssim \|f\|_0 \leq \|f\|_\alpha$.

    For the first estimate, fix $x, y \in \mathbb T^d$ and let $k \in \mathbb Z$ such that 
    \begin{equation*}
        2^{-k-1} \leq |x-y| < 2^{-k}.
    \end{equation*}
    Then, 
    \begin{eqnarray*}
        |f(x) - f(y)| & \lesssim & \sum_{j \geq 0} |\Delta_j f(x) - \Delta_j f(y)| \\ 
        & \lesssim & \sum_{j \leq k} |\Delta_j f(x) - \Delta_j f(y)| + \sum_{j \geq k} \|\Delta_j f\|_0 \\ 
        & \lesssim & \sum_{j \leq k} \|\Delta_j f\|_1 |x-y| + \sum_{j \geq k} \|\Delta_j f\|_0.
    \end{eqnarray*}
    By lemma \ref{Bernstein}, 
    \begin{equation*}
        \sum_{j \leq k} \|\Delta_j f\|_1 |x-y| \lesssim \sum_{j \leq k}2^j \|\Delta_j f\|_0 2^{-k(1-\alpha)}|x-y|^\alpha \lesssim \big(\sup_{j \geq 0} 2^{j \alpha} \|\Delta_j f\|_0 \big) |x-y|^\alpha.
    \end{equation*}
    On the other hand, 
    \begin{equation*}
        \sum_{j \geq k} \|\Delta_j f\|_0 \lesssim \big(\sup_{j \geq 0} 2^{j \alpha} \|\Delta_j f\|_0 \big) \sum_{j \geq k} 2^{-j\alpha} \lesssim 2^{-(k+1)\alpha} \sup_{j \geq 0} 2^{j \alpha} \|\Delta_j f\|_0 \lesssim \big(\sup_{j \geq 0} 2^{j \alpha} \|\Delta_j f\|_0 \big) |x-y|^\alpha.
    \end{equation*}
    Bringing everything together, we find that 
    \begin{equation*}
        [f]_\alpha \lesssim \sup_{j \geq 0} 2^{j \alpha} \|\Delta_j f\|_0,
    \end{equation*}
    which implies the wanted inequality since $\|f - \hat f (0)\|_\alpha \lesssim [f]_\alpha$. 
\end{proof}

As a corollary, the standard estimate for $0$-homogeneous operators follows.

\begin{cor} \label{0-hom_estim}
Let $0 < \alpha < 1$ and $T$ be a Fourier multiplier operator of order $0$. Then, there exists a constant $C > 0$ depending on $T$ and $\alpha$  such that 
\begin{equation*}
    \|Tf\|_\alpha \leq C \|f\|_\alpha,
\end{equation*}
for all $f \in C^\infty (\mathbb T^d)$. 
\end{cor}

\begin{proof}
    We have
    \begin{equation*}
        \|Tf\|_\alpha \lesssim \sup_{j \geq 0} 2^{j \alpha} \|T \Delta_j f\|_0 \lesssim \sup_{j \geq 0} 2^{j \alpha} \| \Delta_j f\|_0 \lesssim \|f\|_\alpha,
    \end{equation*}
    where the first and last inequalities follow from lemma \ref{Besov_char} and the second one from lemma \ref{Bernstein}.
\end{proof}

\begin{rem}
    Corollary \ref{0-hom_estim} can be used to deduce estimates also for multiplier operators of different orders. Particularly relevant will be the estimate 
    \begin{equation*}
        \|\Lambda f\|_\alpha \lesssim \|f\|_{1+\alpha},
    \end{equation*}
    which follows by writing $\Lambda = \mathcal{R}\cdot \nabla$, where $\mathcal{R}$ is the ($0$-order) Riesz transform. 
\end{rem}

\section{Mollification, composition, and singular integral estimates}

\begin{prop}\cite[Lemma 2.1]{CDLS12}\label{prop-molli}
Let $\phi$ be a symmetric mollifier with $\int \phi=1$. For any smooth function $f$, the estimate 
\[\|f-f*\phi_\ell\|_N\lesssim \ell^2\|f\|_{N+2}, \ \ \forall \ N\geq 0\]
holds with implicit constant depending only on $N$.
\end{prop}

\begin{prop}\cite[Proposition A.1]{BDLSV} \label{comp_estim}
    Let $\Psi:\Omega \rightarrow \mathbb R$ and $u: \mathbb R^n \rightarrow \Omega$ be two smooth functions, with $\Omega \subset \mathbb R^N$. Then, for any $m \in \mathbb N \setminus \{0\}$, there exists a constant $C = C(m, N,n)$ such that 
    \begin{equation*}
        [\Psi \circ u]_m \leq C\big([\Psi]_1 \|Du\|_{m-1} + \|D\Psi\|_{m-1} [u]_1^m \big).
    \end{equation*}
\end{prop}



\begin{prop}\cite[Proposition D.1]{BDLSV} \label{prop-commu}
Let $\alpha\in(0,1)$, $N\geq 0$ and $u\in C^{N+\alpha}$ be a vector field. Let $T_K$ be a Calder\'on-Zygmund operator with kernel $K$. Then the estimate
\begin{equation}\notag
\|[u\cdot\nabla, T_K]f\|_{N+\alpha}\lesssim \|u\|_{1+\alpha}\|f\|_{N+\alpha}+\|u\|_{N+1+\alpha}\|f\|_{\alpha}
\end{equation}
holds for any $f\in C^{N+\alpha}$, with implicit constant depending on $\alpha, N, K$.
\end{prop}

\begin{prop}\cite[Lemma A.6]{BSV}\label{prop-commu-loc}
    Let $s\in \R$, $\la\geq 1$, and let $T_K$ be an order $s$ convolution operator localized at length scale $\la^{-1}$. That is, $T_K$ acts on smooth functions $f$ as 
    $$ T_K f(x) = \int_{\R^2} K(y)f(x-y)\,dy  $$
    for some kernel $K:\R^2\to\R$ that obeys
    $$ \||x|^a \na^b K(x)\|_{L^1(\R^2)} \lesssim \la^{b-a+s} $$
    for all $0\leq a , |b|\leq 1$ and some implicit constants $C=C(a,b)$. Then, for any smooth function $f:\T^2\to\R^2$ and smooth incompressible vector field $u:\T^2\to\R^2$, we have
    $$ \|[u\cdot\na, T_K]f\|_{0} \leq \la^s \|\na u\|_{0} \|f\|_0\,. $$
\end{prop}

\section{A bilinear microlocal lemma}
We establish now a bilinear microlocal lemma for the SQG nonlinearity which is very similar to that introduced in \cite{IM}.

Let us recall here the definition of the operator $P_{\approx \lambda}$. Let $F$ be the set of lemma \ref{le-geo} and $A \subset \mathbb R^2$ be an annulus centered at the origin such that for all $\xi \in F$, the vectors $2\xi$ and $\xi/2$ are contained in $A$. Let, then, $\chi:\mathbb R^2 \rightarrow \mathbb R$ be a smooth function with support in a slightly larger annulus $A'$, which moreover satisfies $\chi(x) = 1$, for all $x \in A$. Given $\lambda \in \mathbb N$, we define the rescaled frequency cut-offs:
\begin{equation*}
    \chi_\lambda (\xi) = \chi(\lambda^{-1} \xi)
\end{equation*}
and 
\begin{equation}\label{def.lpproj}
    {P_{\approx \la} f}(\xi) := \sum_{\xi \in \mathbb Z^2} \chi_\lambda(\xi) \hat f(\xi) e^{i \xi \cdot x}.
\end{equation}

\begin{lem}[Bilinear Microlocal Lemma] \label{le-bilinear-odd}
Let $\xi \in \mathbb Z^2$, $a: \mathbb T^2 \rightarrow \mathbb R$ be a smooth function, and $\Phi:\mathbb T^2 \rightarrow \mathbb T^2$ a smooth diffeomorphism satisfying, for all $x \in \mathbb T^2$, $\nabla \Phi^T(x) \xi \in A$, where $A$ is the annulus in the definition of $P_{\approx \lambda}$. Define $\theta_{\xi}: \T^d \to \R$ as
\[\theta_\xi(x)= P_{\approx \la}\left[a(x) \cos{(\lambda \Phi(x)\cdot \xi)}\right]\,.\]
Then, there exists a smooth symmetric $2$-tensor field $B_\la$ such that
\begin{equation}\notag
\nabla^\perp \Lambda^{-1} \theta_\xi \cdot \nabla \theta_\xi =\na^\perp \cdot \div B_{\lambda} 
\end{equation}
Moreover, $B_\lambda$ has the expansion 
\begin{equation}\notag
B_{\lambda}(x)=\frac{1}{4} a^2(x)\frac{\nabla\Phi^T\xi\otimes \nabla\Phi^T\xi}{\lambda|\nabla\Phi^T\xi|^3}+\delta B_\lambda(x)
\end{equation}
where $\delta B_\lambda(x)$ is an explicit error term.
\end{lem}
\begin{proof}
Let us denote $\Theta_\xi' = a(x) e^{i\lambda \Phi(x)\cdot \xi}$ and $\Theta_\xi = P_{\approx \la} \Theta_\xi'$. Then,
$$ \theta_\xi = \frac{\Theta_\xi + \Theta_{- \xi}}2\,. $$

We begin by noting that 
\begin{equation*}
    \nabla^\perp \Lambda^{-1} \theta_\xi \cdot \nabla \theta_\xi= \frac{1}{4} \nabla^\perp \cdot  \sum_{\zeta, \eta \in \{\xi, - \xi\}} \Lambda^{-1} \Theta_\zeta \nabla \Theta_\eta = \frac{1}{8} \nabla^\perp \cdot \sum_{\zeta, \eta \in \{\xi, -\xi \}} \Lambda^{-1} \Theta_\zeta \nabla \Theta_\eta + \Lambda^{-1} \Theta_\eta \nabla \Theta_\zeta.
\end{equation*}
Since gradients are in the kernel of $\nabla^\perp \cdot$, we further rewrite
\begin{equation*}
     \Lambda^{-1} \Theta_\zeta \nabla \Theta_\eta + \Lambda^{-1} \Theta_\eta \nabla \Theta_\zeta = \nabla (\Lambda^{-1} \Theta_\zeta \Theta_\eta) + \Lambda^{-1}\Theta_\eta \nabla \Theta_\zeta - \Theta_\eta \nabla \Lambda^{-1}\Theta_\zeta,
\end{equation*}
so that 
\begin{equation*}
    \nabla^\perp \Lambda^{-1} \theta_\xi \cdot \nabla \theta_\xi = \frac{1}{8} \nabla^\perp \cdot \sum_{\zeta, \eta \in \{\xi, - \xi\}} \underbrace{\Lambda^{-1}\Theta_\eta \nabla \Theta_\zeta - \Theta_\eta \nabla \Lambda^{-1}\Theta_\zeta}_{=: Q[\Theta_\eta, \Theta_\zeta]}.
\end{equation*}
On the other hand, the same procedure can be done with the roles of $\Theta_\zeta$ and $\Theta_\eta$ reversed. This leads to the expression 
\begin{equation*}
    \nabla^\perp \Lambda^{-1} \theta_\xi \cdot \nabla \theta_\xi = \frac{1}{16} \nabla^\perp \cdot \sum_{\zeta, \eta \in \{\xi, - \xi\}} Q[\Theta_\eta, \Theta_\zeta] + Q[\Theta_\zeta, \Theta_\eta].
\end{equation*}
We analyze now the bilinear Fourier multiplier operator $Q$. 

We have 
\begin{eqnarray*}
    \widehat {Q[\Theta_\eta, \Theta_\zeta]}(k) &=& \sum_{j \in \mathbb Z^2} i j \bigg(\frac{1}{|k-j|} - \frac{1}{|j|} \bigg) \widehat \Theta_\eta(k-j) \widehat \Theta_\zeta (j) \\ 
    &=& \sum_{j \in \mathbb Z^2} i j \frac{k \cdot (2j - k)}{|k-j||j|(|k-j|+|j|)} \chi_\lambda(k-j) \chi_\lambda(j) \widehat{\Theta_\eta'}(k-j) \widehat{\Theta_\zeta'}(j).
\end{eqnarray*}
It follows, then, that 
\begin{eqnarray*}
    Q[\Theta_\eta, \Theta_\zeta](x) &=& \sum_{j,k \in \mathbb Z^2} i j \frac{(j+k)\cdot(j-k)}{|k||j|(|k|+|j|)} \chi_\lambda(k) \chi_\lambda(j) \widehat{\Theta_\eta'}(k) \widehat{\Theta_\zeta'}(j) e^{i(j+k)\cdot x} \\ 
    & = & \sum_{j,k \in \mathbb Z^2} i \frac{j \otimes (j-k)}{|k||j|(|k|+|j|)}(j+k) \chi_\lambda(k) \chi_\lambda(j) \widehat{\Theta_\eta'}(k) \widehat{\Theta_\zeta'}(j) e^{i(j+k)\cdot x}.
\end{eqnarray*}
Consequently, 
\begin{eqnarray*}
    Q[\Theta_\eta, \Theta_\zeta] + Q[\Theta_\zeta, \Theta_\eta] &=& \sum_{j,k \in \mathbb Z^2} i \frac{j \otimes (j-k) + k \otimes (k-j)}{|k||j|(|k|+|j|)}(j+k) \chi_\lambda(k) \chi_\lambda(j) \widehat{\Theta_\eta'}(k) \widehat{\Theta_\zeta'}(j) e^{i(j+k)\cdot x} \\ 
    &=& \sum_{j,k \in \mathbb Z^2}i \frac{(j-k) \otimes (j-k)}{|k||j|(|k|+|j|)}(j+k) \chi_\lambda(k) \chi_\lambda(j) \widehat{\Theta_\eta'}(k) \widehat{\Theta_\zeta'}(j) e^{i(j+k)\cdot x}.
\end{eqnarray*}
We define, then, 
\begin{equation*}
    B_\lambda(x) = \frac{1}{16}\sum_{\eta, \zeta \in \{\xi, - \xi\}}\sum_{j,k \in \mathbb Z^2} \frac{(j-k) \otimes (j-k)}{|k||j|(|k|+|j|)} \chi_\lambda(k) \chi_\lambda(j) \widehat{\Theta_\eta'}(k) \widehat{\Theta_\zeta'}(j) e^{i(j+k)\cdot x},
\end{equation*}
and note that it is a symmetric $2$-tensor field, as wanted. 

It remains to justify the claimed expansion. For this purpose, let $K_\lambda$ be the kernel defined on $\mathbb R^2 \times \mathbb R^2$ and taking values in the space of symmetric 2-tensors: 
\begin{equation*}
    K_\lambda(h_1, h_2) = \frac{1}{16 (2\pi)^4} \int_{\mathbb R^2 \times \mathbb R^2} \frac{(\nu_1 - \nu_2) \otimes (\nu_1 - \nu_2)}{|\nu_1||\nu_2|(|\nu_1| + |\nu_2|)} \chi_\lambda(\nu_1) \chi_\lambda(\nu_2) e^{i(\nu_1 \cdot h_1 + \nu_2 \cdot h_2)} d\nu_1 d\nu_2.
\end{equation*}
We note that since $\chi_\lambda$ is smooth and compactly supported away from the origin, $K_\lambda$ is Schwartz. We have, then, the following expression for $B_\lambda$: 
\begin{equation*}
    B_\lambda(x) = \sum_{\eta, \zeta \in \{\xi, -\xi\}} \underbrace{\int_{\mathbb R^2 \times \mathbb R^2} K_\lambda (h_1, h_2) \Theta_\eta'(x-h_1) \Theta_\zeta'(x-h_2) dh_1 dh_2}_{=:B_\lambda^{\eta, \zeta}},
\end{equation*}
where we identify $\Theta_\eta'$ and $\Theta_\zeta'$ with their periodic extensions. Recalling the definitions of $\Theta_\eta'$ and $\Theta_\zeta'$, we have 
\begin{equation*}
    B_\lambda^{\eta,\zeta}(x) = \int_{\mathbb R^2 \times \mathbb R^2} K_\lambda(h_1, h_2) a(x-h_1)a(x-h_2) e^{i\lambda \Phi(x-h_1)\cdot \eta} e^{i \lambda \Phi(x-h_2)\cdot \zeta}\, dh_1dh_2.
\end{equation*}
The desired expansion will follow as a consequence of Taylor's formula: 
\begin{equation*}
    \Phi(x-h) = \Phi(x) - \nabla \Phi(x)h + R_{\Phi}(x, h),
\end{equation*}
where the remainder is given by
\begin{equation*}
    R_\Phi(x, h) = \sum_{j,k = 1}^2 h^j h^k \int_{0}^1 (1-s) \partial_j \partial_k \Phi(x-sh) ds.
\end{equation*}
Then, 
\begin{equation*}
    a(x-h)e^{i\lambda \Phi(x-h)\cdot \eta} = \big(a(x) +\underbrace{(a(x-h)e^{i\lambda R_\Phi(x,h)\cdot \eta} - a(x))}_{=:Y^\eta_\lambda(x,h)} \big) e^{i\lambda \Phi(x)\cdot \eta}e^{-i \lambda \nabla \Phi(x) h \cdot \eta}.
\end{equation*}
For convenience of notation, we also introduce the function
\begin{equation*}
    Y_\lambda^{\eta,\zeta}(x,h_1,h_2) = a(x)(Y_\lambda^\eta(x,h_1) + Y_\lambda^\zeta(x,h_2)) +  Y_\lambda^\eta(x,h_1)Y_\lambda^\zeta(x,h_2).
\end{equation*}
With this expansion, we have
\begin{eqnarray*}
    B_\lambda^{\eta, \zeta}(x) &=& e^{i\lambda \Phi(x)\cdot(\eta+\zeta)} \int_{\mathbb R^2 \times \mathbb R^2} a^2(x) K_\lambda(h_1, h_2) e^{-i\lambda (\nabla \Phi^T \eta)\cdot h_1} e^{-i\lambda (\nabla \Phi^T \zeta)\cdot h_2} dh_1 dh_2 \\ 
    && + \underbrace{e^{i\lambda \Phi(x)\cdot(\eta+\zeta)} \int_{\mathbb R^2 \times \mathbb R^2} Y_\lambda^{\eta,\zeta}(x,h_1,h_2)K_\lambda(h_1, h_2) e^{-i\lambda (\nabla \Phi^T \eta)\cdot h_1} e^{-i\lambda (\nabla \Phi^T \zeta)\cdot h_2} dh_1 dh_2}_{=:\delta B_\lambda^{\eta, \zeta}(x)}. 
\end{eqnarray*}
We have, then, 
\begin{equation*}
    B_\lambda^{\eta, \zeta}(x) = a^2(x) e^{i\lambda \Phi(x)\cdot(\eta+\zeta)} \widehat{K_\lambda}(\lambda \nabla \Phi^T \eta, \lambda \nabla \Phi^T \zeta) + \delta B_\lambda^{\eta, \zeta}(x).
\end{equation*}
Note that $\widehat {K_\lambda} (\nu, \nu) = 0$ for all $\nu \in \mathbb R^2$, while 
\begin{equation*}
    \widehat{K_\lambda}(\nu, -\nu) = \frac{1}{8} \frac{\nu \otimes \nu}{|\nu|^3},
\end{equation*}
for all $\nu \in \chi_\lambda^{-1}(\{1\})$. In particular, the latter holds for $\nu = \pm \lambda \nabla \Phi^T \xi$.
We conclude, then, that 
\begin{equation*}
    B_\lambda(x) = \frac{1}{4} a^2(x)\frac{\nabla\Phi^T\xi\otimes \nabla\Phi^T\xi}{\lambda|\nabla\Phi^T\xi|^3}+\delta B_\lambda(x),
\end{equation*}
with 
\begin{equation*}
    \delta B_\lambda(x) = \sum_{\eta, \zeta \in \{\xi, -\xi\}} \delta B_\lambda^{\eta, \zeta}(x).
\end{equation*}

\end{proof}

In order to estimate the error term $\delta B_\xi$, we will need the following estimate, obtained by scaling, for the physical space kernel $K_\la$. For $\bar h = (h_1,h_2) \in \R^2\times \R^2$, we have for every $m\in \N$
\begin{equation}\label{est.K}
    \la^m \||\bar h|^m K_\la (h_1,h_2)\|_{L^1(\R^2\times \R^2)} \lesssim_m \la^{-1}\,.
\end{equation}


\medskip

\section{Tools of convex integration}
\label{sec-geo}

\subsection{A geometric lemma}

\begin{lem}[\cite{S12}]\label{le-geo}
Denote by $B_{1/2}(Id)$ the metric ball centered at the identity in the space $\mathcal S^{2\times 2}$ of symmetric $2\times 2$ matrices. There exist a finite set $F\subset \mathbb Z^2$ and smooth functions $\gamma_{\xi}: B_{1/2}(Id)\to \mathbb R$ for any $\xi\in F$ such that 
\[R=\sum_{\xi\in F}\gamma_{\xi}^2(R)\xi\otimes \xi\]
for $\forall R\in B_{1/2}(Id)$.
\end{lem}

\subsection{An inverse divergence operator}
We use the following inverse-divergence operator
\begin{equation} \label{invdiv}
    (\div^{-1} u)^{ij} =  \Delta^{-1} (\partial_i u^j + \partial_j  u^i -  \delta_{ij}\div u),
\end{equation}
which maps smooth, mean-zero vector fields $u$ to smooth, symmetric and trace-free $2$-tensors $\div^{-1}u$. 

\begin{prop}[\cite{ChDLS12}]
If $u$ is a smooth, mean-zero vector field, then the $2$-tensor field $\div^{-1} u$ defined by \eqref{invdiv} is symmetric and satisfies 
\begin{equation*}
    \div \div^{-1} u = u.  
\end{equation*}
\end{prop}

\end{document}